\documentclass[final]{siamart0216}
\usepackage{amsfonts}

\usepackage{bm}
\usepackage{amsfonts,amsmath,amssymb}
\usepackage{mathrsfs,mathtools,stmaryrd,wasysym}
\usepackage{enumerate}
\usepackage{enumitem}
\usepackage{esint}
\usepackage{graphicx,psfrag}
\usepackage{hyperref}
\DeclareMathAlphabet{\mathpzc}{OT1}{pzc}{m}{it}

\newcommand{\FF}[1]{{\color{black}#1}}
\newcommand{\DQ}[1]{{\color{black}#1}}
\newcommand{\EO}[1]{{\color{black}#1}}
\newcommand{\alert}[1]{{\color{black}#1}}

\usepackage{lineno}
\usepackage{pgf,tikz,pgfplots}
\usepackage{pstricks-add}

\DeclareFontEncoding{FMS}{}{}
\DeclareFontSubstitution{FMS}{futm}{m}{n}
\DeclareFontEncoding{FMX}{}{}
\DeclareFontSubstitution{FMX}{futm}{m}{n}
\DeclareSymbolFont{fouriersymbols}{FMS}{futm}{m}{n}
\DeclareSymbolFont{fourierlargesymbols}{FMX}{futm}{m}{n}
\DeclareMathDelimiter{\VERT}{\mathord}{fouriersymbols}{152}{fourierlargesymbols}{147}

\newcommand{\T}{\mathscr{T}}
\newcommand{\Sides}{\mathscr{S}}

\usetikzlibrary{arrows}

\newsiamremark{remark}{Remark}

\newcommand{\TheTitle}{A posteriori error estimates for a distributed optimal control problem of the stationary Navier--Stokes equations}
\newcommand{\ShortTitle}{A control problem for the Navier--Stokes equations}
\newcommand{\TheAuthors}{A. Allendes, F. Fuica, E. Ot\'arola, D. Quero}

\headers{\ShortTitle}{\TheAuthors}

\title{{\TheTitle}\thanks{AA is partially supported by CONICYT through FONDECYT project 1170579. FF is supported by UTFSM through Beca de Mantenci\'on. EO is partially supported by CONICYT through FONDECYT Project 11180193.}}

\author{Alejandro Allendes\thanks{Departamento de Matem\'atica, Universidad T\'ecnica Federico Santa Mar\'ia, Valpara\'iso, Chile.
(\email{alejandro.allendes@usm.cl}).}
\and
Francisco Fuica\thanks{Departamento de Matem\'atica, Universidad T\'ecnica Federico Santa Mar\'ia, Valpara\'iso, Chile.
(\email{francisco.fuica@sansano.usm.cl}).}
\and
Enrique Ot\'arola\thanks{Departamento de Matem\'atica, Universidad T\'ecnica Federico Santa Mar\'ia, Valpara\'iso, Chile.
(\email{enrique.otarola@usm.cl}, \url{http://eotarola.mat.utfsm.cl/}).}
\and
Daniel Quero\thanks{Departamento de Matem\'atica, Universidad T\'ecnica Federico Santa Mar\'ia, Valpara\'iso, Chile.
(\email{daniel.quero@alumnos.usm.cl}).}}

\ifpdf
\hypersetup{
  pdftitle={\TheTitle},
  pdfauthor={\TheAuthors}
}
\fi

\date{Draft version of \today.}

\begin{document}

\maketitle

\begin{abstract}
\EO{In two and three dimensional Lipschitz, but not necessarily convex, polytopal domains, we propose and analyze a posteriori error estimators for an optimal control problem involving the stationary Navier--Stokes equations; control constraints are also considered. We devise two strategies of discretization: a semi discrete scheme where the control variable is not discretized -- the so-called variational discrezation approach -- and a fully discrete scheme where the control is discretized with piecewise quadratic functions. For each solution solution technique, we design an a posteriori error estimator that can be decomposed as the sum of contributions related to the discretization of the state and adjoint equations and, additionally, the discretization of the control variable for when the fully discrete scheme is considered. We prove that the devised error estimators are reliable and also explore local efficiency estimates. Numerical experiments reveal a competitive performance of adaptive loops based on the devised a posteriori error estimators.}
\end{abstract}



\section{Introduction}\label{sec:intro}

In this work we shall be interested in the design and analysis of a posteriori error \EO{estimators} for a distributed optimal control problem involving the stationary Navier--Stokes equations; control constraints are also considered. To make matters precise, let $\Omega\subset\mathbb{R}^d$, with $d\in \{2, 3\}$, be an open and bounded polytopal domain with Lipschitz boundary $\partial\Omega$. Given a desired state $\mathbf{y}_\Omega \in \mathbf{L}^2(\Omega)$ and a regularization parameter $\alpha>0$, we define the quadratic cost functional
\begin{equation*}\label{def:cost_functional}
J(\mathbf{y},\mathbf{u}):=\frac{1}{2}\|\mathbf{y}-\mathbf{y}_\Omega\|_{\mathbf{L}^2(\Omega)}^2+\frac{\alpha}{2}\|\mathbf{u}\|_{\mathbf{L}^2(\Omega)}^2.
\end{equation*}
\EO{We shall be concerned with the following optimal control problem}: Find
\begin{equation}\label{eq:minimize_cost_func}
\min{J(\mathbf{y},\mathbf{u})}
\end{equation}
subject to the stationary Navier--Stokes equations
\begin{equation}\label{eq:state_equations}
-\nu\Delta \mathbf{y}+(\mathbf{y}\cdot\nabla)\mathbf{y}+\nabla p = \mathbf{u} \text{ in } \Omega, \quad \text{div }\mathbf{y}=0 \text{ in } \Omega,\quad \mathbf{y}=\boldsymbol 0  \text{ on } \partial\Omega,
\end{equation}
and the control constraints
\begin{equation}\label{def:box_constraints}
\mathbf{u}\in\mathbb{U}_{ad}, \qquad \mathbb{U}_{ad}:=\{ \mathbf{v}\in \mathbf{L}^2(\Omega): \mathbf{a} \leq \mathbf{v} \leq \mathbf{b} \text{ a.e. } \text{in } \Omega \},
\end{equation}
with $\mathbf{a}, \mathbf{b} \in  \mathbb{R}^d$ satisfying $\mathbf{a} < \mathbf{b}$. We immediately comment that, throughout this work, vector inequalities must be understood componentwise.  In \eqref{eq:state_equations}, $\nu> 0$ denotes the kinematic viscosity.

The numerical analysis of optimal control problems governed by the stationary Navier--Stokes equations has been previously considered in a number of works. For a slightly different cost functional $J$, which in contrast to \eqref{def:cost_functional} measures the difference $\mathbf{y} - \mathbf{y}_\Omega$ in the $\mathbf{L}^4(\Omega)$-norm, the authors of \cite{MR1079020} \EO{analyzed} inf-sup stable finite element approximations, \EO{satisfying \cite[(4.1)--(4.3)]{MR1079020}}, of suitable distributed and boundary optimal control problems; control constraints are not considered. In two and three dimensions and under the assumptions that $\Omega$ \EO{is of class $C^{1,1}$,} the mesh--size is sufficiently small, and that both the optimal state $(\bar{\mathbf{y}},\bar{p})$ and adjoint state $(\bar{\mathbf{z}},\bar{r})$ belong to $\mathbf{H}^2(\Omega)\times H^1(\Omega)$, the authors prove, for the distributed case, that 
$
\| \bar{\mathbf{u}}-\bar{\mathbf{u}}_{\T} \|_{\mathbf{L}^2(\Omega)}
\lesssim
h^2_{\T}
$
\cite[Corollary 4.5 and section 5.2]{MR1079020}. Here, $\bar{\mathbf{u}}_\T$ denotes the corresponding finite element approximation of the optimal control variable $\bar{\mathbf{u}}$. Later, the authors of \cite{MR2338434} derived error estimates for suitable finite element approximations of \eqref{eq:minimize_cost_func}--\eqref{def:box_constraints}. Notice that control constraints are considered. Under the assumption that $\Omega$ is of class $C^2$, the authors show that the $\mathbf{L}^2(\Omega)$-norm of the error approximation of the control variable behaves as $h_{\T}^2$, when the control set is not discretized, and as $h_{\T}$, when such a set is discretized by using piecewise constant functions \cite[Theorem 4.18]{MR2338434}.  These error estimates are obtained for local solutions of the optimal control problem which satisfy a second order sufficient optimality condition and which are nonsingular, in the sense that the linearized Navier--Stokes equations around these solutions define some isomorphisms.

A class of numerical methods that has proven useful for approximating the solution to PDE--constrained optimization problems, and the ones we will use in this work, are the so-called adaptive finite element methods (AFEMs). AFEMs are iterative methods that improve the quality of the finite element approximation to a partial differential equation (PDE) while striving to keep an optimal distribution of computational resources measured in terms of degrees of freedom. An essential ingredient of an AFEM is an a posteriori error estimator, which is a computable quantity that depends on the discrete solution and data and provides information about the local quality of the approximate solution. The a posteriori error analysis for optimal control problems that are based on the minimization of a quadratic functional subject to a \emph{linear} PDE and control constraints has achieved several advances in recent years. \EO{We refer the reader to \cite{MR1780911,MR3621827,MR2434065,MR3212590,MR1887737,MR1950625} for a discussion. In particular, we mention the work \cite{MR3621827} where the authors consider the so-called variational discretization approach, introduced by Hinze in \cite{MR2122182}, and prove the convergence and quasi-optimality of suitable AFEMs for a class of linear-quadratic control-constrained optimal control problems.} As opposed to these advances, the analysis of AFEMs for optimal control problems involving nonlinear equations is rather scarce. We mention the approach introduced in \cite{MR1780911} for estimating the error in terms of the cost functional for semilinear optimal control problems \cite[section 6]{MR1780911} and its extensions to problems with control constraints  \cite{MR2421327,MR2373479}. Recently, the authors of \cite{semilinear_aposteriori} have studied a posteriori error estimates for a distributed semilinear elliptic optimal control problem\FF{; the results obtained in this work complement and extend the ones derived in \cite{MR2024491}. In \cite{semilinear_aposteriori}, the authors} have proposed a general framework that, on the basis of global reliability estimates for the state and adjoints equations and second order optimality conditions, yields a global reliability result for the proposed error estimator of the underlying optimal control problem. For a particular residual--type setting, the authors obtain, on the basis of bubble functions arguments, local efficiency estimates. Regarding a posteriori error estimates for optimal control problems involving \eqref{eq:state_equations} we mention references \cite{MR1867885,MR1923792}. \EO{Within the setting of boundary control problems,} the authors of \cite{MR1867885,MR1923792} invoke the approach of \cite{MR1780911} and construct an upper bound for the error $J(\bar{\mathbf{y}},\bar{\mathbf{u}})-J(\bar{\mathbf{y}}_\T,\bar{\mathbf{u}}_\T)$. An efficiency analysis is, however, not provided. \EO{To conclude this paragraph, we mention the work \cite{MR3843444}, where the authors develop and implement, on the basis of a dual-weighted residual approach, an adaptive solution technique for the optimal control of a time--discrete Cahn--Hilliard--Navier--Stokes system with variable densities.}

\EO{In this work, we consider two strategies to discretize the optimal control problem  \eqref{eq:minimize_cost_func}--\eqref{def:box_constraints}: a semi discrete scheme, where the control variable is not discretized  -- the so-called variational discrezation approach \cite{MR2122182} -- and a fully discrete scheme, where the control is discretized. For each one of these strategies we devise a residual--based a posteriori error estimator. For the fully discrete scheme the error estimator is decomposed as the sum of three contributions, which are related to the discretization of the state and adjoint equations and the control variable. Instead, the error estimator for the variational discretization approach is decomposed only in two contributions that are related to the discretization of the state and adjoint equations}. We must immediately mention that, as is usual in the a posteriori error analysis for the Navier--Stokes equations, we shall operate under a smallness assumption on data; see \cite{MR1885308,MR1259620,MR3059294}. Under this assumption, in two and three dimensional Lipschitz polytopes, we obtain reliability and efficiency estimates. \EO{On the basis of the devised error estimators, we also design simple adaptive strategies that exhibit, for the examples that we perform, optimal experimental rates of convergence for all the optimal variables but with the exception of the control variable when the fully discrete scheme is considered}. Several remarks and comparisons with the existing literature are now in order:
\begin{enumerate}
\item[$\bullet$] \EO{We show that our error estimators are globally reliable; see Theorems \ref{thm:control_bound} and \ref{thm:control_bound_variational}. In contrast to \cite{MR1867885,MR1923792}, we also explore local efficiency estimates; see Theorems \ref{thm:global_eff} and \ref{thm:global_eff_var}.}

\item[$\bullet$] In contrast to the a priori theory developed in \cite{MR2338434,MR1079020}, our a posteriori error analysis only requires that $\Omega$ is a Lipschitz polytope, $\mathbf{y}_{\Omega} \in \mathbf{L}^2(\Omega)$, \EO{and that the optimal state $(\bar{\mathbf{y}},\bar{p})$ and the optimal adjoint state $(\bar{\mathbf{z}},\bar{r})$ 
belong to $ \mathbf{H}_0^1(\Omega) \cap \mathbf{W}^{1,3}(\Omega) \times L^2(\Omega)$ and $\mathbf{H}_0^1(\Omega)\times L^2(\Omega)$, respectively.}

\item[$\bullet$] As opposed to the case when the state equation is linear, where, in general, only first order optimality conditions are needed to obtain a posteriori error estimates, the strategy that we develop here relies on the use of a second order sufficient optimality condition and on the particular structure of the associated critical cone; see Theorems \ref{thm:error_bound_control_tilde}, \ref{thm:control_bound}, \FF{and \ref{thm:control_bound_variational}.}
\end{enumerate}

The rest of the paper is organized as follows. In section \ref{sec:not_and_prel} we set notation and review some preliminaries for the Navier--Stokes equations. Basic results for the optimal control problem \eqref{eq:minimize_cost_func}--\eqref{def:box_constraints} as well as first and second order optimality conditions are reviewed in section \ref{sec:ocp}. \EO{The core of our work are sections \ref{sec:a_posteriori} and \ref{sec:a_posteriori_semi}, where we design an a posteriori error estimator for the fully discrete and the semi discrete scheme, respectively. We show the global reliability of each error estimator and analyze local efficiency estimates.} Finally, two and three dimensional numerical examples are presented in section \ref{sec:numerical_ex}. These examples illustrate the theory and reveal a competitive performance of the devised \EO{AFEMs.}


\section{Notation and preliminaries}\label{sec:not_and_prel}
Let us set notation and recall some facts that will be useful later.


\subsection{Notation}\label{sec:notation}

\EO{Let $d \in \{1,2,3\}$ and $U \subset \mathbb{R}^d$ be an open and bounded domain. We shall use standard notation for Lebesgue and Sobolev spaces.} The space of functions in $L^2(U)$ that have zero average is denoted by $L^2_0(U)$. By $W^{m,t}(U)$, we denote the Sobolev space of functions in $L^t(U)$ with partial derivatives of order up to $m$ in $L^t(U)$. Here, \EO{$0 < m < \infty$} and $1 \leq t \leq \infty$. The closure with respect to the norm in $W^{m,t}(U)$ of the space of $C^{\infty}$ functions compactly supported in $U$ is denoted by $W_0^{m,t}(U)$.
When $t = 2$ and $m \in [0,\infty)$, we set $H^{m}(U) := W^{m,2}(U)$ and $H_0^{m}(U) := W_0^{m,2}(U)$.
We use bold letters to denote the vector-valued counterparts of the aforementioned spaces. In particular, we set
\[
\mathbf{V}(U):= \{ \mathbf{v} \in \mathbf{H}_{0}^{1}(U): \text{div }\mathbf{v}  = 0\}.
\]

If $\mathcal{X}$ and $\mathcal{Y}$ are normed vector spaces, we write $\mathcal{X}\hookrightarrow\mathcal{Y}$ to denote that $\mathcal{X}$ is continuously embedded in $\mathcal{Y}$. The relation $a \lesssim b$ indicates that $a \leq C b$, with a positive constant that depends neither on $a$, $b$ nor the discretization \EO{parameters}. The value of $C$ might change at each occurrence.

Finally, throughout this work, $\Omega$ denotes an open and bounded polytopal domain in $\mathbb{R}^d$ $(d \in \{2,3\})$ with Lipschitz boundary $\partial\Omega$.


\subsection{Preliminaries for the Navier--Stokes equations}\label{sec:preliminaries}

Let us, for the sake of future reference, collect here some standard results \EO{concerning} the analysis of \eqref{eq:state_equations}.

In order to write a weak formulation for \eqref{eq:state_equations}, we introduce the trilinear form
\begin{equation*}
b(\mathbf{v}_1;\mathbf{v}_2,\mathbf{v}_3):=((\mathbf{v}_1\cdot \nabla)\mathbf{v}_2,\mathbf{v}_3)_{\mathbf{L}^2(\Omega)}.
\end{equation*}
The form $b$ satisfies the following properties \cite[Chapter IV, Lemma 2.2]{MR851383}, \cite[Chapter II, Lemma 1.3]{MR603444}: Let $\mathbf{v}_1\in \mathbf{V}(\Omega)$ and $\mathbf{v}_2, \mathbf{v}_3 \in \mathbf{H}_0^1(\Omega)$. Then, we have
\begin{equation}\label{eq:properties_trilinear}
b(\mathbf{v}_1;\mathbf{v}_2,\mathbf{v}_3)=-b(\mathbf{v}_1;\mathbf{v}_3,\mathbf{v}_2), 
\qquad  
b(\mathbf{v}_1;\mathbf{v}_2,\mathbf{v}_2)=0.
\end{equation}
The form $b$ is well-defined and continuous on \EO{$\mathbf{H}_0^1(\Omega)$} and
\begin{equation}\label{eq:trilinear_embedding}
|b(\mathbf{v}_1;\mathbf{v}_2,\mathbf{v}_3)|\leq \mathcal{C}_b\|\nabla \mathbf{v}_1\|_{\mathbf{L}^2(\Omega)}\|\nabla \mathbf{v}_2\|_{\mathbf{L}^2(\Omega)}\|\nabla \mathbf{v}_3\|_{\mathbf{L}^2(\Omega)},
\end{equation}
where $\mathcal{C}_b>0$; see \cite[Lemma IX.1.1]{Gal11} and \cite[Chapter II, Lemma 1.1]{MR603444}.

\EO{We will also make use of the fact that, on Lipschitz domains, the divergence operator is surjective from $\mathbf{H}_0^1(\Omega)$ to $L_0^2(\Omega)$. This implies that there is a constant $\beta >0$, that depends only on $d$ and $\Omega$,} such that \cite[Chapter I, section 5.1]{MR851383}, \cite[Corollary B. 71]{MR2050138}
\begin{equation}\label{eq:inf_sup_cond}
\sup_{\mathbf{v}\in\mathbf{H}_0^1(\Omega)}\frac{(q,\text{div }\mathbf{v})_{L^2(\Omega)}}{\|\nabla \mathbf{v}\|_{\mathbf{L}^2(\Omega)}}\geq \beta\|q\|_{L^2(\Omega)}  \qquad \forall q\in L_0^2(\Omega).
\end{equation}

With these ingredients at hand, we introduce the following weak formulation of problem \eqref{eq:state_equations} \EO{\cite[Chapter IV, problem (2.8)]{MR851383}: Given $\mathbf{f} \in \mathbf{H}^{-1}(\Omega)$, find $(\mathbf{y},p) \in \mathbf{V}(\Omega)\times L_0^2(\Omega)$ such that 
\begin{equation}\label{eq:weak_ns_divergence_free}
\nu(\nabla \mathbf{y}, \nabla \mathbf{v})_{\mathbf{L}^2(\Omega)}+b(\mathbf{y};\mathbf{y},\mathbf{v})-(p,\text{div } \mathbf{v})_{L^2(\Omega)}  = \langle \mathbf{f},\mathbf{v} \rangle  \quad  \forall \mathbf{v} \in \mathbf{H}_0^1(\Omega).
\end{equation}}
Here, $\langle \cdot,\cdot \rangle$ denotes the duality pairing between $\mathbf{H}^{-1}(\Omega)$ and $\mathbf{H}_0^1(\Omega)$.

Denote by $C_2$ the \EO{best} constant in the Sobolev embedding $\mathbf{H}_0^1(\Omega)\hookrightarrow \mathbf{L}^2(\Omega)$. The following result states the existence and uniqueness of solutions for the Navier--Stokes equations for small data (see \cite[Chapter IV, Theorem 2.2]{MR851383} and \cite[Chapter II, Theorem 1.3]{MR603444}). Since it will be useful later, we restrict the discussion to $\mathbf{f}\in\mathbf{L}^2(\Omega)$.
 
\begin{theorem}[well--posedness]\label{thm:well_posedness_navier_stokes}
If $\|\mathbf{f}\|_{\mathbf{L}^2(\Omega)}< C_2^{-1}\mathcal{C}_b^{-1}\nu^2 $, then there exists a unique solution \FF{$(\mathbf{y},p) \in \mathbf{V}(\Omega)\times L_0^2(\Omega)$ of problem \eqref{eq:weak_ns_divergence_free}}. In addition, we have
\begin{equation}\label{eq:stability_state_eq}
\|\nabla \mathbf{y}\|_{\mathbf{L}^2(\Omega)}\leq \theta \mathcal{C}_b^{-1}\nu, \quad \theta\in[0,1).
\end{equation}
\end{theorem}

\EO{
\begin{remark}[equivalent formulation]\label{rmk:equivalent}
Notice that problem \eqref{eq:weak_ns_divergence_free} can be equivalently formulated as follows: Find $(\mathbf{y},p)\in\mathbf{H}_0^1(\Omega)\times L_0^2(\Omega)$ such that 
\begin{equation}\label{eq:weak_navier_stokes_eq}
\begin{aligned}
\nu(\nabla \mathbf{y}, \nabla \mathbf{v})_{\mathbf{L}^2(\Omega)}+b(\mathbf{y};\mathbf{y},\mathbf{v})-(p,\text{div } \mathbf{v})_{L^2(\Omega)} & = \langle \mathbf{f},\mathbf{v} \rangle \quad & \forall \mathbf{v} \in \mathbf{H}_0^1(\Omega),\\
(q,\text{div } \mathbf{y})_{L^2(\Omega)} & =   0  \quad & \forall q\in L^2_0(\Omega);
\end{aligned}
\end{equation}
see \cite[Chapter IV, Section 2]{MR851383} for details. Since formulations \eqref{eq:weak_ns_divergence_free} and \eqref{eq:weak_navier_stokes_eq} are equivalent, Theorem \ref{thm:well_posedness_navier_stokes} guarantees the existence of a unique solution $(\mathbf{y},p)\in\mathbf{H}_0^1(\Omega)\times L_0^2(\Omega)$ to \eqref{eq:weak_navier_stokes_eq} under the assumption that $\|\mathbf{f}\|_{\mathbf{L}^2(\Omega)}< C_2^{-1}\mathcal{C}_b^{-1}\nu^2$.
\end{remark}
}


\section{The optimal control problem}\label{sec:ocp}
In this section we present a weak formulation for problem \eqref{eq:minimize_cost_func}--\eqref{def:box_constraints}. We review first and second order optimality conditions in sections \ref{sec:first_second_opt_cond} and \ref{sec:second_opt_cond}, respectively, and introduce, in section \ref{sec:fem_ocp}, \EO{finite element discretization schemes}. 

We consider the following weak version of the control problem \eqref{eq:minimize_cost_func}--\eqref{def:box_constraints}: Find
\begin{equation}\label{eq:min_weak_setting}
\min_{\mathbf{H}_0^1(\Omega)\times\mathbb{U}_{ad}}J(\mathbf{y},\mathbf{u})
\end{equation}
subject to the weak formulation of the stationary Navier--Stokes equations
\begin{equation}\label{eq:weak_st_equation}
\begin{aligned}
\nu(\nabla \mathbf{y}, \nabla \mathbf{v})_{\mathbf{L}^2(\Omega)}+b(\mathbf{y};\mathbf{y},\mathbf{v})-(p,\text{div } \mathbf{v})_{L^2(\Omega)} &  =  (\mathbf{u},\mathbf{v})_{\mathbf{L}^2(\Omega)}  & \forall \mathbf{v} \in \mathbf{H}_0^1(\Omega),\\
(q,\text{div } \mathbf{y})_{L^2(\Omega)} & =   0   &\forall q\in L^2_0(\Omega).
\end{aligned}
\end{equation}
\EO{We observe that, in view of the results of \cite[Chapter IV]{MR851383}, problem \eqref{eq:weak_st_equation} can be equivalently written as follows:
\begin{equation}\label{eq:st_reduced_form}
\mathbf{y}\in\mathbf{V}(\Omega): \quad \nu(\nabla\mathbf{y},\nabla \mathbf{v})_{\mathbf{L}^2(\Omega)}+b(\mathbf{y};\mathbf{y},\mathbf{v})=(\mathbf{u},\mathbf{v})_{\mathbf{L}^2(\Omega)} \quad \forall \mathbf{v}\in \mathbf{V}(\Omega).
\end{equation}

Let us define $\mathfrak{M}_{ad}:=\sup_{\mathbf{u}\in\mathbb{U}_{ad}}\|\mathbf{u}\|_{\mathbf{L}^2(\Omega)}$ and assume that \cite{Ru,MR2192070}
\begin{equation}\label{eq:smallness_assumption}
\frac{\mathcal{C}_b C_{2}}{\nu^2}
\mathfrak{M}_{ad}\leq \theta
<1, \qquad \theta\in (0,1). 
\end{equation}
Owing} to Theorem \ref{thm:well_posedness_navier_stokes} and Remark \ref{rmk:equivalent} we conclude that, for each $\mathbf{u}\in\mathbb{U}_{ad}$, there exists a unique pair $(\mathbf{y},p)\in \mathbf{H}_0^1(\Omega)\times L_0^2(\Omega)$ that solves \eqref{eq:weak_st_equation} \FF{or equivalently \eqref{eq:st_reduced_form}}. 

\EO{
\begin{remark}[assumption \eqref{eq:smallness_assumption}]
Assumption \eqref{eq:smallness_assumption} is an additional constraint that restricts the admissible set $\mathbb{U}_{ad}$. Observe that, if we assume \eqref{eq:smallness_assumption}, then $\mathbf{a},\mathbf{b}\in \mathbb{R}^{d}$ must satisfy the upper bound $\max\{\|\mathbf{a}\|_{\mathbf{L}^2(\Omega)},\|\mathbf{b}\|_{\mathbf{L}^2(\Omega)}\} \leq \theta C_2^{-1}\mathcal{C}_b^{-1}\nu^2$, with $\theta < 1$. From now on, we shall assume that assumption \eqref{eq:smallness_assumption} holds.
\end{remark}

\begin{remark}[on the constants in assumption \eqref{eq:smallness_assumption}] For practical applications, it would be desirable to have at hand the values of the constants $\mathcal{C}_b$ and $C_{2}$. However, determining the explicit values of these constants is, in general, a very difficult task. In spite of this fact, estimates for these constant can be found in the literature \cite[Lemma IX.1.1, equation (II.5.5)]{Gal11}:
\begin{equation*}
\mathcal{C}_{b}
=
\frac{1}{2}|\Omega|^{\frac{1}{2}}\text{ if } d=2, 
\qquad 
\mathcal{C}_{b}
=
\frac{2\sqrt{2}}{3}|\Omega|^{\frac{1}{6}} \text{ if } d=3,
\qquad
C_{2} 
=
\frac{(d-1)}{\sqrt{d}}|\Omega|^{\frac{1}{d}} \text{ if } d\in\{2,3\}.
\end{equation*}
Here, $|\Omega|$ denotes the Lebesgue measure of  $\Omega$. 
\end{remark}}


\subsection{Local solutions}
Since the optimal control problem \eqref{eq:min_weak_setting}--\eqref{eq:weak_st_equation} is not convex, we \FF{study} optimality conditions in the context of local solutions. A control $\bar{\mathbf{u}} \in \mathbb{U}_{ad}$ is said to be locally optimal in $\mathbf{L}^2(\Omega)$ for \eqref{eq:min_weak_setting}--\eqref{eq:weak_st_equation}, if there exists $\delta>0$ such that
\[
J(\bar{\mathbf{y}},\bar{\mathbf{u}})\leq J(\mathbf{y},\mathbf{u})
\]
for all $\mathbf{u}\in\mathbb{U}_{ad}$ such that $\|\mathbf{u}-\bar{\mathbf{u}}\|_{\mathbf{L}^2(\Omega)}\leq \delta$. Here, $\bar{\mathbf{y}}$ and $\mathbf{y}$ denote the velocity fields associated to $\bar{\mathbf{u}}$ and $\mathbf{u}$, respectively. \EO{We stress that the smallness assumption \eqref{eq:smallness_assumption} is involved within this definition of \emph{local optimality}.}

The existence of a local solution $(\bar{\mathbf{y}},\bar{\mathbf{u}})\in \mathbf{H}_0^1(\Omega)\times \mathbb{U}_{ad}$ for problem \eqref{eq:min_weak_setting}--\eqref{eq:weak_st_equation} follows standard arguments; see \cite[Theorem 3.1]{MR2109045}.


\subsection{First order optimality conditions}\label{sec:first_second_opt_cond}

\EO{We now turn our attention to the discussion of} first and second order optimality conditions for problem \eqref{eq:min_weak_setting}--\eqref{eq:weak_st_equation}. We begin such a discussion by introducing the so-called control-to-state map $\mathcal{S}: \mathbf{L}^2(\Omega)\rightarrow \mathbf{V}(\Omega)$ which, given a control $\mathbf{u}\in \mathbb{U}_{ad}$, associates to it the velocity field $\mathbf{y} \in \mathbf{V}(\Omega)$ \EO{of the unique pair $(\mathbf{y},p)$} that solves \eqref{eq:st_reduced_form} under the smallness assumption \eqref{eq:smallness_assumption}. With this operator at hand, we introduce the reduced cost functional 
\[
j(\mathbf{u}):=J(\mathcal{S}\mathbf{u},\mathbf{u})=\frac{1}{2}\|\mathcal{S}\mathbf{u}-\mathbf{y}_\Omega\|_{\mathbf{L}^2(\Omega)}^2+\frac{\alpha}{2}\|\mathbf{u}\|_{\mathbf{L}^2(\Omega)}^2.
\]

Under the smallness assumption \eqref{eq:smallness_assumption}, the control-to-state map $\mathcal{S}$ is Fr\'echet differentiable from $\mathbf{L}^2(\Omega)$ to $\mathbf{V}(\Omega)$; see \cite[Lemma 3.8]{MR2192070}. As a consequence, if $\bar{\mathbf{u}}$ denotes a local optimal control for problem \eqref{eq:min_weak_setting}--\eqref{eq:weak_st_equation}, $\bar{\mathbf{u}}$ satisfies the variational inequality
\begin{equation}\label{eq:variational_ineq}
(\bar{\mathbf{z}}+\alpha \bar{\mathbf{u}},\mathbf{u}-\bar{\mathbf{u}})_{\mathbf{L}^2(\Omega)}
\geq 
0 
\quad \forall \mathbf{u}\in\mathbb{U}_{ad},
\end{equation}
where $(\bar{\mathbf{z}},\bar{r})\in \mathbf{H}_0^1(\Omega)\times L_0^2(\Omega)$ is the unique solution to the adjoint equations
\begin{equation}\label{eq:weak_adjoint_equation}
\begin{array}{rcl}
\nu(\nabla \mathbf{w}, \nabla \bar{\mathbf{z}})_{\mathbf{L}^2(\Omega)}
+
b(\bar{\mathbf{y}};\mathbf{w},\bar{\mathbf{z}})
+
b(\mathbf{w};\bar{\mathbf{y}},\bar{\mathbf{z}})
-
(\bar{r},\text{div } \mathbf{w})_{L^2(\Omega)} 
& \hspace{-0.3cm} = \hspace{-0.25cm} & (\bar{\mathbf{y}}-\mathbf{y}_\Omega,\mathbf{w})_{\mathbf{L}^2(\Omega)}, 
\\
(s,\text{div } \bar{\mathbf{z}})_{L^2(\Omega)} & \hspace{-0.3cm} = \hspace{-0.25cm} & 0,
\end{array}\hspace{-0.3cm}
\end{equation}
for all $(\mathbf{w},s)\in \mathbf{H}_0^1(\Omega)\times L_0^2(\Omega)$. For details, we refer the reader to \cite[Theorem 2.2]{Ru}, \cite[Theorem 3.10]{MR2192070}, and \cite[Theorem 3.2]{MR2338434}.

\begin{remark}[well--posedness of adjoint equations]
\label{remark:wp_ad}
Define $\mathcal{B}: \mathbf{H}_0^1(\Omega) \times \mathbf{H}_0^1(\Omega)$ by
\begin{equation*}
\mathcal{B}(\mathbf{w},\mathbf{v}):=\nu (\nabla \mathbf{w},\nabla\mathbf{v})_{\mathbf{L}^2(\Omega)}+b(\bar{\mathbf{y}};\mathbf{w},\mathbf{v})+b(\mathbf{w};\bar{\mathbf{y}}, \mathbf{v}).
\end{equation*}
Assume that \eqref{eq:smallness_assumption} holds. Theorem \ref{thm:well_posedness_navier_stokes} thus yields $\| \nabla \bar{ \mathbf{y}} \|_{\mathbf{L}^2(\Omega)} \leq \theta \mathcal{C}_b^{-1} \nu$. Consequently,
\begin{equation}
\label{eq:coercivity_bilinear_form_B}
\mathcal{B}(\mathbf{w},\mathbf{w}) \geq \nu  \|  \nabla \mathbf{w} \|^2_{\mathbf{L}^2(\Omega)} - \mathcal{C}_b \| \nabla  \mathbf{w} \|^2_{\mathbf{L}^2(\Omega)} \| \nabla \bar{\mathbf{y}} \|_{\mathbf{L}^2(\Omega)} \geq \nu(1-\theta) \|  \nabla \mathbf{w} \|^2_{\mathbf{L}^2(\Omega)},
\end{equation}
where, we recall $\theta < 1$. We have thus proved that $\mathcal{B}$ is coercive on $\mathbf{H}_0^1(\Omega) \times \mathbf{H}_0^1(\Omega)$. The standard inf--sup theory for saddle point problems \cite[Theorem 2.34]{MR2050138} yields, on the basis of \eqref{eq:inf_sup_cond}, the existence and uniqueness of a solution $(\bar{\mathbf{z}},\bar{r})\in \mathbf{H}_0^1(\Omega)\times L_0^2(\Omega)$ to \eqref{eq:weak_adjoint_equation}. \EO{In addition, by setting $\mathbf{w} = \bar{\mathbf{z}}$ and $s=0$ and utilizing} \eqref{eq:properties_trilinear}, \eqref{eq:trilinear_embedding}, and \eqref{eq:smallness_assumption} we arrive at the bound
\begin{equation}\label{eq:stability_adjoint_par}
\|\nabla\bar{\mathbf{z}}\|_{\mathbf{L}^2(\Omega)} \leq \EO{[\nu(1-\theta)]^{-1}C_2}\|\bar{\mathbf{y}}-\mathbf{y}_\Omega\|_{\mathbf{L}^2(\Omega)}.
\end{equation}
\end{remark}

The local optimal control $\bar{\mathbf{u}}$ satisfies \eqref{eq:variational_ineq} if and only if \cite[equation (2.10)]{Ru}, \cite[equation (3.9)]{MR2338434}
\begin{equation}\label{eq:projection_formula}
\bar{\mathbf{u}}(x)=\Pi_{[\mathbf{a},\mathbf{b}]}\left(-\alpha^{-1}\bar{\mathbf{z}}(x)\right)  \textrm{ a.e. } x \in \Omega,
\end{equation}
where the projection operator $\Pi_{[\mathbf{a},\mathbf{b}]}:\mathbf{L}^1(\Omega)\rightarrow \mathbb{U}_{ad}$ is defined as
\begin{equation}\label{def:projection_operator}
\Pi_{[\mathbf{a},\mathbf{b}]}(\mathbf{v}):=\min\{\mathbf{b},\max\{\mathbf{v},\mathbf{a}\}\}.
\end{equation}


\subsection{Second order optimality conditions}\label{sec:second_opt_cond}

We now follow \cite[section 3.2]{MR2338434} and present necessary and sufficient second order optimality conditions. \EO{To present them, we introduce the variable} $\bar{\mathbf{d}}:=\bar{\mathbf{z}}+\alpha\bar{\mathbf{u}}$ and the cone of critical directions 
\begin{equation*}
\mathbf{C}_{\bar{\mathbf{u}}}:=\{\mathbf{v}\in \mathbf{L}^2(\Omega) \text{ that satisfies } \eqref{eq:critical_cone_charac} \text{ and } \mathbf{v}_i(x)=0 \text{ if } \bar{\mathbf{d}}_i(x)\neq 0, \, i=1,\ldots,d\}.
\end{equation*}
Here, $\bar{\mathbf{d}}_i$ corresponds to the $i$-th component of the vector $\bar{\mathbf{d}}$ and condition \eqref{eq:critical_cone_charac} \EO{reads
\begin{equation}\label{eq:critical_cone_charac}
\mathbf{v}_i(x)
\begin{cases}
\geq 0  \text{ a.e. } x \in \Omega \text{ if } \bar{\mathbf{u}}_i(x) = \mathbf{a}_i \text{ and } \bar{\mathbf{d}}_i(x) = 0, \\
\leq 0  \text{ a.e. } x \in \Omega \text{ if } \bar{\mathbf{u}}_i(x) = \mathbf{b}_i \text{ and } \bar{\mathbf{d}}_i(x) = 0,
\end{cases}
\end{equation}
where $i \in \{1,\dots,d\}$. We} are now in position to present second order necessary and sufficient optimality conditions; see \cite[Theorems 3.6 \EO{and 3.8} and Corollary 3.9]{MR2338434}.

\begin{theorem}[second order optimality conditions]\label{thm:second_necess_opt_cond} 
Assume that \eqref{eq:smallness_assumption} holds. If $\bar{\mathbf{u}} \in \mathbb{U}_{ad}$ is a local minimum for problem \eqref{eq:min_weak_setting}--\eqref{eq:weak_st_equation}, then
\begin{equation*}
j''(\bar{\mathbf{u}})\mathbf{v}^2 \geq 0 \quad \forall \mathbf{v} \in \mathbf{C}_{\bar{\mathbf{u}}}.
\end{equation*}
Conversely, if $(\bar{\mathbf{y}},\bar{p},\bar{\mathbf{z}},\bar{r},\bar{\mathbf{u}}) \in \mathbf{H}_0^1(\Omega)\times L_0^2(\Omega) \times \mathbf{H}_0^1(\Omega)\times L_0^2(\Omega) \times \mathbb{U}_{ad}$ satisfies the first order optimality conditions \eqref{eq:weak_st_equation}, \eqref{eq:variational_ineq}, and \eqref{eq:weak_adjoint_equation}, and  
\begin{equation}\label{eq:equivalent_second_I}
j''(\bar{\mathbf{u}})\mathbf{v}^2 > 0 \quad \forall \mathbf{v} \in \mathbf{C}_{\bar{\mathbf{u}}}\setminus \{\boldsymbol 0\},
\end{equation}
then, there exist \EO{$\delta >0$} and $\varepsilon>0$ such that
\begin{equation*}
j(\mathbf{u})\geq j(\bar{\mathbf{u}})+ \frac{\EO{\delta}}{2}\left(\|\mathbf{u}-\bar{\mathbf{u}}\|_{\mathbf{L}^2(\Omega)}^2+\|\mathbf{y}-\bar{\mathbf{y}}\|_{\mathbf{L}^2(\Omega)}^2\right),
\end{equation*}
for every pair $(\mathbf{u},\mathbf{y})$ that satisfies \eqref{eq:state_equations}, $\mathbf{u} \in \mathbb{U}_{ad}$, and $\|\mathbf{u}-\bar{\mathbf{u}}\|_{\mathbf{L}^2(\Omega)}^2+\|\mathbf{y}-\bar{\mathbf{y}}\|_{\mathbf{L}^2(\Omega)}^2\leq \varepsilon$.
\end{theorem}

We now introduce, \EO{for} $\tau >0$, the cone
\begin{equation}\label{def:cone_tau}
\mathbf{C}_{\bar{\mathbf{u}}}^{\tau}:=\{ 
\mathbf{v}\in \mathbf{L}^2(\Omega) \text{ that satisfies } \eqref{eq:critical_cone_charac_tau} 
\},
\end{equation}
where condition \eqref{eq:critical_cone_charac_tau} reads, \EO{for $ i \in \{1,\ldots,d\}$}, as follows:
\begin{equation}\label{eq:critical_cone_charac_tau}
\mathbf{v}_i(x)
\begin{cases}
= 0 \text{ if } |\bar{\mathbf{d}}_i(x)|> \tau,
\\
\geq 0  \text{ a.e. } x \in \Omega \text{ if } \bar{\mathbf{u}}_i(x) = \mathbf{a}_i \textrm{ and }|\bar{\mathbf{d}}_i(x)| \leq \tau, \\
\leq 0  \text{ a.e. } x \in \Omega \text{ if } \bar{\mathbf{u}}_i(x) = \mathbf{b}_i \textrm{ and } |\bar{\mathbf{d}}_i(x)| \leq \tau.
\end{cases}
\end{equation}

The next result will be of importance for deriving a posteriori error estimates for the discretizations of \eqref{eq:min_weak_setting}--\eqref{eq:weak_st_equation} that we will propose; see \cite[Corollary 3.11]{MR2338434}.

\begin{theorem}[equivalent second order optimality condition]\label{thm:equivalent_opt_cond}
Assume that \eqref{eq:smallness_assumption} holds. Let $(\bar{\mathbf{y}},\bar{p},\bar{\mathbf{z}},\bar{r},\bar{\mathbf{u}}) \in \mathbf{H}_0^1(\Omega)\times L_0^2(\Omega) \times \mathbf{H}_0^1(\Omega)\times L_0^2(\Omega) \times \mathbb{U}_{ad}$ be a local solution to \eqref{eq:min_weak_setting}--\eqref{eq:weak_st_equation} satisfying the first order optimality conditions \eqref{eq:weak_st_equation}, \eqref{eq:variational_ineq}, and \eqref{eq:weak_adjoint_equation}. Then, \eqref{eq:equivalent_second_I} is equivalent to the existence of $\EO{\mu} > 0$ and $\tau > 0$ such that
\begin{equation}\label{eq:equivalent_second_II}
j''(\bar{\mathbf{u}})\mathbf{v}^2 \geq \EO{\mu} \|\mathbf{v}\|_{\mathbf{L}^2(\Omega)}^2 \quad \forall \mathbf{v} \in \mathbf{C}_{\bar{\mathbf{u}}}^{\tau}.
\end{equation}
\end{theorem}

We close this section with the next result.

\begin{lemma}[property of $j''$]\label{lemma:property_of_j}
Let $\mathbf{u}, \mathbf{h}, \mathbf{v} \in \mathbf{L}^\infty(\Omega)$ and $\mathsf{M} >0$ be such that $\max\{\| \mathbf{u}\|_{\mathbf{L}^\infty(\Omega)},\|\mathbf{h}\|_{\mathbf{L}^\infty(\Omega)}\} \leq \mathsf{M}$. Then, there exists $C_{\mathsf{M}} > 0$ such that
\begin{equation}\label{eq:estimate_second_der}
|j''(\mathbf{u} + \mathbf{h})\mathbf{v}^2-j''(\mathbf{u})\mathbf{v}^2|\leq C_{\mathsf{M}} \|\mathbf{h}\|_{\mathbf{L}^\FF{2}(\Omega)}\|\mathbf{v}\|_{\mathbf{L}^2(\Omega)}^2.
\end{equation}
\end{lemma}


\subsection{Finite element approximation}\label{sec:fem_ocp}

We now introduce the discrete setting in which we will operate. We consider $\mathscr{T}=\{T\}$ to be a conforming partition  of $\overline{\Omega}$ into closed simplices $T$ with size $h_T=\text{diam}(T)$. Define $h_{\mathscr{T}}:=\max_{T\in\mathscr{T}}h_T$. We denote by $\mathbb{T}$ the collection of conforming and shape regular meshes that are refinements of an initial mesh $\mathscr{T}_0$. Let $\mathscr{S}$ be the set of internal $(d-1)-$dimensional interelement boundaries $S$ of $\mathscr{T}$. For $T \in \mathscr{T}$, let $\mathscr{S}_T$ denote the subset of $\mathscr{S}$ which contains the sides in $\mathscr{S}$ which are sides of $T$. We denote by $\mathcal{N}_S$ the subset of $\mathscr{T}$ that contains the two elements that have $S$ as a side, i.e., $\mathcal{N}_S=\{T^+,T^-\}$, where $T^+, T^- \in \mathscr{T}$ are such that $S = T^+ \cap T^-$. For $T \in \mathscr{T}$, we define the \emph{star} associated with the element $T$ as
\begin{equation}\label{def:patch}
\mathcal{N}_T:= \left \{ T^{\prime}\in\mathscr{T}: \mathscr{S}_{T}\cap \mathscr{S}_{T^\prime}\neq\emptyset \right \}.
\end{equation}
In an abuse of notation, in what follows, by $\mathcal{N}_T$ we will indistinctively denote either this set or the union of the triangles that comprise it.

For a discrete tensor valued function $\mathbf{W}_{\mathscr{T}}$, we define the jump or interelement residual on the internal side $S\in\mathscr{S}$, shared by the distinct elements $T^+, T^-\in\mathcal{N}_S$, by $\llbracket{\mathbf{W}_{\mathscr{T}}\cdot\mathbf{n}} \rrbracket=\mathbf{W}_{\mathscr{T}}|_{T^+}\cdot\mathbf{n}^+ +\mathbf{W}_{\mathscr{T}}|_{T^-}\cdot\mathbf{n}^-$. Here, $\mathbf{n}^+$ and $\mathbf{n}^-$ are unit normal on $S$ pointing towards $T^+$ and $T^-$, respectively.

\EO{Finally,} we introduce the inf-sup stable finite element spaces that we will consider in our work. Given a mesh $\mathscr{T}\in\mathbb{T}$, we denote by $\mathbf{V}(\mathscr{T})$ and $\mathcal{P}(\mathscr{T})$ the finite element spaces that approximate the velocity field and the pressure,  respectively,  based on the classical Taylor--Hood elements \cite[section 4.2.5]{MR2050138}: 
\begin{align}
&\mathbf{V}(\mathscr{T})=\{\mathbf{v}_{\mathscr{T}}\in\mathbf{C}(\overline{\Omega}): \mathbf{v}_{\mathscr{T}}|_T\in[\mathbb{P}_2(T)]^{d} \ \forall T\in\mathscr{T}\}\cap\mathbf{H}_0^{1}(\Omega), \label{def:velocity_space}
\\
&\mathcal{P}(\mathscr{T})=\{ q_{\mathscr{T}}\in C(\overline{\Omega}) : q_{\mathscr{T}}|_T\in\mathbb{P}_1(T) \ \forall T\in\mathscr{T} \}\cap L_{0}
^{2}(\Omega).\label{def:pressure_space}
\end{align}


\subsubsection{A fully discrete scheme}\label{sec:fully_discrete_scheme}

In \EO{this section, we introduce a fully discrete scheme to approximate the solution to the optimal control problem \eqref{eq:min_weak_setting}--\eqref{eq:weak_st_equation}. The scheme utilizes the classical Taylor--Hood elements to discretize the state and adjoint equations and piecewise quadratic functions to approximate local optimal controls. To be precise,} $\bar{\mathbf{u}}_{\T} \in \mathbb{U}_{ad}(\T)$, where
\begin{equation}
\label{def:control_space}
\mathbb{U}_{ad}(\T) = \mathbb{U}(\T) \cap \mathbb{U}_{ad}, \quad \mathbb{U}(\T) = \{ \mathbf{v}_{\T} \in \mathbf{L}^{\infty}(\Omega): \mathbf{v}_{\T}|_{T} \in [\mathbb{P}_{2}(T)]^{d} \ \forall T \in \T \}. 
\end{equation}

The \EO{fully discrete scheme reads as follows: Find 
\begin{equation}\label{eq:discrete_cost_mini}
\min_{\mathbf{V}(\T)\times\mathbb{U}_{ad}(\T)} J(\mathbf{y}_{\T},\mathbf{u}_{\T})
\end{equation}
subject} to the discrete state equation
\begin{equation}\label{eq:discrete_state_equation}
\begin{aligned}
\nu(\nabla \mathbf{y}_\T, \nabla \mathbf{v}_\T)_{\mathbf{L}^2(\Omega)} 
+
b(\mathbf{y}_\T;\mathbf{y}_\T,\mathbf{v}_\T) 
- 
(p_\T,\text{div } \mathbf{v}_\T)_{L^2(\Omega)} 
&= (\mathbf{u}_\T,\mathbf{v}_\T)_{\mathbf{L}^2(\Omega)},
\\
(q_\T,\text{div } \mathbf{y}_\T)_{L^2(\Omega)} &= 0,
\end{aligned}
\end{equation}
for all $(\mathbf{v}_\T,q_\T) \in \mathbf{V}(\mathscr{T})\times \mathcal{P}(\T)$. 

\EO{In what follows we assume, in addition to \eqref{eq:smallness_assumption}, that we have that
\begin{enumerate}[label=(A.\arabic*)]
\item \label{A1} the mesh $\T$ is sufficiently refined, this is, the parameter $h_{\T}$ is sufficiently small such that the results in \cite[Theorem 4.11]{MR2338434} hold, and

\item \label{A2} the discrete variable $\mathbf{u}_\T\in\mathbb{U}_{ad}(\T)$ is sufficiently close to $\bar{\mathbf{u}}$, i.e., there exist a constant $\rho>0$ such that $\|\mathbf{u}_\T- \bar{\mathbf{u}}\|_{\mathbf{L}^2(\Omega)}\leq \rho$ \cite[Remark 4.9]{MR2338434}.
\end{enumerate} 
We thus invoke \cite[Theorem 4.8]{MR2338434} to guarantee the existence of a unique discrete pair $(\mathbf{y}_\T,p_\T) \in \mathbf{V}(\mathscr{T})\times \mathcal{P}(\T)$ solving \eqref{eq:discrete_state_equation}. In addition, we have that $(\mathbf{y}_\T,p_\T)$ lies in a suitable neighborhood of $(\bar{\mathbf{y}},\bar{p})$. An application of \cite[Theorem 4.11]{MR2338434} yields the existence of at least one solution to our fully discrete problem \eqref{eq:discrete_cost_mini}--\eqref{eq:discrete_state_equation}.}

If $\bar{ \mathbf{u}}_\mathscr{T}$ denotes a local solution, we have
\begin{equation}\label{eq:discrete_var_ineq}
 (\bar{\mathbf{z}}_\mathscr{T}+\alpha\bar{\mathbf{u}}_\mathscr{T},\mathbf{u}_\mathscr{T}-\bar{\mathbf{u}}_\mathscr{T})^{}_{\mathbf{L}^2(\Omega)}  \geq  0 \quad \forall
  \mathbf{u}_\mathscr{T} \in \mathbb{U}_{ad}(\T),
\end{equation}
where the pair $(\bar{\mathbf{z}}_\mathscr{T},\bar{r}_\T) \in \mathbf{V}(\T)\times\mathcal{P}(\T)$ solves
\begin{equation}
\label{eq:discrete_adjoint_equation}
\begin{aligned}
\nu(\nabla \mathbf{w}_\T, \nabla \bar{\mathbf{z}}_\T)_{\mathbf{L}^2(\Omega)}+b(\bar{\mathbf{y}}_\T;\mathbf{w}_\T,\bar{\mathbf{z}}_\T)
\\
+b(\mathbf{w}_\T;\bar{\mathbf{y}}_\T,\bar{\mathbf{z}}_\T)
-(\bar{r}_\T,\text{div } \mathbf{w}_\T)_{L^2(\Omega)} &=(\bar{\mathbf{y}}_\T-\mathbf{y}_\Omega,\mathbf{w}_\T)_{\mathbf{L}^2(\Omega)},
\\
(s_\T,\text{div } \bar{\mathbf{z}}_\T)_{L^2(\Omega)} &= 0,
\end{aligned}
\end{equation}
for all $(\mathbf{w}_{\mathscr{T}},s_\T) \in \mathbf{V}(\T)\times\mathcal{P}(\T)$ \cite[Lemma 4.14]{MR2338434}. Under the assumption that
\begin{equation}\label{eq:assumption_y_T_small_2}
2\|\nabla \bar{\mathbf{y}}_\T \|_{\mathbf{L}^2(\Omega)} \leq \FF{\varrho} \nu \mathcal{C}_b^{-1},
\qquad
\varrho < 1,
\end{equation}
the discrete problem \eqref{eq:discrete_adjoint_equation} admits a unique solution. In fact, define
\begin{equation}
\label{eq:bilinear_form_B_hat}
\mathcal{C}(\mathbf{w},\mathbf{v}):=\nu (\nabla \mathbf{w},\nabla\mathbf{v})_{\mathbf{L}^2(\Omega)}+b(\bar{\mathbf{y}}_{\T};\mathbf{w},\mathbf{v})+b(\mathbf{w};\bar{\mathbf{y}}_{\T}, \mathbf{v}).
\end{equation}
With \eqref{eq:assumption_y_T_small_2} at hand, similar arguments to ones used to derive \eqref{eq:coercivity_bilinear_form_B} yield the coercivity of $\mathcal{C}\FF{(\cdot,\cdot)}$ in $\mathbf{H}_0^1(\Omega) \times \mathbf{H}_0^1(\Omega)$. Since the pair $(\mathbf{V}(\T), \mathcal{P}(\T))$ satisfy a discrete inf-sup condition \cite[Lemma 4.24]{MR2050138}, an application of \cite[Theorem 2.42]{MR2050138} allows us to conclude.

\EO{
\subsubsection{A semi discrete scheme}\label{sec:semi_discrete_scheme}

In this section, we introduce a semidiscrete scheme for  \eqref{eq:min_weak_setting}--\eqref{eq:weak_st_equation} that is based on the so-called variational discretization approach \cite{MR2122182}. This approach, in contrast to the fully discrete scheme of section \ref{sec:fully_discrete_scheme}, discretizes only the state space -- \emph{the control space $\mathbb{U}_{ad}$ is not discretized} -- and induces a discretization of the optimal control variable by projecting the optimal discrete adjoint state into the admissible control set.

We notice that, in practice, AFEMs for the fully discrete scheme tend to generate DOF (degrees of freedom) in the vicinity of the border of the active set of the control variable. These DOF seem unnecessary for an accurate approximation of the control variable; see section \ref{sec:numerical_ex} for numerical evidence that support this claim. This motivates the use of the variational discretization approach within AFEMs.

The semi discrete scheme is defined as follows:
\begin{equation}\label{eq:discrete_cost_var}
\min_{\mathbf{V}(\T)\times\mathbb{U}_{ad}} J(\mathbf{y}_{\T},\mathbf{g})
\end{equation}
subject to the discrete state equation
\begin{equation}\label{eq:discrete_st_variational}
\begin{aligned}
\nu(\nabla \mathbf{y}_\T, \nabla \mathbf{v}_\T)_{\mathbf{L}^2(\Omega)} 
+
b(\mathbf{y}_\T;\mathbf{y}_\T,\mathbf{v}_\T) 
- 
(p_\T,\text{div } \mathbf{v}_\T)_{L^2(\Omega)} 
&= (\mathbf{g},\mathbf{v}_\T)_{\mathbf{L}^2(\Omega)},
\\
(q_\T,\text{div } \mathbf{y}_\T)_{L^2(\Omega)} &= 0,
\end{aligned}
\end{equation}
for all $(\mathbf{v}_\T,q_\T) \in \mathbf{V}(\mathscr{T})\times \mathcal{P}(\T)$. Assume that \eqref{eq:smallness_assumption} and \ref{A1} hold. Assume, in addition, that 
\begin{enumerate}[label=(A.3)]
\item\label{A3} the variable $\mathbf{g}\in\mathbb{U}_{ad}$ is sufficiently close to $\bar{\mathbf{u}}$, i.e., there exist a constant $\rho>0$ such that $\|\mathbf{g}- \bar{\mathbf{u}}\|_{\mathbf{L}^2(\Omega)}\leq \rho$ \cite[Remark 4.9]{MR2338434}.
\end{enumerate} 
Then, the discrete state equation \eqref{eq:discrete_st_variational} admits a unique solution which lies in a  neighborhood of $(\bar{\mathbf{y}},\bar{p})$ \cite[Theorem 4.8]{MR2338434}. In addition, we have that the semidiscrete optimal control problem \eqref{eq:discrete_cost_var}--\eqref{eq:discrete_st_variational} admits at least one solution \cite[Theorem 4.11]{MR2338434}. 

If $\bar{\mathbf{g}}$ denotes a local solution for \eqref{eq:discrete_cost_var}--\eqref{eq:discrete_st_variational}, we have
\begin{equation}\label{eq:discrete_var_ineq_variational}
(\bar{\mathbf{z}}_\mathscr{T}+\alpha\bar{\mathbf{g}},\mathbf{u}-\bar{\mathbf{g}})^{}_{\mathbf{L}^2(\Omega)}  \geq  0 \quad \forall  \mathbf{u} \in \mathbb{U}_{ad},
\end{equation}
where the pair $(\bar{\mathbf{z}}_\mathscr{T},\bar{r}_\T) \in \mathbf{V}(\T)\times\mathcal{P}(\T)$ solves 
\begin{equation}
\label{eq:discrete_adjoint_equation_variational}
\begin{aligned}
\nu(\nabla \mathbf{w}_\T, \nabla \bar{\mathbf{z}}_\T)_{\mathbf{L}^2(\Omega)}+b(\bar{\mathbf{y}}_\T;\mathbf{w}_\T,\bar{\mathbf{z}}_\T)
\\
+b(\mathbf{w}_\T;\bar{\mathbf{y}}_\T,\bar{\mathbf{z}}_\T)
-(\bar{r}_\T,\text{div } \mathbf{w}_\T)_{L^2(\Omega)} &=(\bar{\mathbf{y}}_\T-\mathbf{y}_\Omega,\mathbf{w}_\T)_{\mathbf{L}^2(\Omega)},
\\
(s_\T,\text{div } \bar{\mathbf{z}}_\T)_{L^2(\Omega)} &= 0,
\end{aligned}
\end{equation}
for all $(\mathbf{w}_{\mathscr{T}},s_\T) \in \mathbf{V}(\T)\times\mathcal{P}(\T)$. Here, $\bar{\mathbf{y}}_{\T}=\bar{\mathbf{y}}_\T(\bar{\mathbf{g}})$ solves \eqref{eq:discrete_st_variational} with $\mathbf{g}=\bar{\mathbf{g}}$.}


\section{\EO{A posteriori error analysis: the fully discrete scheme}}\label{sec:a_posteriori}

In this section, we propose and analyze an a posteriori error estimator \EO{for the fully discrete optimal control problem \eqref{eq:discrete_cost_mini}--\eqref{eq:discrete_state_equation}.} This estimator can be decomposed as the sum of three contributions which are related to the discretization of the state \EO{and} adjoint equations and the control set. To obtain a reliability estimate, i.e., an upper bound for the error in terms of the devised a posteriori error estimator, we invoke upper bounds on the error between the solution to the discretization \eqref{eq:discrete_state_equation}--\eqref{eq:discrete_adjoint_equation} and auxiliary variables that we define in the following sections.

In order to guarantee the existence of a local solution $(\bar{\mathbf{y}}_\T,\bar{p}_\T,\bar{\mathbf{z}}_\T,\bar{r}_\T,\bar{\mathbf{u}}_\T)\in \mathbf{V}(\T)\times \mathcal{P}(\T)\times \mathbf{V}(\T)\times \mathcal{P}(\T)\times \mathbb{U}_{ad}(\T)$ to the \EO{fully discrete} optimal control problem, satisfying the discrete system \eqref{eq:discrete_state_equation}--\eqref{eq:discrete_adjoint_equation}, we shall assume, throughout the following sections, \FF{assumptions \eqref{eq:smallness_assumption}, \ref{A1} and \ref{A2}.}


\subsection{A posteriori error analysis for the state equations}\label{sec:a_posteriori_navier_stokes}
We present, inspired \EO{by} reference \cite{MR1259620} (see also \cite[section 9.3]{MR1885308}), a posteriori error estimates for \EO{a suitable discretization of the} stationary Navier--Stokes equations \eqref{eq:state_equations}.

We begin the discussion by introducing the following auxiliary variables. Let $(\hat{\mathbf{y}},\hat{p}) \in \mathbf{H}_0^1(\Omega)\times L_0^2(\Omega)$ be the solution to
\begin{equation}\label{eq:hat_function_st}
\begin{aligned}
\nu(\nabla \hat{\mathbf{y}}, \nabla \mathbf{v})_{\mathbf{L}^2(\Omega)}+b(\hat{\mathbf{y}};\hat{\mathbf{y}},\mathbf{v})-(\hat{p},\text{div } \mathbf{v})_{L^2(\Omega)} & = (\bar{\mathbf{u}}_\T,\mathbf{v})_{\mathbf{L}^2(\Omega)}  & \forall \mathbf{v} \in \mathbf{H}_0^1(\Omega),\\
(q,\text{div } \hat{\mathbf{y}})_{L^2(\Omega)} & =  0  & \forall q\in L^2_0(\Omega).
\end{aligned}
\end{equation}
Theorem \ref{thm:well_posedness_navier_stokes} guarantees the existence of a unique pair 
$(\hat{\mathbf{y}},\hat{p})$ solving problem \eqref{eq:hat_function_st} with
\begin{equation}\label{eq:stab_hat_y}
\|\nabla \hat{\mathbf{y}}\|_{\mathbf{L}^2(\Omega)}\leq \theta\mathcal{C}_{b}^{-1}\nu, \qquad \theta < 1.
\end{equation}
Notice that the pair $(\bar{\mathbf{y}}_{\T},\bar{p}_{\T})$, which solves \eqref{eq:discrete_state_equation} with $\mathbf{u}_{\T}$ replaced by $\bar{\mathbf{u}}_{\T}$, can be seen as the finite element approximation, within the space $\mathbf{V}(\mathscr{T}) \times \mathcal{P}(\T)$, of $(\hat{\mathbf{y}},\hat{p})$. This observation motivates us to define the following a posteriori error estimator:
\begin{gather}\label{eq:error_estimator_st}
\mathcal{E}_{st}^{2}:= \sum_{T\in\T}\mathcal{E}_{st,T}^2, \quad
\mathcal{E}_{st,T}^{2}:=h_T^2\|\bar{\mathbf{u}}_\T+\nu\Delta\bar{\mathbf{y}}_\T-(\bar{\mathbf{y}}_\T\cdot \nabla)\bar{\mathbf{y}}_\T-\nabla \bar{p}_\T\|_{\mathbf{L}^2(T)}^2 \\\nonumber
+ \|\text{div }\bar{\mathbf{y}}_\T\|_{L^2(T)}^2
+ h_T\|\llbracket (\nu\nabla \bar{\mathbf{y}}_\T-\bar{p}_\T\mathbb{I}_{d} )\cdot \mathbf{n}\rrbracket\|_{\mathbf{L}^2(\partial T \setminus \partial \Omega)}^2.
\end{gather}

We present the following global reliability result.

\begin{theorem}[global reliability of $\mathcal{E}_{st}$]\label{thm:navier_stokes_estimator}
Assume that \eqref{eq:smallness_assumption} holds. Let $(\hat{\mathbf{y}},\hat{p})\in \mathbf{H}_0^1(\Omega)\times L_0^2(\Omega)$ be the unique solution to \eqref{eq:hat_function_st}.
Let $(\bar{\mathbf{y}}_{\T},\bar{p}_\T)\in \mathbf{V}(\T)\times \mathcal{P}(\T)$ be the solution to \eqref{eq:discrete_state_equation} with $\mathbf{u}_{\T}$ replaced by $\bar{\mathbf{u}}_{\T}$.
Assume that the estimate
\begin{equation}\label{eq:assumption_y_y_T}
\|\nabla \bar{\mathbf{y}}_\T \|_{\mathbf{L}^2(\Omega)}< \nu \mathcal{C}_b^{-1}
\end{equation}
holds. Then, we have that
\begin{equation}\label{eq:estimate_state_hat_discrete}
\|\nabla (\hat{\mathbf{y}}-\bar{\mathbf{y}}_\T)\|_{\mathbf{L}^2(\Omega)}^2+\|\hat{p}-\bar{p}_\T\|_{L^2(\Omega)}^2 
\lesssim
\mathcal{E}_{st}^2,
\end{equation}
with a hidden constant that is independent of $(\hat{\mathbf{y}},\hat{p})$ and $(\bar{\mathbf{y}}_{\T},\bar{p}_{\T})$, the size of the elements in the mesh $\T$, and $\#\T$.
\end{theorem}
\begin{proof} 
To perform a reliability analysis for the a posteriori error estimator \eqref{eq:error_estimator_st}, we introduce a Ritz projection $(\boldsymbol\varphi,\psi)$ of the residuals \cite{MR1445736}. The pair $(\boldsymbol\varphi,\psi)$ is defined as the solution to the following problem: Find $(\boldsymbol\varphi,\psi)\in\mathbf{H}_0^1(\Omega)\times L_0^2(\Omega)$ such that 
\begin{equation}
\begin{aligned}
(\nabla \boldsymbol \varphi,\nabla \mathbf{v})_{\mathbf{L}^2(\Omega)} 
& = 
\nu(\nabla\hat{\mathbf{e}}_{\mathbf{y}},\nabla\mathbf{v} )_{\mathbf{L}^2(\Omega)} \!-\! (\hat{e}_p,\text{div } \mathbf{v})_{L^2(\Omega)} \!+\! b(\hat{\mathbf{y}};\hat{\mathbf{e}}_{\mathbf{y}},\mathbf{v})\!
+\! b(\hat{\mathbf{e}}_{\mathbf{y}};\bar{\mathbf{y}}_{\T},\mathbf{v}),
\\
(\psi,q)_{L^2(\Omega)} & = (q,\text{div }\hat{\mathbf{e}}_\mathbf{y})_{L^2(\Omega)},
\end{aligned}\hspace{-0.2cm}
\label{eq:ritz_projection_navier}
\end{equation}
for all \EO{$(\mathbf{v},q) \in \mathbf{H}_0^1(\Omega) \times L_0^{2}(\Omega)$.} To shorten notation, we have introduced $(\hat{\mathbf{e}}_{\mathbf{y}},\hat{e}_p):=(\hat{\mathbf{y}}-\bar{\mathbf{y}}_\T,\hat{p}-\bar{p}_\T)$. The existence and uniqueness of the pair $(\boldsymbol\varphi,\psi)$ follows from \EO{applying} the Lax--Milgram Lemma. On the other hand, under assumption \eqref{eq:assumption_y_y_T}, similar arguments to the ones developed in \cite[Theorem 4]{MR1259620} (see also \cite[section 9.3]{MR1885308}) yield
\begin{equation}\label{eq:ineq_aposteriori_navier}
\|\nabla \hat{\mathbf{e}}_{\mathbf{y}}\|_{\mathbf{L}^2(\Omega)}^2+\|\hat{e}_p\|_{L^2(\Omega)}^2
\lesssim \|\nabla \boldsymbol\varphi\|_{\mathbf{L}^2(\Omega)}^2+\|\psi\|_{L^2(\Omega)}^2.
\end{equation}

The rest of the proof is dedicated to bound the terms on the right-hand side of \eqref{eq:ineq_aposteriori_navier}. Let $\mathbf{v} \in \mathbf{H}_0^1(\Omega)$ and set $q=0$ in \eqref{eq:ritz_projection_navier}. This, combined with \eqref{eq:hat_function_st}, yields
\[
(\nabla\boldsymbol\varphi,\nabla \mathbf{v})_{\mathbf{L}^2(\Omega)}=(\bar{\mathbf{u}}_\T,\mathbf{v})_{\mathbf{L}^2(\Omega)}-\nu(\nabla \bar{\mathbf{y}}_\T, \nabla \mathbf{v})_{\mathbf{L}^2(\Omega)}
-b(\bar{\mathbf{y}}_\T;\bar{\mathbf{y}}_\T,\mathbf{v})+(\bar{p}_\T,\text{div } \mathbf{v})_{L^2(\Omega)}.
\]
Denote by $\mathcal{I}_\T:\mathbf{L}^1(\Omega)\rightarrow \mathbf{V}(\T)$ the Cl\'ement interpolation operator \cite{MR2373954,MR0520174}. Invoke the previous relation, the discrete problem \eqref{eq:discrete_state_equation} with $\mathbf{v}_{\T} = \mathcal{I}_{\T} \mathbf{v}$, an elementwise integration by parts formula, standard approximation properties for $\mathcal{I}_\T$, and the finite overlapping property of stars, to conclude that
\begin{multline*}
(\nabla\boldsymbol\varphi,\nabla \mathbf{v})_{\mathbf{L}^2(\Omega)}
\lesssim 
\bigg(\sum_{T\in\T}h_T^2\|\bar{\mathbf{u}}_\T+\nu\Delta\bar{\mathbf{y}}_\T-(\bar{\mathbf{y}}_\T\cdot \nabla)\bar{\mathbf{y}}_\T-\nabla \bar{p}_\T\|_{\mathbf{L}^2(T)}^2\\
+h_T\|\llbracket (\nu\nabla \bar{\mathbf{y}}_\T-\bar{p}_\T\mathbb{I}_{d} )\cdot \mathbf{n}\rrbracket\|_{\mathbf{L}^2(\partial T\setminus\partial\Omega)}^2\bigg)^{\frac{1}{2}}\|\nabla\mathbf{v}\|_{\mathbf{L}^2(\Omega)}.
\end{multline*}
Set $\mathbf{v}=\boldsymbol\varphi$. This yields the estimate $\|\nabla\boldsymbol\varphi\|_{\mathbf{L}^2(\Omega)}\lesssim \mathcal{E}_{st}$.

Now, let $q \in L_0^2(\Omega)$ and set $\mathbf{v}=\boldsymbol 0$ in \eqref{eq:ritz_projection_navier}. The Cauchy--Schwarz inequality yields
\begin{equation*}
(\psi,q)_{L^2(\Omega)}\leq\left(\sum_{T\in\T}\|\text{div }\bar{\mathbf{y}}_\T\|_{L^2(T)}^2\right)^{\frac{1}{2}}\|q\|_{L^2(\Omega)}.
\end{equation*}
Consequently, $\|\psi\|_{L^2(\Omega)}\leq \mathcal{E}_{st}$. 

A collection of the previous estimates \FF{yields} $\|\nabla\boldsymbol\varphi\|_{\mathbf{L}^2(\Omega)}+\|\psi\|_{L^2(\Omega)} \lesssim \mathcal{E}_{st}$. We thus invoke  \eqref{eq:ineq_aposteriori_navier} to arrive at the desired estimate \eqref{eq:estimate_state_hat_discrete}. This concludes the proof.
\end{proof}


\subsection{A posteriori error analysis for the adjoint equations}\label{sec:a_posteriori_adjoint_equations}

In this section, we introduce an auxiliary problem related to the adjoint equations, devise an a posteriori error estimator for such a problem, and obtain a global reliability result. To the best of our knowledge, these results are not available in the literature.

Let $(\hat{\mathbf{z}},\hat{r}) \in \mathbf{H}_0^1(\Omega)\times L_0^2(\Omega)$ be the solution to
\begin{equation}
\label{eq:hat_function_ad}
\hspace{-0.3cm}
\begin{array}{rl}
\nu(\nabla \mathbf{w}, \nabla \hat{\mathbf{z}})_{\mathbf{L}^2(\Omega)}\!+\!b(\bar{\mathbf{y}}_\T;\mathbf{w},\hat{\mathbf{z}})\!+\!b(\mathbf{w};\bar{\mathbf{y}}_\T,\hat{\mathbf{z}})\!-\!(\hat{r},\text{div } \mathbf{w})_{L^2(\Omega)} \hspace{-0.35cm} &=\! (\bar{\mathbf{y}}_\T\!-\!\mathbf{y}_\Omega,\mathbf{w})_{\mathbf{L}^2(\Omega)}, \\
(s,\text{div } \hat{\mathbf{z}})_{L^2(\Omega)}  \hspace{-0.35cm} &=\!  0, 
\end{array}\hspace{-0.57cm}
\end{equation}
for all \EO{$(\mathbf{w},s) \in \mathbf{H}_0^1(\Omega) \times  L_0^2(\Omega)$}. Under assumption \eqref{eq:assumption_y_T_small_2} problem \eqref{eq:hat_function_ad} is well--posed. On the other hand, notice that $(\bar{\mathbf{z}}_\T,\bar{r}_\T)$, the solution to \eqref{eq:discrete_adjoint_equation}, can be seen as the finite element approximation, within the space $\mathbf{V}(\T)\times\mathcal{P}(\T)$, of $(\hat{\mathbf{z}},\hat{r})$. In view of this fact, we define, for $T\in\T$, the local error indicators
\begin{multline}\label{eq:local_error_indicator_ad}
\mathcal{E}_{ad,T}^2:=h_T^2\|\bar{\mathbf{y}}_\T-\mathbf{y}_\Omega+\nu\Delta\bar{\mathbf{z}}_\T-(\nabla \bar{\mathbf{y}}_\T)^{\intercal}\bar{\mathbf{z}}_\T+(\bar{\mathbf{y}}_\T\cdot \nabla)\bar{\mathbf{z}}_\T-\nabla \bar{r}_\T\|_{\mathbf{L}^2(T)}^2\\
+h_T\|\llbracket (\nu\nabla \bar{\mathbf{z}}_\T-\bar{r}_\T\mathbb{I}_{d} )\cdot \mathbf{n}\rrbracket\|_{\mathbf{L}^2(\partial T \setminus \partial \Omega)}^2+\|\text{div }\bar{\mathbf{z}}_\T\|_{L^2(T)}^2,
\end{multline}
and the a posteriori error estimator
\begin{equation}\label{eq:error_estimator_adjoint_eq}
\mathcal{E}_{ad}^2:=\sum_{T\in\T}\mathcal{E}_{ad,T}^2.
\end{equation}

The following result yields an upper bound for the error $\|\nabla(\hat{\mathbf{z}}-\bar{\mathbf{z}}_\T)\|_{\mathbf{L}^2(\Omega)} + \|\hat{r}-\bar{r}_\T\|_{L^2(\Omega)}$ in terms of the computable quantity $\mathcal{E}_{ad}$.

\begin{theorem}[global reliability of $\mathcal{E}_{ad}$]
\label{thm:estimate_aux_variables}
Let $(\hat{\mathbf{y}},\hat{p})$ and $(\bar{\mathbf{y}}_{\T},\bar{p}_{\T})$ be as in the statement of Theorem \ref{thm:navier_stokes_estimator}. Let $(\hat{\mathbf{z}},\hat{r})\in \mathbf{H}_0^1(\Omega)\times L_0^2(\Omega)$ and $(\bar{\mathbf{z}}_{\T},\bar{r}_\T)\in \mathbf{V}(\T)\times \mathcal{P}(\T)$ be the solutions to \eqref{eq:hat_function_ad} and \eqref{eq:discrete_adjoint_equation}, respectively.
Assume that the estimate \eqref{eq:assumption_y_T_small_2} holds. Then, we have \EO{the a posteriori error estimate}
\begin{equation}\label{eq:adjoint_hat_estimate}
\|\nabla(\hat{\mathbf{z}}-\bar{\mathbf{z}}_\T)\|_{\mathbf{L}^2(\Omega)}^2+\|\hat{r}-\bar{r}_\T\|_{L^2(\Omega)}^2\lesssim \mathcal{E}_{ad}^2,
\end{equation}
with a hidden constant that is independent of $(\hat{\mathbf{z}},\hat{r})$, $(\bar{\mathbf{z}}_{\T},\bar{r}_{\T})$, the size of the elements in the mesh $\T$, and $\#\T$.
\end{theorem}
\begin{proof} 
We proceed on the basis of four steps.

\emph{Step 1.} To simplify the presentation of the material, we define the pair $(\hat{\mathbf{e}}_\mathbf{z},\hat{e}_r):=(\hat{\mathbf{z}}-\bar{\mathbf{z}}_\T,\hat{r}-\bar{r}_\T)$. \EO{Define} the Ritz projection $(\boldsymbol\eta,\omega)$ of the residuals associated to the discretization \eqref{eq:discrete_adjoint_equation} of \eqref{eq:hat_function_ad} as the solution to the following problem: Find $(\boldsymbol\eta,\omega)\in\mathbf{H}_0^1(\Omega)\times L_0^2(\Omega)$ such that
\begin{equation}
\label{eq:ritz_projection}
\begin{aligned}
(\nabla \boldsymbol \eta,\nabla \mathbf{w})_{\mathbf{L}^2(\Omega)} & =
\mathcal{C}(\mathbf{w},\hat{\mathbf{e}}_{\mathbf{z}})-(\hat{e}_r,\text{div } \mathbf{w})_{L^2(\Omega)} \quad &\forall \mathbf{w}\in\mathbf{H}_0^1(\Omega), \\
(\omega,s)_{L^2(\Omega)} & = (s,\text{div }\hat{\mathbf{e}}_\mathbf{z})_{L^2(\Omega)}\quad &\forall s\in L_0^2(\Omega),
\end{aligned}
\end{equation} 
where $\mathcal{C}\FF{(\cdot,\cdot)}$ is defined as in \eqref{eq:bilinear_form_B_hat}. The Lax--Milgram Lemma immediately yields the existence and uniqueness of $(\boldsymbol\eta,\omega)\in\mathbf{H}_0^1(\Omega)\times L_0^2(\Omega)$ solving \eqref{eq:ritz_projection}.

The rest of the proof is dedicated to obtain the estimates
\begin{equation}
\label{eq:ineq_aposteriori_ad}
\|\nabla \hat{\mathbf{e}}_\mathbf{z}\|_{\mathbf{L}^2(\Omega)}^2+\|\hat{e}_r\|_{L^2(\Omega)}^2
\lesssim \|\nabla \boldsymbol\eta\|_{\mathbf{L}^2(\Omega)}^2+\|\omega\|_{L^2(\Omega)}^2
\lesssim \mathcal{E}_{ad}^2.
\end{equation}

\emph{Step 2.} The goal of this step is to prove the first estimate in \eqref{eq:ineq_aposteriori_ad}. To accomplish this task, we first observe that the pair $(\hat{\mathbf{e}}_{\mathbf{z}},\hat{e}_r)$ satisfies the identities
\begin{equation*}
\begin{aligned}
\mathcal{C}(\mathbf{w},\hat{\mathbf{e}}_{\mathbf{z}})-(\hat{e}_r,\text{div } \mathbf{w})_{L^2(\Omega)}
&=  
(\nabla \boldsymbol\eta, \nabla \mathbf{w})_{\mathbf{L}^2(\Omega)} \quad & \forall  \mathbf{w}\in\mathbf{H}_0^1(\Omega), \\
(s,\text{div }\hat{\mathbf{e}}_{\mathbf{z}})_{L^2(\Omega)} &=  (\omega,s)_{L^2(\Omega)} \quad & \forall s\in L_0^2(\Omega).
\end{aligned}
\end{equation*}
In view of the fact that $\bar{\mathbf{y}}_\T\in\mathbf{V}(\T)$ satisfies assumption \eqref{eq:assumption_y_T_small_2}, \EO{similar arguments to the ones that lead to} \eqref{eq:coercivity_bilinear_form_B} yield that $\mathcal{C}\FF{(\cdot,\cdot)}$ is coercive in $\mathbf{H}_0^1(\Omega) \times \mathbf{H}_0^1(\Omega)$. \EO{Apply} the inf--sup theory for saddle point problems given by Brezzi in \cite{MR365287} to conclude the stability estimate
\begin{equation}
\label{eq:errors_estimators}
\|\nabla \hat{\mathbf{e}}_\mathbf{z}\|_{\mathbf{L}^2(\Omega)}^2+\|\hat{e}_r\|_{L^2(\Omega)}^2
\lesssim
\|\nabla\boldsymbol\eta\|_{\mathbf{L}^2(\Omega)}^2+\|\omega\|_{L^2(\Omega)}^2,
\end{equation}
with a hidden constant that depends on $\nu$.

\emph{Step 3.}  In this step we obtain the second estimate in \eqref{eq:ineq_aposteriori_ad}. To accomplish this task, we invoke problems \eqref{eq:ritz_projection} and \eqref{eq:hat_function_ad} to arrive at
\begin{equation}
\label{eq:residual_eq_I}
\begin{aligned}
(\nabla\boldsymbol\eta,\nabla \mathbf{w})_{\mathbf{L}^2(\Omega)} & =(\bar{\mathbf{y}}_\T-\mathbf{y}_\Omega,\mathbf{w})_{\mathbf{L}^2(\Omega)}-\nu(\nabla \mathbf{w}, \nabla \bar{\mathbf{z}}_\T)_{\mathbf{L}^2(\Omega)}
-b(\bar{\mathbf{y}}_\T;\mathbf{w},\bar{\mathbf{z}}_\T)
\\
& \quad -b(\mathbf{w};\bar{\mathbf{y}}_\T,\bar{\mathbf{z}}_\T)
+(\bar{r}_\T,\text{div } \mathbf{w})_{L^2(\Omega)},
\\
(\omega,s)_{L^2(\Omega)} &= (s,\text{div } \hat{\mathbf{e}}_{\mathbf{z}})_{L^2(\Omega)},
\end{aligned}
\end{equation}
for all \EO{$(\mathbf{w},s) \in\mathbf{H}_0^1(\Omega) \times L_0^2(\Omega)$.} Let $\mathbf{w} \in \mathbf{H}_0^1(\Omega)$ and set $s=0$ in \eqref{eq:residual_eq_I}. Invoke the discrete problem \eqref{eq:discrete_adjoint_equation} with $\mathbf{w}_\T = \mathcal{I}_\T\mathbf{w}$, an elementwise integration by parts formula, standard approximation properties for $\mathcal{I}_\T$, and the
finite overlapping property of stars, to conclude that
\begin{multline*}
(\nabla\boldsymbol\eta,\nabla \mathbf{w})_{\mathbf{L}^2(\Omega)}
\!\lesssim \!
\bigg(\!\sum_{T\in\T}h_T^2\|\bar{\mathbf{y}}_\T-\mathbf{y}_\Omega+\nu\Delta\bar{\mathbf{z}}_\T-(\nabla \bar{\mathbf{y}}_\T)^{\intercal}\bar{\mathbf{z}}_\T +(\bar{\mathbf{y}}_\T\cdot \nabla)\bar{\mathbf{z}}_\T-\nabla \bar{r}_\T\|_{L^2(T)}^2\\
+h_T\|\llbracket (\nu\nabla \bar{\mathbf{z}}_\T-\bar{r}_\T\mathbb{I}_{d} )\cdot \mathbf{n}\rrbracket\|_{L^2(\partial T\setminus\partial\Omega)}^2\bigg)^{\frac{1}{2}}\|\nabla\mathbf{w}\|_{\mathbf{L}^2(\Omega)}.
\end{multline*}
Set $\mathbf{w}=\boldsymbol\eta$. This yields the estimate $\|\nabla\boldsymbol\eta\|_{\mathbf{L}^2(\Omega)}\lesssim \mathcal{E}_{ad}$.

Now, let $s\in L_0^2(\Omega)$ and set $\mathbf{w}=\boldsymbol 0$ in \eqref{eq:residual_eq_I}. The Cauchy--Schwarz inequality, in view of the fact that $\hat{\mathbf{z}}\in \mathbf{V}(\Omega)$, yields
\begin{equation*}
(\omega,s)_{L^2(\Omega)}\leq\left(\sum_{T\in\T}\|\text{div }\bar{\mathbf{z}}_\T\|_{L^2(T)}^2\right)^{\frac{1}{2}}\|s\|_{L^2(\Omega)},
\end{equation*}
which implies that $\|\omega\|_{L^2(\Omega)}\leq \mathcal{E}_{ad}$.

A collection of the previous estimates \FF{yields}
$\|\nabla\boldsymbol\eta\|_{\mathbf{L}^2(\Omega)} +\|\omega\|_{L^2(\Omega)} \lesssim \mathcal{E}_{ad}$.

\emph{Step 4.} Apply \eqref{eq:errors_estimators} and the bounds obtained in step 3 for $\|\nabla\boldsymbol\eta\|_{\mathbf{L}^2(\Omega)}$ and $\|\omega\|_{L^2(\Omega)}$ to arrive at the desired estimate \eqref{eq:adjoint_hat_estimate}.

\end{proof}


\subsection{\EO{Reliability analysis: the fully discrete scheme}}
\label{sec:a_posteriori_global}

In this section, we design an a posteriori error estimator for \EO{the fully discrete scheme} and provide a reliability analysis. The error estimator \EO{is} decomposed as the sum of three contributions: two contributions related to the discretization of the state and adjoint equations, $\mathcal{E}_{st}$ and $\mathcal{E}_{ad}$, respectively (which have been already introduced in sections \ref{sec:a_posteriori_navier_stokes} and \ref{sec:a_posteriori_adjoint_equations}) and a contribution associated to the discretization of the optimal control variable.
\EO{To present the latter, we define the auxiliary variable}
\begin{equation}\label{def:tilde_u}
\tilde{\mathbf{u}}:=\Pi_{[\mathbf{a},\mathbf{b}]}\left(-\alpha^{-1}\bar{\mathbf{z}}_\T\right).
\end{equation}
\EO{We immediately comment that, in what follows, we assume that \eqref{eq:assumption_y_T_small_2} holds. Consequently, there exists a unique pair $(\FF{\bar{\mathbf{z}}_\T,\bar{r}_{\T}})$ solving \eqref{eq:discrete_adjoint_equation}. This implies that $\tilde{\mathbf{u}}$ is uniquely determined.} A key property in favor of the definition of $\tilde{\mathbf{u}} \in\mathbb{U}_{ad}$ is that $\tilde{\mathbf{u}}$ satisfies the inequality
\begin{equation}\label{eq:var_ineq_tilde_u}
(\bar{\mathbf{z}}_\T+\alpha\tilde{\mathbf{u}}, \mathbf{u}-\tilde{\mathbf{u}})_{\mathbf{L}^2(\Omega)}\geq 0 \quad \forall \mathbf{u}\in \mathbb{U}_{ad}.
\end{equation}
\EO{We refer the reader to \cite[Lemma 2.26]{Troltzsch} for a proof of this result.}

With the variable $\tilde{\mathbf{u}}$ at hand, we define the following error estimator and local error indicators associated to the discretization of the optimal control variable:
\begin{equation}\label{eq:error_estimator_control_var}
\mathcal{E}_{ct}^2:=\sum_{T\in\T}\mathcal{E}_{ct,T}^2,\qquad \mathcal{E}_{ct,T}:=\|\tilde{\mathbf{u}}-\bar{\mathbf{u}}_\T\|_{\mathbf{L}^2(T)}.
\end{equation}

The next result is instrumental for our a posteriori error analysis.

\begin{theorem}[auxiliary control estimate]\label{thm:error_bound_control_tilde}
Assume that the smallness assumption \eqref{eq:smallness_assumption} holds. Let $(\bar{\mathbf{y}},\bar{p},\bar{\mathbf{z}},\bar{r},\bar{\mathbf{u}}) \in \mathbf{H}_0^1(\Omega)\times L_0^2(\Omega) \times \mathbf{H}_0^1(\Omega)\times L_0^2(\Omega) \times \mathbb{U}_{ad}$ be a local solution of \eqref{eq:min_weak_setting}--\eqref{eq:weak_st_equation} that satisfies the sufficient second order optimality condition \eqref{eq:equivalent_second_I}, or equivalently \eqref{eq:equivalent_second_II}. Let $\mathsf{M} \EO{> 0}$ be such that 
$\max \{ \| \bar{\mathbf{u}}+\theta_{\T}(\tilde{\mathbf{u}}-\bar{\mathbf{u}})\|_{\mathbf{L}^{\infty}(\Omega)}, \|\tilde{\mathbf{u}} - \bar{\mathbf{u}} \|_{\mathbf{L}^{\infty}(\Omega)}\} \leq \mathsf{M}$ with $\theta_{\T} \in (0,1)$. \EO{If $\T$ is a mesh such that
\begin{equation}\label{eq:assumption_mesh}
\tilde{\mathbf{u}}-\bar{\mathbf{u}}\in \mathbf{C}_{\bar{\mathbf{u}}}^\tau, \qquad 
\|\bar{\mathbf{z}}-\bar{\mathbf{z}}_\T\|_{\mathbf{L}^2(\Omega)} \leq \alpha\mu(2C_{\mathsf{M}})^{-1},
\end{equation}
then}
\begin{equation}\label{eq:inequality_control_tilde}
\frac{\mu}{2}\|\bar{\mathbf{u}}-\tilde{\mathbf{u}}\|_{\mathbf{L}^2(\Omega)}^2
\leq
(j'(\tilde{\mathbf{u}})-j'(\bar{\mathbf{u}}))(\tilde{\mathbf{u}}-\bar{\mathbf{u}}).
\end{equation}
The constant $C_{\mathsf{M}}$ is given by \eqref{eq:estimate_second_der} while the auxiliary variable $\tilde{\mathbf{u}}$ is defined in \eqref{def:tilde_u}.
\end{theorem}
\begin{proof}
Since $\tilde{\mathbf{u}}-\bar{\mathbf{u}}\in \FF{\mathbf{C}}_{\bar{\mathbf{u}}}^\tau$, with $\FF{\mathbf{C}}_{\bar{\mathbf{u}}}^\tau$ defined in \eqref{def:cone_tau}, and $\bar{\mathbf{u}}$ satisfies the second order optimality condition \eqref{eq:equivalent_second_II}, we are \EO{allowed} to set $\mathbf{v}=\tilde{\mathbf{u}}-\bar{\mathbf{u}}$ in \eqref{eq:equivalent_second_II} \EO{to obtain}
\begin{equation}\label{eq:diff_1}
\mu\|\tilde{\mathbf{u}}-\bar{\mathbf{u}}\|_{\mathbf{L}^2(\Omega)}^2 \leq j''(\bar{\mathbf{u}})(\tilde{\mathbf{u}}-\bar{\mathbf{u}})^2.
\end{equation}
On the other hand, in view of the mean value theorem, we obtain, for some $\theta_{\T} \in (0,1)$, 
\begin{equation*}
\label{eq:mean_value_identity}
(j'(\tilde{\mathbf{u}})-j'(\bar{\mathbf{u}}))(\tilde{\mathbf{u}}-\bar{\mathbf{u}})=j''(\zeta)(\tilde{\mathbf{u}}-\bar{\mathbf{u}})^2,
\end{equation*}
with $\zeta=\bar{\mathbf{u}}+\theta_{\T}(\tilde{\mathbf{u}}-\bar{\mathbf{u}})$. Thus, in view of \eqref{eq:diff_1}, we arrive at
\begin{align}\label{eq:ineq_u_tilde_bar}
\mu\|\tilde{\mathbf{u}}-\bar{\mathbf{u}}\|_{L^2(\Omega)}^2 
& \leq (j'(\tilde{\mathbf{u}})-j'(\bar{\mathbf{u}}))(\tilde{\mathbf{u}}-\bar{\mathbf{u}}) + (j''(\bar{\mathbf{u}})-j''(\zeta))(\tilde{\mathbf{u}}-\bar{\mathbf{u}})^2.
\end{align}
Since $\mathsf{M} > 0$ is such that 
$\max \{ \| \bar{\mathbf{u}}+\theta_{\T}(\tilde{\mathbf{u}}-\bar{\mathbf{u}})\|_{\mathbf{L}^{\infty}(\Omega)}, \|\tilde{\mathbf{u}} - \bar{\mathbf{u}} \|_{\mathbf{L}^{\infty}(\Omega)}\} \leq \mathsf{M}$  and $j$ is of class ${C}^2$ in $\mathbf{L}^2(\Omega)$ \cite[Theorem 3.3]{MR2338434}, we can thus apply \eqref{eq:estimate_second_der} to derive
\begin{equation*}
(j''(\bar{\mathbf{u}})-j''(\zeta))(\tilde{\mathbf{u}}-\bar{\mathbf{u}})^2 
\leq C_{\mathsf{M}} \|\tilde{\mathbf{u}}-\bar{\mathbf{u}}\|_{\mathbf{L}^\FF{2}(\Omega)}\|\tilde{\mathbf{u}}-\bar{\mathbf{u}}\|_{\mathbf{L}^2(\Omega)}^2,
\end{equation*}
where we have also used that $\theta_{\T} \in (0,1)$. Invoke \eqref{eq:projection_formula} and \eqref{def:tilde_u}, the Lipschitz property of the projection operator $\Pi_{[\texttt{a},\texttt{b}]}$, and assumption \eqref{eq:assumption_mesh}, to arrive at
\begin{equation*}
(j''(\bar{\mathbf{u}})-j''(\zeta))(\tilde{\mathbf{u}}-\bar{\mathbf{u}})^2 
\leq C_{\mathsf{M}}{\alpha}^{-1} 
\|\bar{\mathbf{z}} - \bar{\mathbf{z}}_\T \|_{\mathbf{L}^\FF{2}(\Omega)} 
\|\tilde{\mathbf{u}}-\bar{\mathbf{u}}\|_{\mathbf{L}^2(\Omega)}^2
\leq \frac{\mu}{2}\|\tilde{\mathbf{u}}-\bar{\mathbf{u}}\|_{\mathbf{L}^2(\Omega)}^2.
\end{equation*}
Replacing this inequality into \eqref{eq:ineq_u_tilde_bar} allows us to conclude the desired inequality \eqref{eq:inequality_control_tilde}. This concludes the proof.
\end{proof}
\FF{
\begin{remark}[a sufficient condition for $\tilde{\mathbf{u}}-\bar{\mathbf{u}}\in \mathbf{C}_{\bar{\mathbf{u}}}^\tau$] In what follows, we show that
\[
\|\bar{\mathbf{z}}-\bar{\mathbf{z}}_{\T}\|_{\mathbf{L}^\infty(\Omega)}
\leq \tau/2 \implies \tilde{\mathbf{u}}-\bar{\mathbf{u}}\in \mathbf{C}_{\bar{\mathbf{u}}}^\tau.
\]
In fact, since $\tilde{\mathbf{u}}\in \mathbb{U}_{ad}$, we can immediately conclude that $\mathbf{v} =\tilde{\mathbf{u}}-\bar{\mathbf{u}}\geq \boldsymbol 0$ if $\bar{\mathbf{u}}=\mathbf{a}$ and that $\mathbf{v} = \tilde{\mathbf{u}}-\bar{\mathbf{u}}\leq \boldsymbol 0$ if $\bar{\mathbf{u}}=\mathbf{b}$. 
It thus suffices to verify the remaining condition in \eqref{eq:critical_cone_charac_tau}, i.e.,  $\mathbf{v}_i = (\tilde{\mathbf{u}} - \bar{\mathbf{u}})_{i} = 0$ if $|\bar{\mathbf{d}}_i(x)|> \tau$, with $i \in \{1,\cdots,d\}.$ To accomplish this task, we first use the triangle inequality and invoke the Lipschitz property of $\Pi_{[\mathbf{a},\mathbf{b}]}$, in conjunction with the assumption $\|\bar{\mathbf{z}}-\bar{\mathbf{z}}_{\T}\|_{\mathbf{L}^\infty(\Omega)}
\leq \tau/2$, to obtain
\begin{equation}\label{eq:difference_d}
\|\bar{\mathbf{z}}+\alpha \bar{\mathbf{u}}-(\bar{\mathbf{z}}_\T+\alpha\tilde{\mathbf{u}})\|_{\mathbf{L}^\infty(\Omega)}\leq 2\|\bar{\mathbf{z}}-\bar{\mathbf{z}}_\T\|_{\mathbf{L}^\infty(\Omega)}\leq\tau.
\end{equation}
Now, let $\xi\in \Omega$ and $i\in\{1,\ldots,d\}$ be such that $\bar{\mathbf{d}}_{i}(\xi) = (\bar{\mathbf{z}}+\alpha \bar{\mathbf{u}})_{i}(\xi)>\tau$. Since $\tau > 0$, this implies that
$
\bar{\mathbf{u}}_{i}(\xi)> - \alpha^{-1}\bar{\mathbf{z}}_{i}(\xi).
$
Therefore, from the projection formula \eqref{eq:projection_formula}, we conclude that $\bar{\mathbf{u}}_{i}(\xi)=\mathbf{a}_{i}$. On the other hand, since $\xi\in \Omega$ is such that $(\bar{\mathbf{z}}+\alpha \bar{\mathbf{u}})_{i}(\xi)>\tau$, from \eqref{eq:difference_d} we can conclude that
\[
(\bar{\mathbf{z}}_\T+\alpha\tilde{\mathbf{u}})_{i}(\xi)>0,
\]
and thus that
$
\tilde{\mathbf{u}}_{i}(\xi)> -\alpha^{-1}(\bar{\mathbf{z}}_{\T})_{i}(\xi).
$
This, on the basis of the definition of the auxiliary variable $\tilde{\mathbf{u}}$, given in \eqref{def:tilde_u}, yields that $\tilde{\mathbf{u}}_{i}(\xi)=\mathbf{a}_{i}$. Consequently, $\bar{\mathbf{u}}_{i}(\xi)=\tilde{\mathbf{u}}_{i}(\xi)=\mathbf{a}_{i}$. Since $i$ is arbitrary, we conclude that $(\tilde{\mathbf{u}}-\bar{\mathbf{u}})(\xi)=\boldsymbol 0$. Similar arguments allow us to conclude that, if $\bar{\mathbf{d}}_{i}(\xi) =(\bar{\mathbf{z}}+\alpha \bar{\mathbf{u}})_{i}(\xi)<-\tau$, with $i\in\{1,\ldots,d\}$, then $(\tilde{\mathbf{u}}-\bar{\mathbf{u}})(\xi)=\boldsymbol 0$.
\end{remark}
}

The following auxiliary variables are also of particular importance for our reliability analysis. Let $(\tilde{\mathbf{y}},\tilde{p}) \in \mathbf{H}_0^1(\Omega)\times L_0^2(\Omega)$ be the solution to
\begin{equation}\label{eq:tilde_y}
\begin{aligned}
\nu(\nabla \tilde{\mathbf{y}}, \nabla \mathbf{v})_{\mathbf{L}^2(\Omega)}+b(\tilde{\mathbf{y}};\tilde{\mathbf{y}},\mathbf{v})-(\tilde{p},\text{div } \mathbf{v})_{L^2(\Omega)} & = (\tilde{\mathbf{u}},\mathbf{v})_{\mathbf{L}^2(\Omega)}  & \forall \mathbf{v} \in \mathbf{H}_0^1(\Omega),\\
(q,\text{div } \tilde{\mathbf{y}})_{L^2(\Omega)} & =  0  & \forall q\in L^2_0(\Omega).
\end{aligned}
\end{equation}
We also introduce the pair $(\tilde{\mathbf{z}},\tilde{r}) \in \mathbf{H}_0^1(\Omega)\times L_0^2(\Omega)$ as the solution to
\begin{equation*}
\begin{aligned}
\nu(\nabla \mathbf{w}, \nabla \tilde{\mathbf{z}})_{\mathbf{L}^2(\Omega)}+b(\tilde{\mathbf{y}};\mathbf{w},\tilde{\mathbf{z}})+b(\mathbf{w};\tilde{\mathbf{y}},\tilde{\mathbf{z}})-(\tilde{r},\text{div } \mathbf{w})_{L^2(\Omega)} &= (\tilde{\mathbf{y}}-\mathbf{y}_\Omega,\mathbf{w})_{\mathbf{L}^2(\Omega)}, \\
(s,\text{div } \tilde{\mathbf{z}})_{L^2(\Omega)} & =  0,
\end{aligned}
\end{equation*}
for all $\mathbf{w}  \in \mathbf{H}_0^1(\Omega)$ and $s \in L_0^2(\Omega)$.

To present the following result, we define $\mathbf{e}_{\mathbf{y}}:= \bar{\mathbf{y}}-\bar{\mathbf{y}}_\T$, $e_p:=\bar{p}-\bar{p}_\T$, $\mathbf{e}_{\mathbf{z}}:= \bar{\mathbf{z}}-\bar{\mathbf{z}}_\T$, $e_r:=\bar{r}-\bar{r}_\T$, $\mathbf{e}_{\mathbf{u}} := \bar{\mathbf{u}}-\bar{\mathbf{u}}_\T$, the total error norm
\begin{equation}\label{def:error_norm}
\|\mathbf{e} \|^2_{\Omega}: =
\|\nabla \mathbf{e}_{\mathbf{y}}\|_{\mathbf{L}^2(\Omega)}^2 +
\|e_p \|_{L^{2}(\Omega)}^2 +
\|\nabla \mathbf{e}_{\mathbf{z}} \|_{\mathbf{L}^2(\Omega)}^2
+ \|e_r\|_{L^2(\Omega)}^2
+ \|\mathbf{e}_{\mathbf{u}}\|_{\mathbf{L}^2(\Omega)}^2,
\end{equation}
\EO{and the a posteriori error estimator
\begin{equation}\label{def:error_estimator_ocp}
\mathcal{E}_{ocp}^2:=  \mathcal{E}_{ad}^2 + \mathcal{E}_{st}^2+ \mathcal{E}_{ct}^2.
\end{equation}
The estimators $\mathcal{E}_{st}$, $\mathcal{E}_{ad}$, and $\mathcal{E}_{ct}$, are defined as in \eqref{eq:error_estimator_st}, \eqref{eq:error_estimator_adjoint_eq}, and \eqref{eq:error_estimator_control_var}, respectively.}

We are now ready to state and prove the main result of this section.

\begin{theorem}[global reliability \EO{of $\mathcal{E}_{ocp}$}]
\label{thm:control_bound}
Assume that assumptions \eqref{eq:smallness_assumption} and \eqref{eq:assumption_y_T_small_2} hold. Let $(\bar{\mathbf{y}},\bar{p},\bar{\mathbf{z}},\bar{r},\bar{\mathbf{u}}) \in \mathbf{H}_0^1(\Omega)\times L_0^2(\Omega) \times \mathbf{H}_0^1(\Omega)\times L_0^2(\Omega) \times \mathbb{U}_{ad}$ be a local solution of \eqref{eq:min_weak_setting}--\eqref{eq:weak_st_equation} that satisfies the sufficient second order optimality condition \eqref{eq:equivalent_second_I}, or equivalently \eqref{eq:equivalent_second_II}. Let $\bar{\mathbf{u}}_{\T}$ be a local minimum of the \FF{fully} discrete optimal control problem \eqref{eq:discrete_cost_mini}--\eqref{eq:discrete_state_equation}, with $(\bar{\mathbf{y}}_{\T},\bar{p}_\T)$ and $(\bar{\mathbf{z}}_{\T},\bar{r}_\T)$ being the corresponding state and adjoint state variables, respectively. Let $\T$ be a mesh such that \eqref{eq:assumption_mesh} holds. Then
\begin{equation}\label{eq:global_rel}
\|\mathbf{e}  \|_{\Omega}^2
\lesssim 
\mathcal{E}_{ocp}^2,
\end{equation} 
with a hidden constant that is independent of the continuous and discrete optimal variables, the size of the elements in the mesh $\T$, and $\#\T$.
\end{theorem}
\begin{proof}
We proceed in six steps.

\emph{Step 1.} The goal of this step is to control $\|\mathbf{e}_\mathbf{u}\|_{\mathbf{L}^2(\Omega)}$ in \eqref{def:error_norm}. Invoke the auxiliary variable $\tilde{\mathbf{u}}$, defined in \eqref{def:tilde_u}, and definition \eqref{eq:error_estimator_control_var} to arrive at
\begin{equation}\label{eq:control_estimate_I}
\|\mathbf{e}_\mathbf{u}\|_{\mathbf{L}^2(\Omega)}\leq \|\bar{\mathbf{u}}-\tilde{\mathbf{u}}\|_{\mathbf{L}^2(\Omega)}+\mathcal{E}_{ct}.
\end{equation}
It thus suffices to bound $\|\bar{\mathbf{u}}-\tilde{\mathbf{u}}\|_{\mathbf{L}^2(\Omega)}$. To accomplish this task, we set $\mathbf{u} = \tilde{\mathbf{u}}$ in \eqref{eq:variational_ineq} and $\mathbf{u} = \bar{\mathbf{u}}$ in \eqref{eq:var_ineq_tilde_u} to obtain
\[
j'(\bar{\mathbf{u}})(\tilde{\mathbf{u}}-\bar{\mathbf{u}})=(\bar{\mathbf{z}}+\alpha\bar{\mathbf{u}}, \tilde{\mathbf{u}}-\bar{\mathbf{u}})_{\mathbf{L}^2(\Omega)}\geq 0,
\qquad
 -(\bar{\mathbf{z}}_\T+\alpha\tilde{\mathbf{u}},\tilde{\mathbf{u}}-\bar{\mathbf{u}})_{\mathbf{L}^2(\Omega)}\geq 0.
\]
With these estimates at hand, we invoke inequality \eqref{eq:inequality_control_tilde} to conclude 
\begin{align*}
\tfrac{\mu}{2}\|\bar{\mathbf{u}}-\tilde{\mathbf{u}}\|_{\mathbf{L}^2(\Omega)}^2
&\leq j'(\tilde{\mathbf{u}})(\tilde{\mathbf{u}}-\bar{\mathbf{u}})-j'(\bar{\mathbf{u}})(\tilde{\mathbf{u}}-\bar{\mathbf{u}})\leq  j'(\tilde{\mathbf{u}})(\tilde{\mathbf{u}}-\bar{\mathbf{u}})\\
&\leq j'(\tilde{\mathbf{u}})(\tilde{\mathbf{u}}-\bar{\mathbf{u}})-(\bar{\mathbf{z}}_\T+\alpha\tilde{\mathbf{u}},\tilde{\mathbf{u}}-\bar{\mathbf{u}})_{\mathbf{L}^2(\Omega)}
= (\tilde{\mathbf{z}}-\bar{\mathbf{z}}_\T,\tilde{\mathbf{u}}-\bar{\mathbf{u}})_{\mathbf{L}^2(\Omega)}.
\end{align*}
Adding and subtracting 
the auxiliary variable $\hat{\mathbf{z}}$, \EO{where $(\hat{\mathbf{z}},\hat{r})$ denotes the} solution to \eqref{eq:hat_function_ad}, and utilizing the Cauchy--Schwarz and triangle inequalities we obtain
\begin{equation*}
\|\bar{\mathbf{u}}-\tilde{\mathbf{u}}\|_{\mathbf{L}^2(\Omega)}
\lesssim
\|\tilde{\mathbf{z}}-\hat{\mathbf{z}}\|_{\mathbf{L}^2(\Omega)}+\|\hat{\mathbf{z}}-\bar{\mathbf{z}}_\T\|_{\mathbf{L}^2(\Omega)}.
\end{equation*}
A Poincar\'e inequality combined with the a posteriori error estimate \eqref{eq:adjoint_hat_estimate} yield
\begin{equation}\label{eq:control_estimate_III}
\|\bar{\mathbf{u}}-\tilde{\mathbf{u}}\|_{\mathbf{L}^2(\Omega)}
\lesssim
\|\nabla(\tilde{\mathbf{z}}-\hat{\mathbf{z}})\|_{\mathbf{L}^2(\Omega)}+ \mathcal{E}_{ad}.
\end{equation}

We now estimate the remaining term $\|\nabla(\tilde{\mathbf{z}}-\hat{\mathbf{z}})\|_{\mathbf{L}^2(\Omega)}$. To accomplish this task, we notice that the pair $(\tilde{\mathbf{z}}-\hat{\mathbf{z}},\tilde{r}-\hat{r})\in\mathbf{H}_0^1(\Omega)\times L_0^2(\Omega)$ solves 
\begin{equation*}
\begin{aligned}
\nu(\nabla \mathbf{w}, \nabla (\tilde{\mathbf{z}}-\hat{\mathbf{z}}))_{\mathbf{L}^2(\Omega)}+b(\tilde{\mathbf{y}};\mathbf{w},\tilde{\mathbf{z}})-b(\bar{\mathbf{y}}_\T;\mathbf{w},\hat{\mathbf{z}})
\\
+b(\mathbf{w};\tilde{\mathbf{y}},\tilde{\mathbf{z}})-b(\mathbf{w};\bar{\mathbf{y}}_\T,\hat{\mathbf{z}})-(\tilde{r}-\hat{r},\text{div } \mathbf{w})_{L^2(\Omega)} & = (\tilde{\mathbf{y}}-\bar{\mathbf{y}}_\T,\mathbf{w})_{\mathbf{L}^2(\Omega)},\\
(s,\text{div } (\tilde{\mathbf{z}}-\hat{\mathbf{z}}))_{L^2(\Omega)} &=  0,
\end{aligned}
\end{equation*}
for all $\mathbf{w} \in \mathbf{H}_0^1(\Omega)$ and $s \in L_0^2(\Omega)$, respectively. Set $s=0$ and $\mathbf{w}=\tilde{\mathbf{z}}-\hat{\mathbf{z}}$ to obtain
\begin{multline*}
\nu\|\nabla (\tilde{\mathbf{z}}-\hat{\mathbf{z}})\|_{\mathbf{L}^2(\Omega)}^2+b(\tilde{\mathbf{y}}-\bar{\mathbf{y}}_\T;\tilde{\mathbf{z}}-\hat{\mathbf{z}},\tilde{\mathbf{z}})+b(\bar{\mathbf{y}}_\T;\tilde{\mathbf{z}}-\hat{\mathbf{z}},\tilde{\mathbf{z}}-\hat{\mathbf{z}})\\
+b(\tilde{\mathbf{z}}-\hat{\mathbf{z}};\tilde{\mathbf{y}}-\bar{\mathbf{y}}_\T,\tilde{\mathbf{z}})+b(\tilde{\mathbf{z}}-\hat{\mathbf{z}};\bar{\mathbf{y}}_\T,\tilde{\mathbf{z}}-\hat{\mathbf{z}})=(\tilde{\mathbf{y}}-\bar{\mathbf{y}}_\T,\tilde{\mathbf{z}}-\hat{\mathbf{z}})_{\mathbf{L}^2(\Omega)}.
\end{multline*}
Invoke now estimate \eqref{eq:trilinear_embedding} and the Cauchy--Schwarz inequality to obtain
\begin{multline*}
\nu\|\nabla (\tilde{\mathbf{z}}-\hat{\mathbf{z}})\|_{\mathbf{L}^2(\Omega)}^2 \leq 2\mathcal{C}_{b}\|\nabla(\tilde{\mathbf{z}}-\hat{\mathbf{z}})\|_{\mathbf{L}^2(\Omega)}\|\nabla(\tilde{\mathbf{y}}-\bar{\mathbf{y}}_\T)\|_{\mathbf{L}^2(\Omega)}\|\nabla\tilde{\mathbf{z}}\|_{\mathbf{L}^2(\Omega)}\\
+2\mathcal{C}_{b}\|\nabla(\tilde{\mathbf{z}}-\hat{\mathbf{z}})\|_{\mathbf{L}^2(\Omega)}^2\|\nabla \bar{\mathbf{y}}_\T\|_{\mathbf{L}^2(\Omega)}+\|\tilde{\mathbf{y}}-\bar{\mathbf{y}}_\T\|_{\mathbf{L}^2(\Omega)}\|\tilde{\mathbf{z}}-\hat{\mathbf{z}}\|_{\mathbf{L}^2(\Omega)}.
\end{multline*}
Then, in view of assumption \eqref{eq:assumption_y_T_small_2}, it immediately follows that
\begin{multline*}
\nu(1-\theta)\|\nabla (\tilde{\mathbf{z}}-\hat{\mathbf{z}})\|_{\mathbf{L}^2(\Omega)}^2
 \leq 
 \|\tilde{\mathbf{y}}-\bar{\mathbf{y}}_\T\|_{\mathbf{L}^2(\Omega)}\|\tilde{\mathbf{z}}-\hat{\mathbf{z}}\|_{\mathbf{L}^2(\Omega)}\\
+2\mathcal{C}_{b}\|\nabla(\tilde{\mathbf{z}}-\hat{\mathbf{z}})\|_{\mathbf{L}^2(\Omega)}\|\nabla(\tilde{\mathbf{y}}-\bar{\mathbf{y}}_\T)\|_{\mathbf{L}^2(\Omega)}\|\nabla\tilde{\mathbf{z}}\|_{\mathbf{L}^2(\Omega)}.
\end{multline*}
Applying a Poincar\'e inequality, adding and subtracting the auxiliary variable $\hat{\mathbf{y}}$, \EO{where $(\hat{\mathbf{y}},\hat{p})$ denotes} the solution to \eqref{eq:hat_function_st}, and using the triangle inequality, we arrive at
\begin{equation}\label{eq:tilde_hat_adjoint_IV}
\|\nabla (\tilde{\mathbf{z}}-\hat{\mathbf{z}})\|_{\mathbf{L}^2(\Omega)} \lesssim 
\left(1+\|\nabla\tilde{\mathbf{z}}\|_{\mathbf{L}^2(\Omega)}\right)
\left(\|\nabla(\tilde{\mathbf{y}}-\hat{\mathbf{y}})\|_{\mathbf{L}^2(\Omega)}+\|\nabla(\hat{\mathbf{y}}-\bar{\mathbf{y}}_\T)\|_{\mathbf{L}^2(\Omega)}\right).
\end{equation}
Notice that \EO{an} stability estimate for the problem that $(\tilde{\mathbf{z}},\tilde{r})$ solves yields
\[
\|\nabla \tilde{\mathbf{z}}\|_{\mathbf{L}^2(\Omega)}
\leq \frac{C_2}{\nu(1-\theta)} \|\tilde{\mathbf{y}} - \mathbf{y}_\Omega\|_{\mathbf{L}^2(\Omega)}
\leq \frac{C_2}{\nu(1-\theta)} \left( C_2\mathcal{C}_b^{-1} \theta \nu + \|\mathbf{y}_\Omega\|_{\mathbf{L}^2(\Omega)}\right),
\]
where we have also used \eqref{eq:stability_state_eq}. Replacing this estimate into \eqref{eq:tilde_hat_adjoint_IV} and invoking the a posteriori error estimate \eqref{eq:estimate_state_hat_discrete} we obtain 
\begin{equation}\label{eq:tilde_hat_adjoint_V}
\|\nabla (\tilde{\mathbf{z}}-\hat{\mathbf{z}})\|_{\mathbf{L}^2(\Omega)} \lesssim \|\nabla(\tilde{\mathbf{y}}-\hat{\mathbf{y}})\|_{\mathbf{L}^2(\Omega)}+\mathcal{E}_{st},
\end{equation}
with a hidden constant that is independent of the continuous and discrete optimal variables, the size of the elements in the mesh $\T$, and $\#\T$ but depends on the continuous problem data and \EO{the constants} $C_2$, $\mathcal{C}_b$, \EO{$\nu$}, and $\theta$.

The rest of this step is dedicated to bound the term $\|\nabla(\tilde{\mathbf{y}}-\hat{\mathbf{y}})\|_{\mathbf{L}^2(\Omega)}$ in \eqref{eq:tilde_hat_adjoint_V}. To accomplish this task, we first notice that $(\tilde{\mathbf{y}}-\hat{\mathbf{y}},\tilde{p}-\hat{p})\in \mathbf{H}_0^1(\Omega)\times L_0^2(\Omega)$ solves 
\begin{equation*}
\hspace{-0.2cm}
\begin{array}{rl}
\nu(\nabla (\tilde{\mathbf{y}}\!-\!\hat{\mathbf{y}}),\nabla \mathbf{v})_{\mathbf{L}^2(\Omega)}\!+\!b(\tilde{\mathbf{y}};\tilde{\mathbf{y}},\mathbf{v})\!-\!b(\hat{\mathbf{y}};\hat{\mathbf{y}},\mathbf{v}) \!-\!(\tilde{p}-\hat{p},\text{div }\mathbf{v})_{L^2(\Omega)} &\hspace{-0.3cm} =  \! (\tilde{\mathbf{u}}-\bar{\mathbf{u}}_\T,\mathbf{v})_{\mathbf{L}^2(\Omega)}, \\
(q,\text{div }(\tilde{\mathbf{y}}-\hat{\mathbf{y}}))_{L^2(\Omega)} & \hspace{-0.3cm}= \!0,
\end{array}
\end{equation*}
for all $\mathbf{v} \in\mathbf{H}_0^1(\Omega)$ and  $q \in L_0^2(\Omega)$, respectively. Set $\mathbf{v}=\tilde{\mathbf{y}}-\hat{\mathbf{y}}$ and $q=0$, and invoke the second property for the form $b$ stated in \eqref{eq:properties_trilinear} to arrive at
\begin{equation*}
\nu \|\nabla (\tilde{\mathbf{y}}-\hat{\mathbf{y}})\|_{\mathbf{L}^2(\Omega)}^2+b(\tilde{\mathbf{y}}-\hat{\mathbf{y}};\hat{\mathbf{y}},\tilde{\mathbf{y}}-\hat{\mathbf{y}})= (\tilde{\mathbf{u}}-\bar{\mathbf{u}}_\T,\tilde{\mathbf{y}}-\hat{\mathbf{y}})_{\mathbf{L}^2(\Omega)}.
\end{equation*}
Estimates \eqref{eq:trilinear_embedding} and \eqref{eq:stab_hat_y} thus yield
$
\|\nabla (\tilde{\mathbf{y}}-\hat{\mathbf{y}})\|_{\mathbf{L}^2(\Omega)}
\lesssim
\|\tilde{\mathbf{u}}-\bar{\mathbf{u}}_\T\|_{\mathbf{L}^2(\Omega)}=\mathcal{E}_{ct},
$
upon using \eqref{eq:error_estimator_control_var}. Replacing \EO{this bound} into \eqref{eq:tilde_hat_adjoint_V}, and the obtained one into  \eqref{eq:control_estimate_III}, we obtain
\begin{equation}\label{eq:tilde_hat_adjoint_VI}
\|\bar{\mathbf{u}}-\tilde{\mathbf{u}}\|_{\mathbf{L}^2(\Omega)} 
\lesssim 
\mathcal{E}_{ad}+ \mathcal{E}_{st}+\mathcal{E}_{ct}.
\end{equation}
On the basis of \eqref{eq:tilde_hat_adjoint_VI} and \eqref{eq:control_estimate_I}, we can thus obtain the a posteriori error estimate
\begin{equation}\label{eq:final_estimate_control}
\|\mathbf{e}_{\mathbf{u}}\|_{\mathbf{L}^2(\Omega)}
\lesssim
\mathcal{E}_{ad}+ \mathcal{E}_{st}+\mathcal{E}_{ct}.
\end{equation}

\emph{Step 2.} The goal of this step is to bound $\|\nabla \mathbf{e}_{\mathbf{y}}\|_{\mathbf{L}^2(\Omega)}$ in \eqref{def:error_norm}. We begin with a simple application of the triangle inequality and \eqref{eq:estimate_state_hat_discrete} \EO{to obtain}
\begin{equation}\label{eq:state_velocity_I}
\|\nabla \mathbf{e}_{\mathbf{y}}\|_{\mathbf{L}^2(\Omega)}
\FF{\lesssim}
\|\nabla(\bar{\mathbf{y}}-\hat{\mathbf{y}})\|_{\mathbf{L}^2(\Omega)}+\mathcal{E}_{st}.
\end{equation}
We now bound $\|\nabla(\bar{\mathbf{y}}-\hat{\mathbf{y}})\|_{\mathbf{L}^2(\Omega)}$. To accomplish this task, we first notice that the pair $(\bar{\mathbf{y}}-\hat{\mathbf{y}},\bar{p}-\hat{p})\in\mathbf{H}_0^1(\Omega)\times L_0^2(\Omega)$ solves the problem 
\begin{equation}\label{eq:equation_hat_bar_st}
\begin{aligned}
\EO{\nu}(\nabla(\bar{\mathbf{y}}-\hat{\mathbf{y}}),\nabla \mathbf{v})_{\mathbf{L}^2(\Omega)}+b(\bar{\mathbf{y}};\bar{\mathbf{y}}-\hat{\mathbf{y}},\mathbf{v})+b(\bar{\mathbf{y}}-\hat{\mathbf{y}};\hat{\mathbf{y}},\mathbf{v})
\\
-(\bar{p}-\hat{p},\text{div } \mathbf{v})_{L^2(\Omega)} & = (\bar{\mathbf{u}}-\bar{\mathbf{u}}_\T,\mathbf{v})_{\mathbf{L}^2(\Omega)},\\
(q,\text{div } (\bar{\mathbf{y}}-\hat{\mathbf{y}}))_{L^2(\Omega)} &= 0,
\end{aligned}
\end{equation}
for all $\mathbf{v} \in\mathbf{H}_0^1(\Omega)$ and $q \in L_0^2(\Omega)$, respectively. Set $\mathbf{v}=\bar{\mathbf{y}}-\hat{\mathbf{y}}$ and $q=0$, and invoke \eqref{eq:properties_trilinear} and the fact that $\bar{\mathbf{y}}-\hat{\mathbf{y}}\in\mathbf{V}(\Omega)$ to arrive at
\begin{equation*}
\nu \|\nabla (\bar{\mathbf{y}}-\hat{\mathbf{y}})\|_{\mathbf{L}^2(\Omega)}^2+b(\bar{\mathbf{y}}-\hat{\mathbf{y}};\hat{\mathbf{y}},\bar{\mathbf{y}}-\hat{\mathbf{y}})= (\bar{\mathbf{u}}-\bar{\mathbf{u}}_\T,\bar{\mathbf{y}}-\hat{\mathbf{y}})_{\mathbf{L}^2(\Omega)}.
\end{equation*}
We thus invoke \eqref{eq:trilinear_embedding} and the stability estimate \eqref{eq:stab_hat_y} to obtain
\begin{equation}\label{eq:hat_bar_st_II}
\|\nabla (\bar{\mathbf{y}}-\hat{\mathbf{y}})\|_{\mathbf{L}^2(\Omega)}
\lesssim
\|\mathbf{e}_{\mathbf{u}}\|_{\mathbf{L}^2(\Omega)}.
\end{equation}
We finally replace estimate \eqref{eq:hat_bar_st_II} into \eqref{eq:state_velocity_I} and invoke
\eqref{eq:final_estimate_control} to obtain the error estimate
\begin{equation}\label{eq:final_estimate_st_velocity}
\|\nabla\mathbf{e}_{\mathbf{y}}\|_{\mathbf{L}^2(\Omega)}
\lesssim
\mathcal{E}_{ad}+ \mathcal{E}_{st}+\mathcal{E}_{ct}.
\end{equation}

\emph{Step 3.} We now estimate the term $\|e_{p}\|_{L^2(\Omega)}$ in \eqref{def:error_norm}. A trivial application of the triangle inequality in conjunction with the a posteriori estimate \eqref{eq:estimate_state_hat_discrete} yield
\begin{equation}\label{eq:state_pressure_I}
\|e_{p}\|_{L^2(\Omega)}
\lesssim
\|\bar{p}-\hat{p}\|_{L^2(\Omega)}+\mathcal{E}_{st}.
\end{equation}
It thus suffices to bound $\|\bar{p}-\hat{p}\|_{L^2(\Omega)}$. To do this, we utilize the inf-sup condition \eqref{eq:inf_sup_cond}, the fact that $(\bar{\mathbf{y}}-\hat{\mathbf{y}},\bar{p}-\hat{p})\in\mathbf{H}_0^1(\Omega)\times L_0^2(\Omega)$ solves \eqref{eq:equation_hat_bar_st} and \eqref{eq:trilinear_embedding}. In fact, we have
\begin{multline}\label{eq:state_pressure_II}
\|\bar{p}-\hat{p}\|_{L^2(\Omega)}
\lesssim
\sup_{\mathbf{v}\in\mathbf{H}_0^1(\Omega)}\frac{(\bar{p}-\hat{p},\text{div } \mathbf{v})_{L^2(\Omega)}}{\|\nabla \mathbf{v}\|_{\mathbf{L}^2(\Omega)}} \lesssim \|\nabla(\bar{\mathbf{y}}-\hat{\mathbf{y}})\|_{\mathbf{L}^2(\Omega)}\\
+ \|\nabla(\bar{\mathbf{y}}-\hat{\mathbf{y}})\|_{\mathbf{L}^2(\Omega)}(\|\nabla\bar{\mathbf{y}}\|_{\mathbf{L}^2(\Omega)}+\|\nabla\hat{\mathbf{y}}\|_{\mathbf{L}^2(\Omega)})+\|\mathbf{e}_{\mathbf{u}}\|_{\mathbf{L}^2(\Omega)}.
\end{multline}
Since the smallness assumption \eqref{eq:smallness_assumption} holds, we immediately \EO{have} the stability estimates \eqref{eq:stability_state_eq} and \eqref{eq:stab_hat_y}. Thus,
$
\|\nabla\bar{\mathbf{y}}\|_{\mathbf{L}^2(\Omega)}+\|\nabla\hat{\mathbf{y}}\|_{\mathbf{L}^2(\Omega)}\leq 2\theta\mathcal{C}_{b}^{-1}\nu,
$
\EO{with $\theta < 1$.} Replace this estimate into \eqref{eq:state_pressure_II} and invoke \eqref{eq:hat_bar_st_II} and \eqref{eq:final_estimate_control} to obtain $\|\bar{p}-\hat{p}\|_{L^2(\Omega)}\lesssim \mathcal{E}_{ad}+ \mathcal{E}_{st}+\mathcal{E}_{ct}$. This estimate, in view of \eqref{eq:state_pressure_I}, yields the a posteriori error estimate
\begin{equation}\label{eq:final_estimate_st_pressure}
\|e_{p}\|_{L^2(\Omega)}
\lesssim
\mathcal{E}_{ad}+ \mathcal{E}_{st}+\mathcal{E}_{ct}.
\end{equation}

\emph{Step 4.} We bound $\|\nabla \mathbf{e}_{\mathbf{z}}\|_{\mathbf{L}^2(\Omega)}$. To accomplish this task, we apply the triangle inequality and invoke the a posteriori estimate \eqref{eq:adjoint_hat_estimate}. These arguments yield
\begin{equation}\label{eq:adjoint_state_velocity_I}
\|\nabla \mathbf{e}_{\mathbf{z}}\|_{\mathbf{L}^2(\Omega)}
\lesssim
\|\nabla(\bar{\mathbf{z}}-\hat{\mathbf{z}})\|_{\mathbf{L}^2(\Omega)}+ \mathcal{E}_{ad}.
\end{equation}
\EO{To bound} $\|\nabla(\bar{\mathbf{z}}-\hat{\mathbf{z}})\|_{\mathbf{L}^2(\Omega)}$ we observe that $(\bar{\mathbf{z}}-\hat{\mathbf{z}},\bar{r}-\hat{r})\in\mathbf{H}_0^1(\Omega)\times L_0^2(\Omega)$ solves 
\begin{align}\label{eq:bar_hat_adjoint_I}
\begin{split}
\nu(\nabla \mathbf{w}, \nabla (\bar{\mathbf{z}}-\hat{\mathbf{z}}))_{\mathbf{L}^2(\Omega)}+b(\bar{\mathbf{y}}-\bar{\mathbf{y}}_\T;\mathbf{w},\bar{\mathbf{z}})+b(\bar{\mathbf{y}}_\T;\mathbf{w},\bar{\mathbf{z}}-\hat{\mathbf{z}})
\\ 
+b(\mathbf{w};\bar{\mathbf{y}}-\bar{\mathbf{y}}_\T,\bar{\mathbf{z}})
+b(\mathbf{w};\bar{\mathbf{y}}_\T,\bar{\mathbf{z}}-\hat{\mathbf{z}})-(\bar{r}-\hat{r},\text{div } \mathbf{w})_{L^2(\Omega)} & \!= \! (\bar{\mathbf{y}}\!-\!\bar{\mathbf{y}}_\T,\mathbf{w})_{\mathbf{L}^2(\Omega)},
\\
(s,\text{div } (\bar{\mathbf{z}}-\hat{\mathbf{z}}))_{L^2(\Omega)} & \!= \!0,
\end{split}\hspace{-0.4cm}
\end{align}
for all \EO{$(\mathbf{w},s) \in \mathbf{H}_0^1(\Omega) \times L_0^2(\Omega)$. Set $(\mathbf{w},q)=(\bar{\mathbf{z}}-\hat{\mathbf{z}},0)$} and invoke \eqref{eq:trilinear_embedding} to obtain
\begin{multline*}
\nu\|\nabla (\bar{\mathbf{z}}-\hat{\mathbf{z}})\|_{\mathbf{L}^2(\Omega)}^2\leq 
2\mathcal{C}_{b}\|\nabla (\bar{\mathbf{y}}-\bar{\mathbf{y}}_\T)\|_{\mathbf{L}^2(\Omega)}\|\nabla \bar{\mathbf{z}}\|_{\mathbf{L}^2(\Omega)}\|\nabla (\bar{\mathbf{z}}-\hat{\mathbf{z}})\|_{\mathbf{L}^2(\Omega)}\\
+ 2\mathcal{C}_{b}\|\nabla (\bar{\mathbf{z}}-\hat{\mathbf{z}})\|_{\mathbf{L}^2(\Omega)}^2\|\nabla \bar{\mathbf{y}}_\T\|_{\mathbf{L}^2(\Omega)}+\|\bar{\mathbf{y}}-\bar{\mathbf{y}}_\T\|_{\mathbf{L}^2(\Omega)}\|\bar{\mathbf{z}}-\hat{\mathbf{z}}\|_{\mathbf{L}^2(\Omega)}.
\end{multline*}
Utilize \eqref{eq:assumption_y_T_small_2} and a Poincar\'e inequality to obtain
\begin{equation}
\label{eq:bar_hat_adjoint_II}
\nu(1-\theta) \|\nabla (\bar{\mathbf{z}}-\hat{\mathbf{z}})\|_{\mathbf{L}^2(\Omega)} 
\leq 
\left(2\mathcal{C}_{b}\|\nabla \bar{\mathbf{z}}\|_{\mathbf{L}^2(\Omega)}
+ C_2^2\right) \| \nabla( \bar{\mathbf{y}}-\bar{\mathbf{y}}_\T)\|_{\mathbf{L}^2(\Omega)}.
\end{equation}
We thus invoke the stability estimate \eqref{eq:stability_adjoint_par}, the smallness assumption \eqref{eq:smallness_assumption}, and the results of Theorem \ref{thm:well_posedness_navier_stokes} to obtain
\begin{equation}
\label{eq:bound_bar_z}
\|\nabla \bar{\mathbf{z}}\|_{\mathbf{L}^2(\Omega)} \leq \FF{[\nu(1-\theta)]^{-1}C_2} \left(C_2\theta\mathcal{C}_{b}^{-1}\nu+\|\mathbf{y}_\Omega\|_{\mathbf{L}^2(\Omega)}\right).
\end{equation}
Replace this estimate into \eqref{eq:bar_hat_adjoint_II} to obtain
$
\|\nabla (\bar{\mathbf{z}}-\hat{\mathbf{z}})\|_{\mathbf{L}^2(\Omega)}
\lesssim
\|\nabla \mathbf{e}_\mathbf{y}\|_{\mathbf{L}^2(\Omega)},
$
\EO{with} a hidden constant that is independent of the continuous and discrete optimal variables, the size of the elements in the mesh $\T$, and $\#\T$ but depends on the continuous problem data and \EO{the constants} $C_2$, $\mathcal{C}_b$, \EO{$\nu$,} and $\theta$. We thus invoke \eqref{eq:final_estimate_st_velocity} to obtain
\begin{equation}\label{eq:bar_hat_adjoint_IV}
\|\nabla (\bar{\mathbf{z}}-\hat{\mathbf{z}})\|_{\mathbf{L}^2(\Omega)}
\lesssim
\mathcal{E}_{ad}+ \mathcal{E}_{st}+\mathcal{E}_{ct},
\end{equation}
which, in view of  \eqref{eq:adjoint_state_velocity_I}, yields the a posteriori error estimate
\begin{equation}\label{eq:final_estimate_adj_velocity}
\|\nabla\mathbf{e}_{\mathbf{z}}\|_{\mathbf{L}^2(\Omega)}
\lesssim
\mathcal{E}_{ad}+ \mathcal{E}_{st}+\mathcal{E}_{ct}.
\end{equation}

\emph{Step 5.} We now control $\|e_{r}\|_{L^2(\Omega)}$ in \eqref{def:error_norm}. We begin by applying \eqref{eq:adjoint_hat_estimate} to derive
\begin{equation}\label{eq:adj_pressure_I}
\|e_{r}\|_{L^2(\Omega)}
\lesssim
\|\bar{r}-\hat{r}\|_{L^2(\Omega)}+\mathcal{E}_{ad}.
\end{equation}
To estimate $\|\bar{r}-\hat{r}\|_{L^2(\Omega)}$ we utilize the inf--sup condition \eqref{eq:inf_sup_cond}, problem \eqref{eq:bar_hat_adjoint_I}, and \eqref{eq:trilinear_embedding}:
\begin{multline*}
\|\bar{r}-\hat{r}\|_{L^2(\Omega)}
\lesssim
\sup_{\mathbf{w}\in\mathbf{H}_0^1(\Omega)}\frac{(\bar{r}-\hat{r},\text{div } \mathbf{w})_{L^2(\Omega)}}{\|\nabla \mathbf{w}\|_{\mathbf{L}^2(\Omega)}} \lesssim \|\nabla(\bar{\mathbf{z}}-\hat{\mathbf{z}})\|_{\mathbf{L}^2(\Omega)}+\|\bar{\mathbf{y}}-\bar{\mathbf{y}}_\T\|_{\mathbf{L}^2(\Omega)}\\
+ \|\nabla(\bar{\mathbf{y}}-\bar{\mathbf{y}}_\T)\|_{\mathbf{L}^2(\Omega)}\|\nabla\bar{\mathbf{z}}\|_{\mathbf{L}^2(\Omega)}+\|\nabla\bar{\mathbf{y}}_\T\|_{\mathbf{L}^2(\Omega)}\|\nabla(\bar{\mathbf{z}}-\hat{\mathbf{z}})\|_{\mathbf{L}^2(\Omega)}.
\end{multline*}
We thus invoke assumption \eqref{eq:assumption_y_T_small_2} and estimate \eqref{eq:bound_bar_z} to arrive at
\begin{equation*}
\|\bar{r}-\hat{r}\|_{L^2(\Omega)}
\lesssim
\|\nabla(\bar{\mathbf{z}}-\hat{\mathbf{z}})\|_{\mathbf{L}^2(\Omega)}+\|\nabla\mathbf{e}_{\mathbf{y}}\|_{\mathbf{L}^2(\Omega)},
\end{equation*}
with a hidden constant that is independent of the continuous and discrete optimal variables, the size of the elements in the mesh $\T$, and $\#\T$ but depends on the continuous problem data and \EO{the constants} $C_2$, $\mathcal{C}_b$, \EO{$\nu$}, and $\theta$. The estimates \eqref{eq:final_estimate_st_velocity} and \eqref{eq:bar_hat_adjoint_IV} immediately yield $\|\bar{r}-\hat{r}\|_{L^2(\Omega)} \lesssim \mathcal{E}_{ad}+ \mathcal{E}_{st}+\mathcal{E}_{ct}$. Finally, we replace this estimate into \eqref{eq:adj_pressure_I} to obtain the a posteriori error estimate
\begin{equation}\label{eq:final_estimate_adj_pressure}
\|e_r\|_{L^2(\Omega)}
\lesssim
\mathcal{E}_{ad}+ \mathcal{E}_{st}+\mathcal{E}_{ct}.
\end{equation}


\emph{Step 6.} The desired estimate \eqref{eq:global_rel} follows from collecting the estimates \eqref{eq:final_estimate_control}, \eqref{eq:final_estimate_st_velocity}, \eqref{eq:final_estimate_st_pressure}, \eqref{eq:final_estimate_adj_velocity}, and \eqref{eq:final_estimate_adj_pressure}. This concludes the proof.
\end{proof}


\subsection{Local efficiency analysis: \EO{the fully discrete scheme}}\label{sec:efficiency}
In this section, we analyze the efficiency properties of the a posteriori error estimator $\mathcal{E}_{ocp}$, defined in \eqref{def:error_estimator_ocp}, on the basis of standard bubble function arguments. Before proceeding with such an analysis, we introduce the following notation: For an edge, triangle, or tetrahedron $G$, let $\mathcal{V}(G)$ be the set of vertices of $G$. With this notation at hand, we introduce, for  $T\in\mathscr{T}$ and $S\in\mathscr{S}$, the following standard element and edge bubble functions \cite{MR3059294}:
\begin{equation*}\label{def:standard_bubbles}
\varphi^{}_{T}=
(d+1)^{(d+1)}\prod_{\textsc{v} \in \mathcal{V}(T)} \lambda^{}_{\textsc{v}},
\qquad
\varphi^{}_{S}=
d^{d} \prod_{\textsc{v} \in \mathcal{V}(S)}\lambda^{}_{\textsc{v}}|^{}_{T'} \text { with } T' \subset \mathcal{N}_{S}.
\end{equation*}
In these formulas, by $\lambda^{}_{\textsc{v}}$ we denote the barycentric coordinate function associated to $\textsc{v} \in \mathcal{V}(T)$. We recall that $\mathcal{N}_{S}$ corresponds to the patch composed of the two elements of $\mathscr{T}$ sharing $S$.

We \EO{derive} local efficiency estimates for the indicator $\mathcal{E}_{st,T}$ defined in \eqref{eq:error_estimator_st}.

\begin{theorem}[\FF{local estimates for} $\mathcal{E}_{st}$]\label{thm:local_eff_st}
Assume that assumptions \eqref{eq:smallness_assumption} \FF{and \eqref{eq:assumption_y_T_small_2} hold}. Let $(\bar{\mathbf{y}},\bar{p},\bar{\mathbf{z}},\bar{r},\bar{\mathbf{u}}) \in \mathbf{H}_0^1(\Omega)\times L_0^2(\Omega) \times \mathbf{H}_0^1(\Omega)\times L_0^2(\Omega) \times \mathbb{U}_{ad}$ be a local solution of \eqref{eq:min_weak_setting}--\eqref{eq:weak_st_equation}. Let $\bar{\mathbf{u}}_{\T}$ be a local minimum of the associated \FF{fully} discrete optimal control problem \FF{\eqref{eq:discrete_cost_mini}--\eqref{eq:discrete_state_equation}}, with $(\bar{\mathbf{y}}_{\T},\bar{p}_\T)$ and $(\bar{\mathbf{z}}_{\T},\bar{r}_\T)$ being the corresponding state and adjoint state discrete variables, respectively. Then, for $T\in\T$, the local error indicator $\mathcal{E}_{st,T}$ satisfies
\begin{equation}\label{eq:local_eff_st}
\mathcal{E}_{st,T}
\lesssim
\FF{\|\mathbf{e}_{\mathbf{y}}\|_{\mathbf{H}^1(\mathcal{N}_{T})}+\|e_p\|_{L^2(\mathcal{N}_{T})}+h_T\|\mathbf{e}_{\mathbf{u}}\|_{\mathbf{L}^2(\mathcal{N}_{T})}},
\end{equation}
where $\mathcal{N}_{T}$ is defined as in \eqref{def:patch}. The hidden constant is independent of the continuous and discrete optimal variables, the size of the elements in the mesh $\T$, and $\#\T$.
\end{theorem}
\begin{proof}

We begin by noticing that, since $(\bar{\mathbf{y}},\bar{p})\in \mathbf{H}_0^1(\Omega)\times L_0^2(\Omega)$ solves \eqref{eq:weak_st_equation} with $\mathbf{u}$ replaced by $\bar{\mathbf{u}}$, an elementwise integration by parts formula allows us to derive 
\begin{multline}\label{eq:error_eq_st}
\nu(\nabla \mathbf{e}_{\mathbf{y}},\nabla \mathbf{v})_{\mathbf{L}^2(\Omega)}+b(\FF{\mathbf{e}_{\mathbf{y}};\bar{\mathbf{y}}},\mathbf{v})+b(\FF{\bar{\mathbf{y}}_\T;\mathbf{e}_{\mathbf{y}},}\mathbf{v})-(e_p,\text{div }\mathbf{v})_{L^2(\Omega)}
\\
+(q,\text{div }\mathbf{e}_{\mathbf{y}})_{L^2(\Omega)}-(\mathbf{e}_{\mathbf{u}},\mathbf{v})_{\mathbf{L}^2(\Omega)}=\sum_{T\in\T}\left(\bar{\mathbf{u}}_\T+\nu\Delta\bar{\mathbf{y}}_\T-(\bar{\mathbf{y}}_\T\cdot \nabla)\bar{\mathbf{y}}_\T-\nabla \bar{p}_\T,\mathbf{v}\right)_{\mathbf{L}^2(T)}
\\
+\sum_{S\in\Sides}\left(\llbracket (\nu\nabla \bar{\mathbf{y}}_\T-\bar{p}_\T\mathbb{I}_{d} )\cdot \mathbf{n}\rrbracket,\mathbf{v}\right)_{\mathbf{L}^2(S)}-\sum_{T\in\T}(q,\text{div }\bar{\mathbf{y}}_\T)_{L^2(T)},
\end{multline}
which holds for every $\mathbf{v}\in \mathbf{H}_0^1(\Omega)$ and $q\in L_0^2(\Omega)$. With the aid of this identity, in the following steps, we will estimate separately each of the individual terms that appear in the definition of the local error indicator $\mathcal{E}_{st,T}$.

We now proceed on the basis of four steps.

\emph{Step 1.} Let $T\in \T$. Define \[
\mathbf{R}_{T}^{st}:=\left(\bar{\mathbf{u}}_\T+\nu\Delta\bar{\mathbf{y}}_\T-(\bar{\mathbf{y}}_\T\cdot \nabla)\bar{\mathbf{y}}_\T-\nabla \bar{p}_\T\right)|^{}_T.
\]
We bound $h_T\| \mathbf{R}_{T}^{st}\|_{\mathbf{L}^2(T)}$ in \eqref{eq:error_estimator_st}. To accomplish this task, we set $\mathbf{v} = \varphi_T\mathbf{R}_{T}^{st}$ and $q=0$ in \eqref{eq:error_eq_st} and utilize standard properties of the bubble function $\varphi_T$ \EO{combined with basic inequalities to arrive at}
\begin{multline*}
\|\mathbf{R}_{T}^{st}\|_{\mathbf{L}^2(T)}^2
\lesssim 
\left(\|\mathbf{e}_{\mathbf{y}}\|_{\mathbf{L}^2(T)}\|\nabla \bar{\mathbf{y}}\|_{\FF{\mathbf{L}^{d}}(T)}+\|\bar{\mathbf{y}}_\T\|_{\FF{\mathbf{L}^d}(T)}\|\nabla\mathbf{e}_{\mathbf{y}}\|_{\FF{\mathbf{L}^2}(T)}\right)\|\varphi_T\mathbf{R}_{T}^{st}\|_{\FF{\mathbf{L}^\mathfrak{d}}(T)}\\
+\left(\|\nabla \mathbf{e}_{\mathbf{y}}\|_{\mathbf{L}^2(T)}+\|e_p\|_{L^2(T)}\right)\|\nabla(\varphi_T\mathbf{R}_{T}^{st})\|_{\mathbf{L}^2(T)}
+
\|\mathbf{e}_{\mathbf{u}}\|_{\mathbf{L}^2(T)}\|\varphi_T\mathbf{R}_{T}^{st}\|_{\mathbf{L}^2(T)},
\end{multline*}
\FF{where $\mathfrak{d}=\infty$ if $d=2$ and $\mathfrak{d}=6$ if $d=3$.} We thus apply \EO{inverse inequalities \cite[Lemma 4.5.3]{MR2373954}} and bubble functions arguments to obtain
\begin{multline}\label{eq:eff_st_I}
\|\mathbf{R}_{T}^{st}\|_{\mathbf{L}^2(T)}
\lesssim 
\FF{h_T^{-1}}\left(\|\mathbf{e}_{\mathbf{y}}\|_{\mathbf{L}^2(T)}\|\nabla\bar{\mathbf{y}}\|_{\FF{\mathbf{L}^d}(T)}+\|\bar{\mathbf{y}}_\T\|_{\FF{\mathbf{L}^d}(T)}\|\nabla\mathbf{e}_{\mathbf{y}}\|_{\FF{\mathbf{L}^2}(T)}\right)\\
+h_T^{-1}\left(\|\nabla \mathbf{e}_{\mathbf{y}}\|_{\mathbf{L}^2(T)}+\|e_p\|_{L^2(T)}\right)+\|\mathbf{e}_{\mathbf{u}}\|_{\mathbf{L}^2(T)}.
\end{multline}
\EO{Observe that, since $d \in \{2,3\}$, $\mathbf{H}_0^1(\Omega)\hookrightarrow \mathbf{L}^{d}(\Omega)$. This and \eqref{eq:assumption_y_T_small_2} yield}
\begin{equation}\label{eq:eff_st_II}
\FF{\|\bar{\mathbf{y}}_\T\|_{\mathbf{L}^d(T)}\leq \|\bar{\mathbf{y}}_\T\|_{\mathbf{L}^d(\Omega)}\leq C\|\nabla\bar{\mathbf{y}}_\T\|_{\mathbf{L}^2(\Omega)} < C\nu\mathcal{C}_{b}^{-1}/2,}
\qquad
\EO{C > 0}.
\end{equation}
\EO{On the other hand, in view of \eqref{eq:stability_state_eq}, when $d=2$, and the fact that $\bar{\mathbf{y}}\in \mathbf{W}^{1,3}(\Omega)$, when $d=3$ (see Remark \ref{rk:regularity}) we have}
$
\FF{\|\nabla \bar{\mathbf{y}}\|_{\mathbf{L}^d(T)}
\leq \|\nabla \bar{\mathbf{y}}\|_{\mathbf{L}^d(\Omega)} < C,}
$
\EO{where $C > 0$}. Replacing this estimate and \eqref{eq:eff_st_II} into inequality \eqref{eq:eff_st_I}, we obtain
\begin{equation}\label{eq:final_residual_st}
h_T^2\|\mathbf{R}_{T}^{st}\|_{\mathbf{L}^2(T)}^2
\lesssim
\FF{\|\mathbf{e}_{\mathbf{y}}\|_{\mathbf{H}^1(T)}^2+\|e_p\|_{L^2(T)}^2+h_T^2\|\mathbf{e}_{\mathbf{u}}\|_{\mathbf{L}^2(T)}^2},
\end{equation}
with a hidden constant that is independent of the continuous and discrete optimal variables, the size of the elements in the mesh $\T$, and $\#\T$ but depends on the continuous problem data and \EO{the constants $\nu$ and $\mathcal{C}_b$.}

\emph{Step 2.} Let $T\in\T$ and $S\in\Sides_T$. Define $\mathbf{J}_{S}^{st}:= \llbracket (\nu\nabla \bar{\mathbf{y}}_\T-\bar{p}_\T\mathbb{I}_{d} )\cdot \mathbf{n}\rrbracket$. We bound the jump term $h_T\|\mathbf{J}_{S}^{st}\|_{\mathbf{L}^2(S)}^2$ in \eqref{eq:error_estimator_st}. To accomplish this task, we set $\mathbf{v}=\varphi_{S}\mathbf{J}_{S}^{st}$ and $q=0$ in \eqref{eq:error_eq_st} and proceed on the basis of similar arguments to the ones that lead to \eqref{eq:eff_st_I}. These arguments yield
\begin{multline*}
\|\mathbf{J}_{S}^{st}\|_{\mathbf{L}^2(S)}^2
\lesssim
\sum_{T'\in\mathcal{N}_S}\bigg(\|\mathbf{e}_{\mathbf{u}}\|_{\mathbf{L}^2(T')}+\|\mathbf{R}_{T'}^{st}\|_{\mathbf{L}^2(T')}+\FF{h_{T'}^{-1}}\left(\|\mathbf{e}_{\mathbf{y}}\|_{\mathbf{L}^2(T')}\|\nabla\bar{\mathbf{y}}\|_{\FF{\mathbf{L}^d}(T')}\right.
\\
\left.
+\|\bar{\mathbf{y}}_\T\|_{\FF{\mathbf{L}^d}(T')}\|\nabla\mathbf{e}_{\mathbf{y}}\|_{\FF{\mathbf{L}^2}(T')}\right)+h_{T'}^{-1}\left(\|\nabla \mathbf{e}_{\mathbf{y}}\|_{\mathbf{L}^2(T')}+\|e_p\|_{L^2(T')}\right)\bigg)h_T^{\frac{1}{2}}\|\mathbf{J}_S^{st}\|_{\mathbf{L}^2(S)}.
\end{multline*}
Invoke \eqref{eq:eff_st_II} and \eqref{eq:final_residual_st} to arrive at
\begin{equation}\label{eq:final_jump_st}
h_{T}\|\mathbf{J}_{S}^{st}\|_{\mathbf{L}^2(S)}^2
\lesssim 
\sum_{T'\in\mathcal{N}_S}\left(\FF{\|\mathbf{e}_{\mathbf{y}}\|_{\mathbf{H}^1(T')}^2+\|e_p\|_{L^2(T')}^2+h_{T'}^2\|\mathbf{e}_{\mathbf{u}}\|_{\mathbf{L}^2(T')}^2}\right).
\end{equation}

\emph{Step 3.} Let $T\in \T$. The goal of this step is to control the term $\|\text{div }\bar{\mathbf{y}}_\T\|_{L^2(T)}^2$ in \eqref{eq:error_estimator_st}. From the incompressibility condition $\text{div }\bar{\mathbf{y}} = 0$, it immediately follows that
\begin{equation}\label{eq:final_div_st}
\|\text{div }\bar{\mathbf{y}}_\T\|_{L^2(T)}
\leq 
\|\text{div }\mathbf{e}_\mathbf{y}\|_{L^2(T)}
\lesssim 
\|\nabla\mathbf{e}_\mathbf{y}\|_{\mathbf{L}^2(T)}
\lesssim 
\|\mathbf{e}_\mathbf{y}\|_{\mathbf{H}^1(T)}.
\end{equation}

\emph{Step 4.} The proof concludes by gathering the estimates \eqref{eq:final_residual_st}, \eqref{eq:final_jump_st}, and \eqref{eq:final_div_st}.
\end{proof}

\EO{
\begin{remark}[$\mathbf{W}^{1,3}(\Omega)$-regularity of $\bar{\mathbf{y}}$ for $d=3$]
\label{rk:regularity}
Observe that the pair $(\bar{\mathbf{y}} , \bar p) \in \mathbf{H}_0^1(\Omega) \times L_0^2(\Omega)$ can be seen as the solution to the following Stokes system:
\begin{equation*}
-\nu\Delta \bar{\mathbf{y}}+\nabla \bar p = \bar{\mathbf{u}} - (\bar{\mathbf{y}}\cdot\nabla)\bar{\mathbf{y}}\text{ in } \Omega, \quad \text{div }\bar{\mathbf{y}}=0 \text{ in } \Omega,\quad \bar{\mathbf{y}}=\boldsymbol 0  \text{ on } \partial\Omega.
\end{equation*}
Since $d=3$, $\bar{\mathbf{u}} \in \mathbf{L}^2(\Omega)$, and $(\bar{\mathbf{y}}\cdot\nabla)\bar{\mathbf{y}} \in \mathbf{W}^{-1,3}(\Omega)$, an application of the regularity results on Lipschitz domains of \cite[Corollary 1.7 (with $\alpha=-1$ and $q=2$)]{MR2987056} yields $\bar{\mathbf{y}} \in \mathbf{W}_0^{1,3}(\Omega)$. 
\end{remark}
}

We now investigate the local efficiency properties of $\mathcal{E}_{ad,T}$ defined in \eqref{eq:local_error_indicator_ad}. To accomplish this task, for any $\mathbf{g}\in \mathbf{L}^2(\Omega)$ and $\mathcal{M}\subset\T$, we define the oscillation term
\begin{equation}\label{eq:osc_term}
\mathrm{osc}_{\mathcal{M}}(\mathbf{g}):=\left(\sum_{T\in\mathcal{M}}h_T^2\|\mathbf{g}-\Pi_{T}(\mathbf{g})\|_{\mathbf{L}^2(T)}^2\right)^{\frac{1}{2}},
\end{equation}
where $\Pi_{T}$ denotes the $\mathbf{L}^2$--projection onto piecewise \EO{quadratic} functions over ${T}$.

\begin{theorem}[\FF{local estimates for} $\mathcal{E}_{ad}$]\label{thm:local_eff_ad}
Assume that assumptions \eqref{eq:smallness_assumption} \FF{and \eqref{eq:assumption_y_T_small_2} hold}. Let $(\bar{\mathbf{y}},\bar{p},\bar{\mathbf{z}},\bar{r},\bar{\mathbf{u}}) \in \mathbf{H}_0^1(\Omega)\times L_0^2(\Omega) \times \mathbf{H}_0^1(\Omega)\times L_0^2(\Omega) \times \mathbb{U}_{ad}$  be a local solution of \eqref{eq:min_weak_setting}--\eqref{eq:weak_st_equation}. Let $\bar{\mathbf{u}}_{\T}$ be a local minimum of the associated discrete optimal control problem with $(\bar{\mathbf{y}}_{\T},\bar{p}_\T)$ and $(\bar{\mathbf{z}}_{\T},\bar{r}_\T)$ being the corresponding state and adjoint state discrete variables, respectively. Then, for $T\in\T$, $\mathcal{E}_{ad,T}$ \EO{satisfies
\begin{multline}\label{eq:local_eff_ad}
\mathcal{E}_{ad,T}
\lesssim 
\|\nabla\mathbf{e}_{\mathbf{z}}\|_{\mathbf{L}^2(\mathcal{N}_T)}
+
\|\nabla\mathbf{e}_{\mathbf{y}}\|_{\mathbf{L}^2(\mathcal{N}_T)}
\\
+
h_T^{-\frac{1}{2}}\|\mathbf{e}_{\mathbf{z}}\|_{\mathbf{L}^2(\mathcal{N}_T)}
+
h_T^{-\frac{1}{2}}\|\mathbf{e}_{\mathbf{y}}\|_{\mathbf{L}^2(\mathcal{N}_T)}
+
\|e_r\|_{L^2(\mathcal{N}_{T})}+\mathrm{osc}_{\mathcal{N}_T}(\mathbf{y}_\Omega),
\end{multline}
where} $\mathcal{N}_T$ and $\mathrm{osc}_{\mathcal{N}_T}(\mathbf{y}_\Omega)$ are defined as in \eqref{def:patch} and \eqref{eq:osc_term}, respectively. The hidden constant is independent of the continuous and discrete optimal variables, the size of the elements in the mesh $\T$, and $\#\T$.
\end{theorem}
\begin{proof} 
Since the pair $(\bar{\mathbf{z}},\bar{r})\in \mathbf{H}_0^1(\Omega)\times L_0^2(\Omega)$ solves \eqref{eq:weak_adjoint_equation}, an elementwise integration by parts formula yields the identity
\begin{multline}\label{eq:error_eq_ad}
\nu(\nabla\mathbf{w}, \nabla \mathbf{e}_{\mathbf{z}})_{\mathbf{L}^2(\Omega)}+b(\mathbf{e}_{\mathbf{y}};\mathbf{w},\bar{\mathbf{z}})+b(\bar{\mathbf{y}}_\T;\mathbf{w},\mathbf{e}_\mathbf{z})+b(\mathbf{w};\mathbf{e}_{\mathbf{y}},\bar{\mathbf{z}})+b(\mathbf{w};\bar{\mathbf{y}}_{\T},\mathbf{e}_{\mathbf{z}})\\
-(e_r,\text{div }\mathbf{w})_{L^2(\Omega)}+(s,\text{div }\mathbf{e}_{\mathbf{z}})_{L^2(\Omega)}-(\mathbf{e}_{\mathbf{y}},\mathbf{w})_{\mathbf{L}^2(\Omega)}=\sum_{T\in\T}\bigg( (\Pi_T(\mathbf{y}_\Omega)-\mathbf{y}_\Omega,\mathbf{w})_{\mathbf{L}^2(T)}\\
+\left(\bar{\mathbf{y}}_\T-\Pi_T(\mathbf{y}_\Omega)+\nu\Delta\bar{\mathbf{z}}_\T-(\nabla \bar{\mathbf{y}}_\T)^{\intercal}\bar{\mathbf{z}}_\T+(\bar{\mathbf{y}}_\T\cdot \nabla)\bar{\mathbf{z}}_\T-\nabla \bar{r}_\T,\mathbf{w}\right)_{\mathbf{L}^2(T)}\\
-(s,\text{div }\bar{\mathbf{z}}_\T)_{L^2(T)}\bigg)+\sum_{S\in\Sides}\left(\llbracket (\nu\nabla \bar{\mathbf{z}}_\T-\bar{r}_\T\mathbb{I}_{d} )\cdot \mathbf{n}\rrbracket,\mathbf{w}\right)_{\mathbf{L}^2(S)},
\end{multline}
which holds for every $\mathbf{w}\in\mathbf{H}_0^1(\Omega)$ and $s\in L_0^2(\Omega)$. With equation \eqref{eq:error_eq_ad} at hand, in the following steps, we will estimate separately each of the individual terms that appear in the definition of $\mathcal{E}_{ad,T}$.

We now proceed on the basis of four steps.

\emph{Step 1.} Let $T\in \T$. Define
\begin{equation*}
\begin{aligned}
\mathbf{R}_{T}^{ad} & :=\left(\bar{\mathbf{y}}_\T-\mathbf{y}_\Omega+\nu\Delta\bar{\mathbf{z}}_\T-(\nabla \bar{\mathbf{y}}_\T)^{\intercal}\bar{\mathbf{z}}_\T+(\bar{\mathbf{y}}_\T\cdot \nabla)\bar{\mathbf{z}}_\T-\nabla \bar{r}_\T\right)|^{}_T,
\\
\hat{\mathbf{R}}_{T}^{ad} & :=\left(\bar{\mathbf{y}}_\T-\Pi_T(\mathbf{y}_\Omega)+\nu\Delta\bar{\mathbf{z}}_\T-(\nabla \bar{\mathbf{y}}_\T)^{\intercal}\bar{\mathbf{z}}_\T+(\bar{\mathbf{y}}_\T\cdot \nabla)\bar{\mathbf{z}}_\T-\nabla \bar{r}_\T\right)|^{}_T.
\end{aligned}
\end{equation*}
We estimate the residual term $h^2_T\| \mathbf{R}_{T}^{ad} \|_{\mathbf{L}^2(T)}^2$ in \eqref{eq:local_error_indicator_ad}. We begin with a simple application of the triangle inequality to obtain 
\begin{equation}\label{eq:estimate_Rad_triangular}
h_T\|\mathbf{R}_{T}^{ad}\|_{\mathbf{L}^2(T)}
\leq
h_T\| \hat{\mathbf{R}}_{T}^{ad}\|_{\mathbf{L}^2(T)}+\mathrm{osc}_{T}(\mathbf{y}_\Omega).
\end{equation}
\EO{It thus suffices to control the term $h_T\|\hat{\mathbf{R}}_{T}^{ad}\|_{\mathbf{L}^2(T)}$. We proceed differently according to the spatial dimension.}

\EO{Let $d=2$. Set $\mathbf{w} = \varphi_T^2\hat{\mathbf{R}}_{T}^{ad}$} and $s=0$ in identity \eqref{eq:error_eq_ad}, utilize standard properties of the bubble function $\varphi_T$ and \EO{inverse inequalities \cite[Lemma 4.5.3]{MR2373954} to arrive at
\begin{multline}\label{eq:eff_ad_I_new}
\|\hat{\mathbf{R}}_{T}^{ad}\|^2_{\mathbf{L}^2(T)} 
\lesssim
\| \varphi_T \mathbf{e}_{\mathbf{y}}\|_{\mathbf{L}^4(T)}
( \| \nabla \varphi_T \hat{\mathbf{R}}_{T}^{ad} \|_{\mathbf{L}^2(T)}  + \| \varphi_T \nabla \hat{\mathbf{R}}_{T}^{ad} \|_{\mathbf{L}^2(T)})
\|\bar{\mathbf{z}}\|_{\mathbf{L}^{4}(T)} 
\\
\| \varphi_T \mathbf{e}_{\mathbf{z}}\|_{\mathbf{L}^4(T)}
( \| \nabla \varphi_T \hat{\mathbf{R}}_{T}^{ad} \|_{\mathbf{L}^2(T)}  + \| \varphi_T \nabla \hat{\mathbf{R}}_{T}^{ad} \|_{\mathbf{L}^2(T)})\|\bar{\mathbf{y}}_{\T}\|_{\mathbf{L}^{4}(T)} 
\\
+ \| \varphi_T^2 \hat{\mathbf{R}}_{T}^{ad} \|_{\mathbf{L}^\infty(T)} \left( \| \nabla \mathbf{e}_{\mathbf{y}} \|_{\mathbf{L}^2(T)}
\| \bar{\mathbf{z}} \|_{\mathbf{L}^2(T)} + \| \nabla \bar{\mathbf{y}}_{\T} \|_{\mathbf{L}^2(T)}
\| \mathbf{e}_{\mathbf{z}} \|_{\mathbf{L}^2(T)} \right)
\\
+\left(\|\nabla\mathbf{e}_{\mathbf{z}}\|_{\mathbf{L}^2(T)}+\|e_r\|_{L^2(T)}\right) \| \nabla (\varphi_T^2 \hat{\mathbf{R}}_{T}^{ad}) \|_{\mathbf{L}^2(T)}  
\\
+\left( \|\mathbf{e}_{\mathbf{y}}\|_{\mathbf{L}^2(T)}+\|\Pi_T(\mathbf{y}_\Omega)-\mathbf{y}_\Omega\|_{\mathbf{L}^2(T)} \right)  \| \varphi_T^2 \hat{\mathbf{R}}_{T}^{ad} \|_{\mathbf{L}^2(T)}  .
\end{multline}
Observe that $\| \varphi_T \mathbf{e}_{\mathbf{y}}\|_{\mathbf{L}^4(T)} \lesssim h_T^{\frac{1}{2}}\| \nabla(\varphi_T \mathbf{e}_{\mathbf{y}})\|_{\mathbf{L}^2(T)} \lesssim h_T^{\frac{1}{2}}\| \nabla \mathbf{e}_{\mathbf{y}} \|_{\mathbf{L}^2(T)} +  h_T^{-\frac{1}{2}}\| \mathbf{e}_{\mathbf{y}} \|_{\mathbf{L}^2(T)}$, upon utilizing \cite[Lemma II.3.2 and inequality (II.5.5)]{Gal11}. On the other hand, basic bubble function arguments yield the estimates
\[ 
\| \nabla \varphi_T \hat{\mathbf{R}}_{T}^{ad} \|_{\mathbf{L}^2(T)}  + \| \varphi_T \nabla \hat{\mathbf{R}}_{T}^{ad} \|_{\mathbf{L}^2(T)} \lesssim h_T^{-1}\| \hat{\mathbf{R}}_{T}^{ad} \|_{\mathbf{L}^2(T)}  
\]
and $\| \varphi_T^2 \hat{\mathbf{R}}_{T}^{ad} \|_{\mathbf{L}^\infty(T)} \lesssim h_T^{-1}\|  \hat{\mathbf{R}}_{T}^{ad} \|_{\mathbf{L}^2(T)}$. These estimates allow us to obtain the bound
\begin{multline}
h_T^2 \|\hat{\mathbf{R}}_{T}^{ad}\|^2_{\mathbf{L}^2(T)} \lesssim (h_T + 1) \| \nabla \mathbf{e}_{\mathbf{y}} \|^2_{\mathbf{L}^2(T)} +  (h_T^{-1} + h_T^2)\| \mathbf{e}_{\mathbf{y}} \|^2_{\mathbf{L}^2(T)}
\\
+
(h_T + 1)\| \nabla \mathbf{e}_{\mathbf{z}} \|^2_{\mathbf{L}^2(T)} 
+  (1+h_T^{-1})\| \mathbf{e}_{\mathbf{z}} \|^2_{\mathbf{L}^2(T)} + \|e_r\|_{L^2(T)}^2 + \mathrm{osc}_{T}^2(\mathbf{y}_\Omega).
\label{eq:aux_residual_Rad}
\end{multline}
To obtain the previous estimate, we have also used the Sobolev embedding $\mathbf{H}_0^1(\Omega)\hookrightarrow \mathbf{L}^4(\Omega)$, the smallness assumption \eqref{eq:assumption_y_y_T} and} the stability estimate \eqref{eq:bound_bar_z}, which yields
\begin{equation}\label{eq:eff_ad_II}
\|\bar{\mathbf{z}}\|_{\mathbf{L}^2(T)}\leq \|\bar{\mathbf{z}}\|_{\mathbf{L}^2(\Omega)} \leq C_{2}\|\nabla \bar{\mathbf{z}}\|_{\mathbf{L}^2(\Omega)}\leq \frac{ C_2^2}{\nu(1-\theta)}\left(C_2\theta\mathcal{C}_{b}^{-1}\nu+\|\mathbf{y}_\Omega\|_{\mathbf{L}^2(\Omega)}\right).
\end{equation}

\EO{We now analyze the case $d=3$. Similarly, we set $\mathbf{w} = \varphi_T^2\hat{\mathbf{R}}_{T}^{ad}$ and $s=0$ in identity \eqref{eq:error_eq_ad} and obtain
\begin{multline}\label{eq:eff_ad_IV_new}
\|\hat{\mathbf{R}}_{T}^{ad}\|^2_{\mathbf{L}^2(T)} 
\lesssim
\| \varphi_T \mathbf{e}_{\mathbf{y}}\|_{\mathbf{L}^3(T)}
( \| \nabla \varphi_T \hat{\mathbf{R}}_{T}^{ad} \|_{\mathbf{L}^2(T)}  + \| \varphi_T \nabla \hat{\mathbf{R}}_{T}^{ad} \|_{\mathbf{L}^2(T)})
\|\bar{\mathbf{z}}\|_{\mathbf{L}^{6}(T)} 
\\
\| \varphi_T \mathbf{e}_{\mathbf{z}}\|_{\mathbf{L}^3(T)}
( \| \nabla \varphi_T \hat{\mathbf{R}}_{T}^{ad} \|_{\mathbf{L}^2(T)}  + \| \varphi_T \nabla \hat{\mathbf{R}}_{T}^{ad} \|_{\mathbf{L}^2(T)})\|\bar{\mathbf{y}}_{\T}\|_{\mathbf{L}^{6}(T)} 
\\
+ \| \varphi_T \hat{\mathbf{R}}_{T}^{ad} \|_{\mathbf{L}^6(T)} \left( \| \nabla \mathbf{e}_{\mathbf{y}} \|_{\mathbf{L}^2(T)}
\| \varphi_T\bar{\mathbf{z}} \|_{\mathbf{L}^3(T)} + \| \nabla \bar{\mathbf{y}}_{\T} \|_{\mathbf{L}^2(T)}
\| \varphi_T\mathbf{e}_{\mathbf{z}} \|_{\mathbf{L}^3(T)} \right)
\\
+\left(\|\nabla\mathbf{e}_{\mathbf{z}}\|_{\mathbf{L}^2(T)}+\|e_r\|_{L^2(T)}\right) \| \nabla (\varphi_T^2 \hat{\mathbf{R}}_{T}^{ad}) \|_{\mathbf{L}^2(T)}  
\\
+\left( \|\mathbf{e}_{\mathbf{y}}\|_{\mathbf{L}^2(T)}+\|\Pi_T(\mathbf{y}_\Omega)-\mathbf{y}_\Omega\|_{\mathbf{L}^2(T)} \right)  \| \varphi_T^2 \hat{\mathbf{R}}_{T}^{ad} \|_{\mathbf{L}^2(T)}  .
\end{multline}
Observe that $\| \varphi_T^2 \hat{\mathbf{R}}_{T}^{ad} \|_{\mathbf{L}^6(T)}  \lesssim \|  \hat{\mathbf{R}}_{T}^{ad} \|_{\mathbf{L}^6(T)} \lesssim h_T^{-1} \| \hat{\mathbf{R}}_{T}^{ad} \|_{\mathbf{L}^2(T)}$ \cite[Lemma 4.5.3]{MR2373954}. On the other hand, since $\varphi_T \mathbf{e}_{\mathbf{z}}\in \mathbf{H}_0^1(T)$, we invoke the Cauchy--Schwarz inequality and the Sobolev embedding result of \cite[inequality (II.3.7)]{Gal11}, to conclude that 
\begin{align*}
\|\varphi_T\mathbf{e}_{\mathbf{z}}\|_{\mathbf{L}^3(T)}
\leq 
|T|^{\frac{1}{6}}\|\varphi_T\mathbf{e}_{\mathbf{z}}\|_{\mathbf{L}^6(T)} 
&\lesssim
h_T^{\frac{1}{2}}\|\nabla(\varphi_T\mathbf{e}_{\mathbf{z}})\|_{\mathbf{L}^2(T)} \\
&\lesssim
h_T^{-\frac{1}{2}}\|\mathbf{e}_{\mathbf{z}}\|_{\mathbf{L}^2(T)} + h_T^{\frac{1}{2}}\|\nabla\mathbf{e}_{\mathbf{z}}\|_{\mathbf{L}^2(T)}.
\end{align*}
A similar estimate holds for $\| \varphi_T \mathbf{e}_{\mathbf{y}}\|_{\mathbf{L}^3(T)}$. With these estimates at hand, similar arguments to the ones that lead to \eqref{eq:aux_residual_Rad}} \EO{allow us to conclude that
\begin{multline}\label{eq:final_residual_ad_d=3}
h_T^2\|\hat{\mathbf{R}}_{T}^{ad}\|_{\mathbf{L}^2(T)}^2
\lesssim
\|\nabla\mathbf{e}_{\mathbf{z}}\|_{\mathbf{L}^2(T)}^2+\|\nabla\mathbf{e}_{\mathbf{y}}\|_{\mathbf{L}^2(T)}^2 
\\
+ h_T^{-1}\|\mathbf{e}_{\mathbf{z}}\|_{\mathbf{L}^2(T)}^2+h_T^{-1}\|\mathbf{e}_{\mathbf{y}}\|_{\mathbf{L}^2(T)}^2
+\|e_r\|_{L^2(T)}^2+\mathrm{osc}_{T}^2(\mathbf{y}_\Omega).
\end{multline}
}
\EO{The desired estimate follows from gathering \eqref{eq:estimate_Rad_triangular} with  \eqref{eq:aux_residual_Rad} and \eqref{eq:final_residual_ad_d=3}}.

\emph{Step 2.} Let $T\in\T$ and $S\in\Sides_T$. We bound $h_T\|\llbracket (\nu\nabla \bar{\mathbf{z}}_\T-\bar{r}_\T\mathbb{I}_{d} )\cdot \mathbf{n}\rrbracket\|_{\mathbf{L}^2(S)}^2$ in \eqref{eq:local_error_indicator_ad}. To simplify the presentation of the material, we define
\[
\mathbf{J}_{S}^{ad}:= \llbracket (\nu\nabla \bar{\mathbf{z}}_\T-\bar{r}_\T\mathbb{I}_{d} )\cdot \mathbf{n}\rrbracket.
\]
\FF{As in the previous step, we proceed differently according to the spatial dimension.} \EO{If $d=2$, we set $\mathbf{w}=\varphi_{S}^2\mathbf{J}_{S}^{ad}$} and $s=0$ in \eqref{eq:error_eq_ad} and proceed on the basis of similar arguments to the ones used to derive \eqref{eq:eff_ad_I_new}. These arguments yield
\begin{multline*}
\|\mathbf{J}_{S}^{ad}\|_{\mathbf{L}^2(S)}^2
\lesssim
\sum_{T'\in\mathcal{N}_S}
\bigg(
\FF{h_{T'}^{-1}\left(\|\varphi_S\mathbf{e}_{\mathbf{y}}\|_{\mathbf{L}^4(T')}\|\bar{\mathbf{z}}\|_{\mathbf{L}^4(T')} + \|\bar{\mathbf{y}}_\T\|_{\mathbf{L}^4(T')}\|\varphi_S\mathbf{e}_{\mathbf{z}}\|_{\mathbf{L}^4(T')}
\right)}
\\
\FF{+h_{T'}^{-1}\left(\|\nabla\mathbf{e}_{\mathbf{y}}\|_{\mathbf{L}^2(T')}\|\bar{\mathbf{z}}\|_{\mathbf{L}^2(T')}+\|\nabla\bar{\mathbf{y}}_\T\|_{\mathbf{L}^2(T')}\|\mathbf{e}_{\mathbf{z}}\|_{\mathbf{L}^2(T')}+\|\nabla\mathbf{e}_{\mathbf{z}}\|_{\mathbf{L}^2(T')}+\|e_r\|_{L^2(T')}\right)}\\
+\|\mathbf{e}_{\mathbf{y}}\|_{\mathbf{L}^2(T')}+ \|\hat{\mathbf{R}}_{T'}^{ad}\|_{\mathbf{L}^2(T')} + \|\Pi_{T'}(\mathbf{y}_\Omega)-\mathbf{y}_\Omega\|_{\mathbf{L}^2(T')}\bigg)h_T^{\frac{1}{2}}\|\mathbf{J}_S^{ad}\|_{\mathbf{L}^2(S)}.
\end{multline*}
Invoke the \EO{estimates that lead to \eqref{eq:aux_residual_Rad} to obtain
\begin{multline}\label{eq:final_jump_ad}
h_T\|\mathbf{J}_{S}^{ad}\|_{\mathbf{L}^2(S)}^2
\lesssim
\sum_{T'\in\mathcal{N}_S}\left( (h_{T'} + 1) \| \nabla \mathbf{e}_{\mathbf{y}} \|^2_{\mathbf{L}^2(T')} +  (h_{T'}^{-1} + h_{T'}^2)\| \mathbf{e}_{\mathbf{y}} \|^2_{\mathbf{L}^2(T')}\right.
\\
\left.+
(h_{T'} + 1)\| \nabla \mathbf{e}_{\mathbf{z}} \|^2_{\mathbf{L}^2(T')} 
+  (1+h_{T'}^{-1})\| \mathbf{e}_{\mathbf{z}} \|^2_{\mathbf{L}^2(T')} + \|e_r\|_{L^2(T')}^2 + \mathrm{osc}_{T'}^2(\mathbf{y}_\Omega)\right).
\end{multline}
In three dimensions, we follow similar arguments. For brevity we skip the details.}

\emph{Step 3.} Let $T\in \T$. Since $\text{div }\bar{\mathbf{z}} = 0$, we immediately obtain that
\begin{equation}\label{eq:final_div_ad}
\|\text{div }\bar{\mathbf{z}}_\T\|_{L^2(T)}
\EO{=} 
\|\text{div }\mathbf{e}_\mathbf{z}\|_{L^2(T)}
\lesssim 
\EO{\| \nabla \mathbf{e}_\mathbf{z}\|_{\mathbf{L}^2(T)}.}
\end{equation}

\emph{Step 4.} The proof concludes by gathering \FF{the estimates obtained in the previous steps.}
\end{proof}

The results obtained in Theorems \ref{thm:local_eff_st} and \ref{thm:local_eff_ad} yield \FF{local estimates for}
\begin{align}\label{def:indicator_ocp}
\mathcal{E}_{ocp,T}^2:= \mathcal{E}_{ad,T}^{2} + \mathcal{E}_{st,T}^{2} + \mathcal{E}_{ct,T}^{2}.
\end{align}

\begin{theorem}[\FF{local estimates for} $\mathcal{E}_{ocp,T}$]
\label{thm:global_eff}
Assume that assumptions \eqref{eq:smallness_assumption} \FF{and \eqref{eq:assumption_y_T_small_2} hold}. Let $(\bar{\mathbf{y}},\bar{p},\bar{\mathbf{z}},\bar{r},\bar{\mathbf{u}}) \in \mathbf{H}_0^1(\Omega)\times L_0^2(\Omega) \times \mathbf{H}_0^1(\Omega)\times L_0^2(\Omega) \times \mathbb{U}_{ad}$ be a local solution of \eqref{eq:min_weak_setting}--\eqref{eq:weak_st_equation}. Let $\bar{\mathbf{u}}_{\T}$ be a local minimum of the associated discrete optimal control problem with $(\bar{\mathbf{y}}_{\T},\bar{p}_\T)$ and $(\bar{\mathbf{z}}_{\T},\bar{r}_\T)$ being the corresponding state and adjoint state discrete variables, respectively. Then, for $T\in\T$, we have that
\begin{multline*}
\mathcal{E}_{ocp,T} 
\lesssim
\|\nabla\mathbf{e}_{\mathbf{z}}\|_{\mathbf{L}^2(\mathcal{N}_T)}+\|\nabla \mathbf{e}_{\mathbf{y}}\|_{\mathbf{L}^2(\mathcal{N}_T)} 
+
\EO{h_T^{-\frac{1}{2}}\|\mathbf{e}_{\mathbf{z}}\|_{\mathbf{L}^2(\mathcal{N}_T)}+h_T^{-\frac{1}{2}}\|\mathbf{e}_{\mathbf{y}}\|_{\mathbf{L}^2(\mathcal{N}_T)}}
\\
+  \|\mathbf{e}_{\mathbf{u}}\|_{\mathbf{L}^2(\mathcal{N}_{T})}+
\|e_p\|_{L^2(\mathcal{N}_{T})}+\|e_r\|_{L^2(\mathcal{N}_{T})}+\mathrm{osc}_{\mathcal{N}_T}(\mathbf{y}_\Omega),
\end{multline*}
where $\mathcal{N}_T$ and $\mathrm{osc}_{\mathcal{N}_T}(\mathbf{y}_\Omega)$ are defined as in \eqref{def:patch} and \eqref{eq:osc_term}, respectively. The hidden constant is independent of the continuous and discrete optimal variables, the size of the elements in the mesh $\T$, and $\#\T$.
\end{theorem}
\begin{proof}
Let $T\in\T$. In view of the local estimates \eqref{eq:local_eff_st} and \eqref{eq:local_eff_ad}, it  suffices to bound $\mathcal{E}_{ct,T}$. Invoke \eqref{eq:error_estimator_control_var} and an application of the triangle inequality to obtain
\begin{align*}
\mathcal{E}_{ct,T} 
& \leq \|\tilde{\mathbf{u}}-\bar{\mathbf{u}}\|_{\mathbf{L}^2(T)} + \|\mathbf{e}_{\mathbf{u}}\|_{\mathbf{L}^2(T)}\\
& = \|\Pi_{[\mathbf{a},\mathbf{b}]}(-\alpha^{-1}\bar{\mathbf{z}}_\T)-\Pi_{[\mathbf{a},\mathbf{b}]}(-\alpha^{-1}\bar{\mathbf{z}})\|_{\mathbf{L}^2(T)}+\|\mathbf{e}_{\mathbf{u}}\|_{\mathbf{L}^2(T)}.
\end{align*}
Invoke the Lipschitz property of $\Pi_{[\mathbf{a},\mathbf{b}]}$, introduced in \eqref{def:projection_operator}, to obtain
\begin{equation}\label{eq:local_eff_ct}
\mathcal{E}_{ct,T}
\leq
\alpha^{-1}\|\mathbf{e}_{\mathbf{z}}\|_{\mathbf{L}^2(T)}+\|\mathbf{e}_{\mathbf{u}}\|_{\mathbf{L}^2(T)}\EO{.}
\end{equation}
The proof concludes by collecting estimates \eqref{eq:local_eff_st}, \eqref{eq:local_eff_ad}, and \eqref{eq:local_eff_ct}.
\end{proof}


\EO{
\section{A posteriori error analysis: the semi discrete scheme}\label{sec:a_posteriori_semi}

In this section, we propose and analyze an a posteriori error estimator for the semi discrete scheme \eqref{eq:discrete_cost_var}--\eqref{eq:discrete_st_variational}. In contrast to the estimator devised for the fully discrete scheme, the a posteriori error estimator is decomposed only in two contributions: one related to the discretization of the state equations and another one related to the discretization of the adjoint equations. 

In order to guarantee the existence of a local solution $(\bar{\mathbf{y}}_\T,\bar{p}_\T,\bar{\mathbf{z}}_\T,\bar{r}_\T,\bar{\mathbf{g}})\in \mathbf{V}(\T) \times \mathcal{P}(\T) \times \mathbf{V}(\T)\times \mathcal{P}(\T) \times \mathbb{U}_{ad}$ to the semi discrete scheme, satisfying the optimality system \eqref{eq:discrete_st_variational}--\eqref{eq:discrete_adjoint_equation_variational}, we shall assume assumptions \eqref{eq:smallness_assumption}, \ref{A1}, \eqref{eq:assumption_y_T_small_2}, and \ref{A3}.


\subsection{Reliability analysis}
Assume that \eqref{eq:smallness_assumption} holds. To perform an analysis, we introduce some auxiliary variables. Let $(\hat{\mathbf{y}},\hat{p}) \in \mathbf{H}_0^1(\Omega)\times L_0^2(\Omega)$ be the solution to
\begin{equation}\label{eq:hat_function_st_var}
\begin{aligned}
\nu(\nabla \hat{\mathbf{y}}, \nabla \mathbf{v})_{\mathbf{L}^2(\Omega)}+b(\hat{\mathbf{y}};\hat{\mathbf{y}},\mathbf{v})-(\hat{p},\text{div } \mathbf{v})_{L^2(\Omega)} & = (\bar{\mathbf{g}},\mathbf{v})_{\mathbf{L}^2(\Omega)}  & \forall \mathbf{v} \in \mathbf{H}_0^1(\Omega),\\
(q,\text{div } \hat{\mathbf{y}})_{L^2(\Omega)} & =  0  & \forall q\in L^2_0(\Omega).
\end{aligned}
\end{equation}
Notice that $(\bar{\mathbf{y}}_{\T},\bar{p}_{\T})$, which solves \eqref{eq:discrete_st_variational} with $\mathbf{g}$ replaced by $\bar{\mathbf{g}}$, can be seen as the finite element approximation within the space $\mathbf{V}(\mathscr{T}) \times \mathcal{P}(\T)$ of $(\hat{\mathbf{y}},\hat{p})$.

Define the following a posteriori error estimator and local error indicators:
\begin{gather}\label{eq:error_estimator_st_variational}
\mathfrak{E}_{st}^{2}:= \sum_{T\in\T}\mathfrak{E}_{st,T}^2, 
\qquad
\mathfrak{E}_{st,T}^{2}:=h_T^2\|\bar{\mathbf{g}}+\nu\Delta\bar{\mathbf{y}}_\T-(\bar{\mathbf{y}}_\T\cdot \nabla)\bar{\mathbf{y}}_\T-\nabla \bar{p}_\T\|_{\mathbf{L}^2(T)}^2 \\\nonumber
+ \|\text{div }\bar{\mathbf{y}}_\T\|_{L^2(T)}^2
+ h_T\|\llbracket (\nu\nabla \bar{\mathbf{y}}_\T-\bar{p}_\T\mathbb{I}_{d} )\cdot \mathbf{n}\rrbracket\|_{\mathbf{L}^2(\partial T \setminus \partial \Omega)}^2.
\end{gather}
Here $(\bar{\mathbf{y}}_{\T},\bar{p}_{\T})=(\bar{\mathbf{y}}_{\T}(\bar{\mathbf{g}}),\bar{p}_{\T}(\bar{\mathbf{g}}))$ denotes the solution to \eqref{eq:discrete_st_variational} with $\mathbf{g}$ replaced by $\bar{\mathbf{g}}$. With these ingredients at hand, Theorem \ref{thm:navier_stokes_estimator} yields the following global reliability bound for the error estimator $\mathfrak{E}_{st}$: If \eqref{eq:assumption_y_y_T} holds, then
\begin{equation}\label{eq:estimate_state_hat_discrete_var}
\|\nabla (\hat{\mathbf{y}}-\bar{\mathbf{y}}_\T)\|_{\mathbf{L}^2(\Omega)}^2+\|\hat{p}-\bar{p}_\T\|_{L^2(\Omega)}^2 
\lesssim
\mathfrak{E}_{st}^2.
\end{equation}

Let $(\hat{\mathbf{z}},\hat{r}) \in \mathbf{H}_0^1(\Omega)\times L_0^2(\Omega)$ be the solution to
\begin{align}\label{eq:hat_function_ad_var}
\begin{split}
\nu(\nabla \mathbf{w}, \nabla \hat{\mathbf{z}})_{\mathbf{L}^2(\Omega)}+b(\bar{\mathbf{y}}_\T;\mathbf{w},\hat{\mathbf{z}})
+b(\mathbf{w};\bar{\mathbf{y}}_\T,\hat{\mathbf{z}})
\\-(\hat{r},\text{div } \mathbf{w})_{L^2(\Omega)} & = (\bar{\mathbf{y}}_\T-\mathbf{y}_\Omega,\mathbf{w})_{\mathbf{L}^2(\Omega)}, \\
(s,\text{div } \hat{\mathbf{z}})_{L^2(\Omega)}   &=  0, 
\end{split}
\end{align}
for all $(\mathbf{w},s) \in \mathbf{H}_0^1(\Omega)\times  L_0^2(\Omega)$. Here, $\bar{\mathbf{y}}_{\T}=\bar{\mathbf{y}}_{\T}(\bar{\mathbf{g}})\in \mathbf{V}(\T)$ denotes the discrete velocity field that solves \eqref{eq:discrete_st_variational} with $\mathbf{g}$ replaced by $\bar{\mathbf{g}}$. Under assumption \eqref{eq:assumption_y_T_small_2}, problem \eqref{eq:hat_function_ad_var} is well--posed. On the other hand, notice that $(\bar{\mathbf{z}}_\T,\bar{r}_\T)$, the solution to \eqref{eq:discrete_adjoint_equation_variational}, can be seen as the finite element approximation, within $\mathbf{V}(\T)\times\mathcal{P}(\T)$, of $(\hat{\mathbf{z}},\hat{r})$.

Define, for $T\in\T$, the local error indicators 
\begin{multline}\label{eq:local_error_indicator_ad_var}
\mathfrak{E}_{ad,T}^2:=h_T^2\|\bar{\mathbf{y}}_\T-\mathbf{y}_\Omega+\nu\Delta\bar{\mathbf{z}}_\T-(\nabla \bar{\mathbf{y}}_\T)^{\intercal}\bar{\mathbf{z}}_\T+(\bar{\mathbf{y}}_\T\cdot \nabla)\bar{\mathbf{z}}_\T-\nabla \bar{r}_\T\|_{\mathbf{L}^2(T)}^2\\
+h_T\|\llbracket (\nu\nabla \bar{\mathbf{z}}_\T-\bar{r}_\T\mathbb{I}_{d} )\cdot \mathbf{n}\rrbracket\|_{\mathbf{L}^2(\partial T \setminus \partial \Omega)}^2+\|\text{div }\bar{\mathbf{z}}_\T\|_{L^2(T)}^2
\end{multline}
and the a posteriori error estimator
\begin{equation}\label{eq:error_estimator_adjoint_eq_var}
\mathfrak{E}_{ad}^2:=\sum_{T\in\T}\mathfrak{E}_{ad,T}^2.
\end{equation}
With these ingredients at hand, we are in position to invoke Theorem \ref{thm:estimate_aux_variables} to immediately arrive at the following global reliability bound:
\begin{equation}\label{eq:adjoint_hat_estimate_var}
\|\nabla(\hat{\mathbf{z}}-\bar{\mathbf{z}}_\T)\|_{\mathbf{L}^2(\Omega)}^2+\|\hat{r}-\bar{r}_\T\|_{L^2(\Omega)}^2\lesssim \mathfrak{E}_{ad}^2.
\end{equation}

In what follows, we present a global reliability result for the error estimator $\mathfrak{E}_{ocp}$ based on the semi discrete scheme \eqref{eq:discrete_cost_var}--\eqref{eq:discrete_st_variational}. To present such a result, we define $\mathfrak{e}_{\mathbf{u}} := \bar{\mathbf{u}}-\bar{\mathbf{g}}$ and the total error norm associated to the semi discrete scheme:
\begin{equation}\label{def:error_norm_variational}
\|\mathfrak{e} \|^2_{\Omega}: =
\|\nabla \mathbf{e}_{\mathbf{y}}\|_{\mathbf{L}^2(\Omega)}^2 +
\|e_p \|_{L^{2}(\Omega)}^2 +
\|\nabla \mathbf{e}_{\mathbf{z}} \|_{\mathbf{L}^2(\Omega)}^2
+ \|e_r\|_{L^2(\Omega)}^2
+ \|\mathfrak{e}_{\mathbf{u}}\|_{\mathbf{L}^2(\Omega)}^2.
\end{equation}

\begin{theorem}[global reliability of $\mathfrak{E}_{ocp}$]
\label{thm:control_bound_variational}
Assume that the smallness assumptions \eqref{eq:smallness_assumption} and \eqref{eq:assumption_y_T_small_2} hold. Let $(\bar{\mathbf{y}},\bar{p},\bar{\mathbf{z}},\bar{r},\bar{\mathbf{u}}) \in \mathbf{H}_0^1(\Omega)\times L_0^2(\Omega) \times \mathbf{H}_0^1(\Omega)\times L_0^2(\Omega) \times \mathbb{U}_{ad}$ be a local solution of \eqref{eq:min_weak_setting}--\eqref{eq:weak_st_equation} that satisfies the sufficient second order optimality condition \eqref{eq:equivalent_second_I}, or equivalently \eqref{eq:equivalent_second_II}. Let $\bar{\mathbf{g}}$ be a local minimum of the semi discrete control problem \eqref{eq:discrete_cost_var}--\eqref{eq:discrete_st_variational}, with $(\bar{\mathbf{y}}_{\T},\bar{p}_\T)$ and $(\bar{\mathbf{z}}_{\T},\bar{r}_\T)$ being the corresponding state and adjoint state variables, respectively. Let $\T$ be a mesh such that \eqref{eq:assumption_mesh} holds. Then,
\begin{equation}\label{eq:global_rel_var}
\|\mathfrak{e}  \|_{\Omega}^2
\lesssim 
\mathfrak{E}_{ocp}^2,
\end{equation}
where 
\begin{equation}\label{def:error_estimator_ocp_var}
\mathfrak{E}_{ocp}^2:=  \mathfrak{E}_{ad}^2 + \mathfrak{E}_{st}^2.
\end{equation}
The estimators $\mathfrak{E}_{st}$ and $\mathfrak{E}_{ad}$ are defined as in \eqref{eq:error_estimator_st_variational} and \eqref{eq:error_estimator_adjoint_eq_var}, respectively. In \eqref{eq:global_rel_var}, the hidden constant is independent of the continuous and discrete optimal variables, the size of the elements in the mesh $\T$, and $\#\T$.
\end{theorem}

\begin{proof}
Within the setting of the variational discretization approach, we have that $\bar{\mathbf{g}}=\tilde{\mathbf{u}}$, where $\tilde{\mathbf{u}}$ is defined in \eqref{def:tilde_u}. This fact immediately implies that $\hat{\mathbf{y}}=\tilde{\mathbf{y}}$, with $\hat{\mathbf{y}}$ and $\tilde{\mathbf{y}}$ given as the solutions to \eqref{eq:hat_function_st_var} and \eqref{eq:tilde_y}, respectively. Consequently, in view of \eqref{eq:control_estimate_III} and \eqref{eq:tilde_hat_adjoint_V}, we can conclude that
\begin{equation*}
\|\mathfrak{e}_{\mathbf{u}}\|_{\mathbf{L}^2(\Omega)} = \|\bar{\mathbf{u}}-\tilde{\mathbf{u}}\|_{\mathbf{L}^2(\Omega)} 
\lesssim 
\|\nabla(\tilde{\mathbf{z}}-\hat{\mathbf{z}})\|_{\mathbf{L}^2(\Omega)}+ \mathfrak{E}_{ad}
\lesssim
\mathfrak{E}_{st}
+ \mathfrak{E}_{ad}.
\end{equation*}
The control of the remaining terms in \eqref{def:error_norm_variational} follows by utilizing the bound $\|\mathfrak{e}_{\mathbf{u}}\|_{\mathbf{L}^2(\Omega)}\lesssim \mathfrak{E}_{ocp}$ and the arguments developed in the proof of Theorem \ref{thm:control_bound}. For brevity we skip the details. 
\end{proof}


\subsection{Efficiency analysis}

The local estimates obtained in Theorems \ref{thm:local_eff_st} and \ref{thm:local_eff_ad} can also be obtained within the setting of the variational discretization approach. In fact, having derived these results, it can be immediately deduced a local estimate for the local error indicator
\begin{align}\label{def:indicator_ocp_variational}
\mathfrak{E}_{ocp,T}^2:= \mathfrak{E}_{ad,T}^{2} + \mathfrak{E}_{st,T}^{2}.
\end{align}

We present the following result.

\begin{theorem}[local estimates for $\mathfrak{E}_{ocp,T}$]
\label{thm:global_eff_var}  
Assume that the smallness assumptions \eqref{eq:smallness_assumption} and \eqref{eq:assumption_y_T_small_2} hold. Let $(\bar{\mathbf{y}},\bar{p},\bar{\mathbf{z}},\bar{r},\bar{\mathbf{u}}) \in \mathbf{H}_0^1(\Omega)\times L_0^2(\Omega) \times \mathbf{H}_0^1(\Omega)\times L_0^2(\Omega) \times \mathbb{U}_{ad}$ be a local solution of \eqref{eq:min_weak_setting}--\eqref{eq:weak_st_equation}. Let $\bar{\mathbf{g}}$ be a local minimum of the semi discrete optimal control problem \eqref{eq:discrete_cost_var}--\eqref{eq:discrete_st_variational}, with $(\bar{\mathbf{y}}_{\T},\bar{p}_\T)$ and $(\bar{\mathbf{z}}_{\T},\bar{r}_\T)$ being the corresponding state and adjoint state discrete variables, respectively. Then, for $T\in\T$, we have
\begin{multline*}
\mathfrak{E}_{ocp,T} 
\lesssim
\|\nabla\mathbf{e}_{\mathbf{z}}\|_{\mathbf{L}^2(\mathcal{N}_T)}+\|\nabla \mathbf{e}_{\mathbf{y}}\|_{\mathbf{L}^2(\mathcal{N}_T)} 
+h_T^{-\frac{1}{2}}\|\mathbf{e}_{\mathbf{z}}\|_{\mathbf{L}^2(\mathcal{N}_T)}+h_T^{-\frac{1}{2}}\|\mathbf{e}_{\mathbf{y}}\|_{\mathbf{L}^2(\mathcal{N}_T)}
\\
+ h_T\|\mathbf{e}_{\mathbf{u}}\|_{\mathbf{L}^2(\mathcal{N}_{T})}+
\|e_p\|_{L^2(\mathcal{N}_{T})}+\|e_r\|_{L^2(\mathcal{N}_{T})}+\mathrm{osc}_{\mathcal{N}_T}(\mathbf{y}_\Omega),
\end{multline*}
where $\mathcal{N}_T$ and $\mathrm{osc}_{\mathcal{N}_T}(\mathbf{y}_\Omega)$ are defined as in \eqref{def:patch} and \eqref{eq:osc_term}, respectively. The hidden constant is independent of the continuous and discrete optimal variables, the size of the elements in the mesh $\T$, and $\#\T$.
\end{theorem}
\begin{proof}
The proof follows the same arguments developed in the proof of Theorem \ref{thm:global_eff}. For brevity, we skip details.
\end{proof}
}


\section{Numerical examples}\label{sec:numerical_ex}

\DQ{
In this section we conduct numerical experiments for the fully discrete scheme of section \ref{sec:fully_discrete_scheme} and the semi discrete one of section \ref{sec:semi_discrete_scheme} and illustrate the performance of the devised a posteriori error estimators $\mathcal{E}_{ocp}$ and $\mathfrak{E}_{ocp}$ defined in \eqref{def:error_estimator_ocp} and \eqref{def:error_estimator_ocp_var}, respectively.

\subsection{Implementation}

The presented numerical examples have been carried out with the help of a code that we implemented using \texttt{C++}. The involved global linear systems were solved by using the multifrontal massively parallel sparse direct solver (MUMPS) \cite{MUMPS1,MUMPS2}. The right hand sides, the approximation errors, and the error estimators were computed by a quadrature formula which is exact for polynomials of degree nineteen $(19)$ for two dimensional domains and degree fourteen $(14)$ for three dimensional domains. 

In what follows, we discuss implementation details for each numerical scheme. We recall that the discrete spaces $\mathbf{V}(\T)$, $\mathcal{P}(\T)$, and $\mathbb{U}(\T)$ are defined by \eqref{def:velocity_space}, \eqref{def:pressure_space}, and \eqref{def:control_space}, respectively.
\begin{itemize}
\item[(i)] \textbf{The fully discrete scheme}: For a given partition $\mathscr{T}$, we seek a discrete solution $(\bar{\mathbf{y}}^{}_\mathscr{T},\bar{p}^{}_\mathscr{T},\bar{\mathbf{z}}^{}_\mathscr{T},\bar{r}^{}_\mathscr{T},\bar{\mathbf{u}}^{}_\mathscr{T})\,\in \mathbf{V}(\mathscr{T}) \times \mathcal{P}(\mathscr{T}) \times \mathbf{V}(\mathscr{T}) \times \mathcal{P}(\mathscr{T}) \times \mathbb{U}_{ad}(\T)$ that solves \eqref{eq:discrete_state_equation}--\eqref{eq:discrete_adjoint_equation}. System \eqref{eq:discrete_state_equation}--\eqref{eq:discrete_adjoint_equation} is solved by using a primal--dual active set strategy \cite[section 2.12.4]{Troltzsch} combined with a fixed point method: for each active set iteration, the ensuing nonlinear system is solved by using a fixed point algorithm. All matrices involved in the left hand side of the resulting linear system have been assembled exactly. 

The total number of degrees of freedom for the fully discrete scheme is $\mathsf{Ndof}=2\left[ \dim(\mathbf{V}(\T)) + \dim(\mathcal{P}(\mathscr{T})) \right] + \dim(\mathbb{U}(\T))$. We measure the error in the norm $\|\mathbf{e}\|_{\Omega}$, which is defined in \eqref{def:error_norm}. Finally,  we introduce the effectivity index $\mathcal{I}:= \mathcal{E}_{ocp}/\|\mathbf{e}\|_{\Omega}$.

\item[(ii)] \textbf{The semi discrete scheme}: For a given partition $\mathscr{T}$, we seek a solution $(\bar{\mathbf{y}}^{}_\mathscr{T},\bar{p}^{}_\mathscr{T},\bar{\mathbf{z}}^{}_\mathscr{T},\bar{r}^{}_\mathscr{T})\,\in \mathbf{V}(\mathscr{T}) \times \mathcal{P}(\mathscr{T}) \times \mathbf{V}(\mathscr{T}) \times \mathcal{P}(\mathscr{T})$ that solves system \eqref{eq:discrete_st_variational}--\eqref{eq:discrete_adjoint_equation_variational}. This system is solved by using an adaptation of the Newton method presented in \cite[Section 3]{MR2867194}. We notice that the numerical implementation of the variational discretization scheme requires the assembling and exact computation of $(\bar{\mathbf{g}},\mathbf{v}_{\mathscr{T}})_{\mathbf{L}^{2}(\Omega)}$. It is thus particularly needed the exact integration of such a term on the elements $T\in\T$ where the control $\bar{\mathbf{g}}$ exhibits kinks. An alternative to accomplish this task is as follows:
\begin{itemize}
\item[S1)] identify the elements $T \in \T$ where the control $\bar{\mathbf{g}}$ have kinks.
\item[S2)] identify the regions of such elements $T$ where the control variable is active/inactive. Notice that, since we are using Taylor--Hood finite elements, these regions have, in general, curved boundaries.
\item[S3)] compute the term $(\bar{\mathbf{g}},\mathbf{v}_{\mathscr{T}})_{\mathbf{L}^{2}(T)}$ by decomposing the integral on the aforementioned regions.
\end{itemize}
We notice that the computational implementation of S2) and S3) is far from being trivial. In view of this fact, we will compute the term $(\bar{\mathbf{g}},\mathbf{v}_{\mathscr{T}})_{\mathbf{L}^{2}(T)}$ with the help of different quadrature formulas. We stress that this numerical implementation leads to an approximated version of the variational discretization approach and immediately mention that, for the different quadrature formulas that we consider, the devised adaptive loops deliver optimal experimental rates of convergence for all the involved variables.

The total number of degrees of freedom for the semi discrete scheme is $\mathsf{Ndof}=2\left[ \dim(\mathbf{V}(\T)) + \dim(\mathcal{P}(\mathscr{T})) \right]$. The error is measured in the norm $\|\mathfrak{e}\|_{\Omega}$ defined in \eqref{def:error_norm_variational}. Finally, we introduce the effectivity index $\mathfrak{I}:= \mathfrak{E}_{ocp}/\|\mathfrak{e}\|_{\Omega}$. 
\end{itemize}

Once the discrete solution is obtained, we compute, for $T\in \T$, the error indicators $\mathcal{E}_{ocp,T}$ and $\mathfrak{E}_{ocp,T}$, defined in \eqref{def:indicator_ocp} and \eqref{def:indicator_ocp_variational}, respectively, to drive the adaptive mesh refinement procedure described in Algorithm \ref{Algorithm1}. A sequence of adaptively refined meshes is thus generated from the initial meshes shown in Figure \ref{fig:initial_meshes}. To visualize such meshes we have used the open--source application ParaView \cite{Ahrens2005ParaViewAE,Ayachit2015ThePG}.

To simplify the construction of exact solutions, we incorporate an extra source term $\mathbf{f} \in \mathbf{L}^{\infty}(\Omega)$ in the right hand side of the momentum equation of \eqref{eq:weak_st_equation}. With such a modification, the right hand side of the first equation in \eqref{eq:weak_st_equation} now reads $(\mathbf{f}+\mathbf{u},\mathbf{v})_{\mathbf{L}^2(\Omega)}$.

\begin{algorithm}[ht]
\caption{\textbf{Adaptive algorithm.}}
\label{Algorithm1}
\textbf{Input:} Initial mesh $\mathscr{T}_{0}$, fluid viscosity $\nu$, desired state $\mathbf{y}_\Omega$, external source $\mathbf{f}$, constraints $\mathbf{a}$ and $\mathbf{b}$, and regularization parameter $\alpha$;
\\
\textbf{Set:} $i=0$.
\\
\DQ{$\boldsymbol{1}$: Choose an initial discrete guess $(\mathbf{y}_{\T_{i}}^{0},p_{\T_{i}}^{0},\mathbf{z}_{\T_{i}}^{0},r_{\T_{i}}^{0}) \in \mathbf{V}(\mathscr{T}_{i}) \times \mathcal{P}(\mathscr{T}_{i}) \times \mathbf{V}(\mathscr{T}_{i}) \times \mathcal{P}(\mathscr{T}_{i})$;
\\
$\boldsymbol{2}$: (i) For the fully discrete scheme choose, in addition, $\mathbf{u}_{\T_{i}}^{0} \in \mathbb{U}(\T_{i})$ and compute $[\bar{\mathbf{y}}_{\T_{i}},\bar{p}_{\T_{i}},\bar{\mathbf{z}}_{\T_{i}},\bar{r}_{\T_{i}},\bar{\mathbf{u}}_{\T_{i}}]=\textbf{Active-Set}[\mathscr{T}_i,\nu,\mathbf{y}_\Omega,\mathbf{f},\mathbf{a},\mathbf{b},\alpha,\mathbf{y}_{\T_{i}}^{0},p_{\T_{i}}^{0},\mathbf{z}_{\T_{i}}^{0},r_{\T_{i}}^{0},\mathbf{u}_{\T_{i}}^{0}]$. \textbf{Active-Set} implements the active set strategy of \cite[section 2.12.4]{Troltzsch}; for each active set iteration, the ensuing nonlinear system is solved by using a fixed point method;

\hspace{0.5cm} (ii) For the semi discrete scheme, compute $[\bar{\mathbf{y}}_{\T_{i}},\bar{p}_{\T_{i}},\bar{\mathbf{z}}_{\T_{i}},\bar{r}_{\T_{i}}] =$\\ $\textbf{Newton-Method}[\mathscr{T}_i,\nu,\mathbf{y}_\Omega,\mathbf{f},\mathbf{a},\mathbf{b},\alpha,\mathbf{y}_{\T_{i}}^{0},p_{\T_{i}}^{0},\mathbf{z}_{\T_{i}}^{0},r_{\T_{i}}^{0}]$. \textbf{Newton-Method} implements an adaptation of the numerical algorithm presented in \cite[Section 3]{MR2867194}; }
\\
\textbf{Adaptive loop:}
\\
$\boldsymbol{3}$: For each $T \in \mathscr{T}_i$ compute the local error indicator $\mathcal{E}_{ocp,T}$ \DQ{($\mathfrak{E}_{ocp,T}$)} defined in \eqref{def:indicator_ocp} \DQ{(\eqref{def:indicator_ocp_variational})};
\\
$\boldsymbol{4}$: Mark an element $T \in \T_i$ for refinement if \DQ{
$$\mathcal{E}_{ocp,T}^{2}> \frac{1}{2}\max_{T'\in \mathscr{T}_i}\mathcal{E}_{ocp,T'}^{2} \left( \mathfrak{E}_{ocp,T}^{2}> \frac{1}{2}\max_{T'\in \mathscr{T}_i}\mathfrak{E}_{ocp,T'}^{2} \right); $$
}
\\
$\boldsymbol{5}$: From step $\boldsymbol{4}$, construct a new mesh $\mathscr{T}_{i+1}$, using a longest edge bisection algorithm. Set $i \leftarrow i + 1$ and go to step $\boldsymbol{1}$.
\end{algorithm}

\begin{figure}[!ht]
\centering
\begin{minipage}{0.310\textwidth}\centering
\includegraphics[trim={0 0 0 0},clip,width=1.5cm,height=1.5cm,scale=0.2]{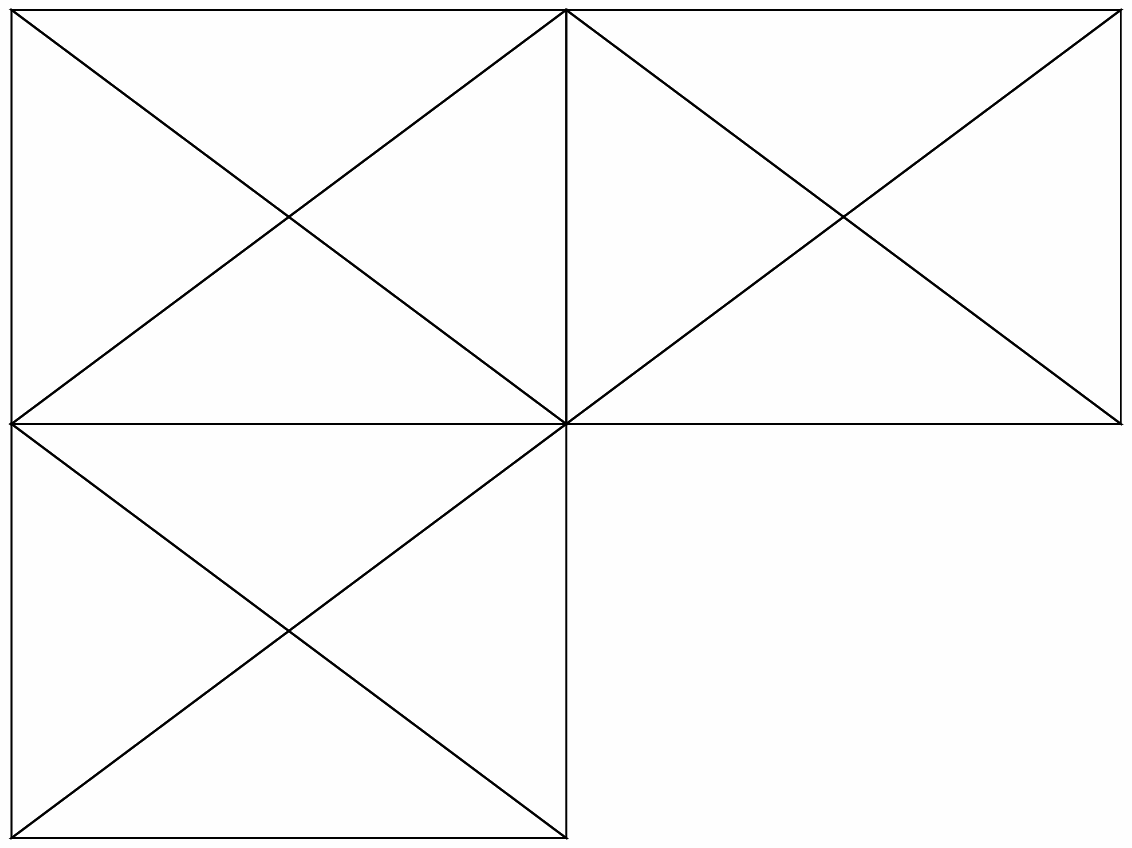}\\
\end{minipage}
\begin{minipage}{0.310\textwidth}\centering
\includegraphics[trim={0 0 0 0},clip,width=1.5cm,height=1.5cm,scale=0.2]{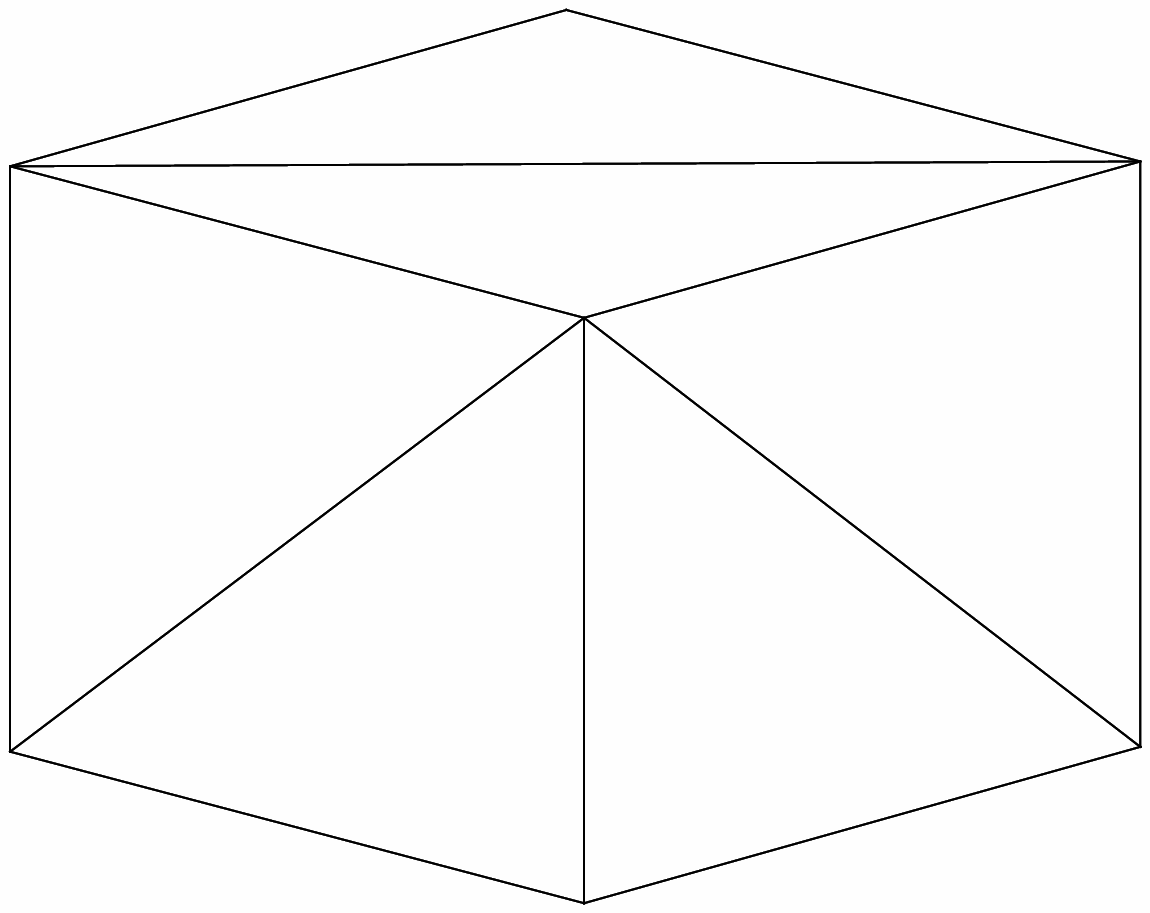}\\
\end{minipage}
\caption{The initial meshes used when the domain $\Omega$ is a $L$-shaped domain (Example 1)  and a cube (Example 2).}
\label{fig:initial_meshes}
\end{figure}

We now provide two numerical experiments where the exact solutions are known. The first example is posed on a two-dimensional L-shaped domain and operate under inhomogeneous Dirichlet boundary conditions for the state and adjoint equations. We notice that this violates the assumption of homogeneous Dirichlet boundary conditions but it retains essential difficulties and allows us to evaluate experimental rates of convergence. The second example is posed on a three-dimensional domain and involves homogeneous Dirichlet boundary conditions for the state and adjoint equations.

\subsection{Example 1 (two dimensional L--shaped domain)}

We set 
\[
\Omega = (-1,1)^{2} \setminus [0,1)\times(-1,0],
\] 
$\mathbf{a} = (-2,-2)$, $\mathbf{b} = (2,2)$, $\alpha = 10^{-4}$, and $\nu = 1$. The optimal state and adjoint state are given, in polar coordinates $(\rho,\vartheta)$, by
\[ \bar{\mathbf{y}}(\rho,\vartheta) = 10^{-2}\rho^{\sigma}\begin{pmatrix}
(1+\sigma)\sin(\vartheta)\psi(\vartheta) + \cos(\vartheta)\psi'(\vartheta) \\
-(1+\sigma)\cos(\vartheta)\psi(\vartheta) + \sin(\vartheta)\psi'(\vartheta)
\end{pmatrix}, \quad \bar{\mathbf{z}}(\rho,\vartheta) = \bar{\mathbf{y}}(\rho,\vartheta) - \begin{pmatrix}
2 \\
2
\end{pmatrix} \]
\[\bar{p}(\rho,\vartheta) = \bar{r}(\rho,\vartheta) = \frac{1}{1-\sigma}\rho^{\sigma - 1}\left((1 + \sigma)^{2})\psi'(\vartheta) + \psi'''(\vartheta) \right),\]
\[\psi(\vartheta) = \left(\frac{\sin((1+\sigma)\vartheta)}{1 + \sigma} + \frac{\sin((\sigma-1)\vartheta)}{\sigma - 1}\right)\cos(\gamma \sigma) - \cos((1 + \sigma)\vartheta) + \cos((\sigma - 1)\vartheta),\]
where $\vartheta \in [0,3\pi/2]$, $\sigma = 856399/1572864$, and $\gamma = 3\pi/2$.

The purpose of this example is twofold. First, we compare uniform versus adaptive refinement by utilizing the devised error estimators $\mathcal{E}_{ocp}$ and $\mathfrak{E}_{ocp}$ within the adaptive schemes of Algorithm \ref{Algorithm1}. Second, within the setting of the semi discrete scheme, we quantify the effect of utilizing different integration rules in the assembling procedure when computing $(\bar{\mathbf{g}},\mathbf{v}_{\mathscr{T}})_{\mathbf{L}^{2}(\Omega)}$ and suitable derivatives involved in the Newton method. We consider three different implementations, which we describe in what follows.
\begin{itemize}
\item[$\bullet$] We first use a quadrature formula which, in two dimensions, is exact for polynomials of degree nineteen $(19)$. To simplify the presentation of the results, we use $\mathfrak{S}_{19}$ to identify the results obtained under this particular implementation.
\item[$\bullet$] Second, we use a quadrature formula which, in two dimensions, is exact for polynomials of degree five $(5)$. We use $\mathfrak{S}_{5}$ to identify the results obtained under this particular implementation.
\item[$\bullet$] Third, we consider a composed quadrature: in the elements $T \in \T$ where the control exhibits kinks we use a quadrature exact for polynomials of degree five, whereas in the remaining elements we utilize a quadrature formula which is exact for polynomials of degree nineteen. We use $\mathfrak{S}_{5c}$ to identify the results obtained under this particular implementation.
\end{itemize}
We also use $\mathfrak{F}$ to identify the results obtained within the setting of the fully discrete scheme. 


\begin{figure}[!ht]
\centering
\psfrag{error ct}{{\large $\| \mathfrak{e}_{\mathfrak{u}}\|_{\mathbf{L}^{2}(\Omega)}$}}
\psfrag{ndofs 03}{{\normalsize $\mathsf{Ndof}^{-0.3}$}}
\psfrag{ndofs 12}{{\normalsize $\mathsf{Ndof}^{-1/2}$}}
\psfrag{Ndofs}{{\large $\mathsf{Ndof}$}}
\psfrag{73}{{\large $\mathfrak{S}_{19}$}}
\psfrag{5e}{{\large $\mathfrak{S}_{5}$}}
\psfrag{5c}{{\large $\mathfrak{S}_{5c}$}}
\psfrag{fd}{{$\mathfrak{F}$}}
\begin{minipage}[c]{0.4\textwidth}\centering
\psfrag{errores totales}{\hspace{0.1cm}\large{$\|\mathbf{e}\|_{\Omega}$ and $\|\mathfrak{e}\|_{\Omega}$}}
\includegraphics[trim={0 0 0 0},clip,width=5.05cm,height=3.2cm,scale=0.30]{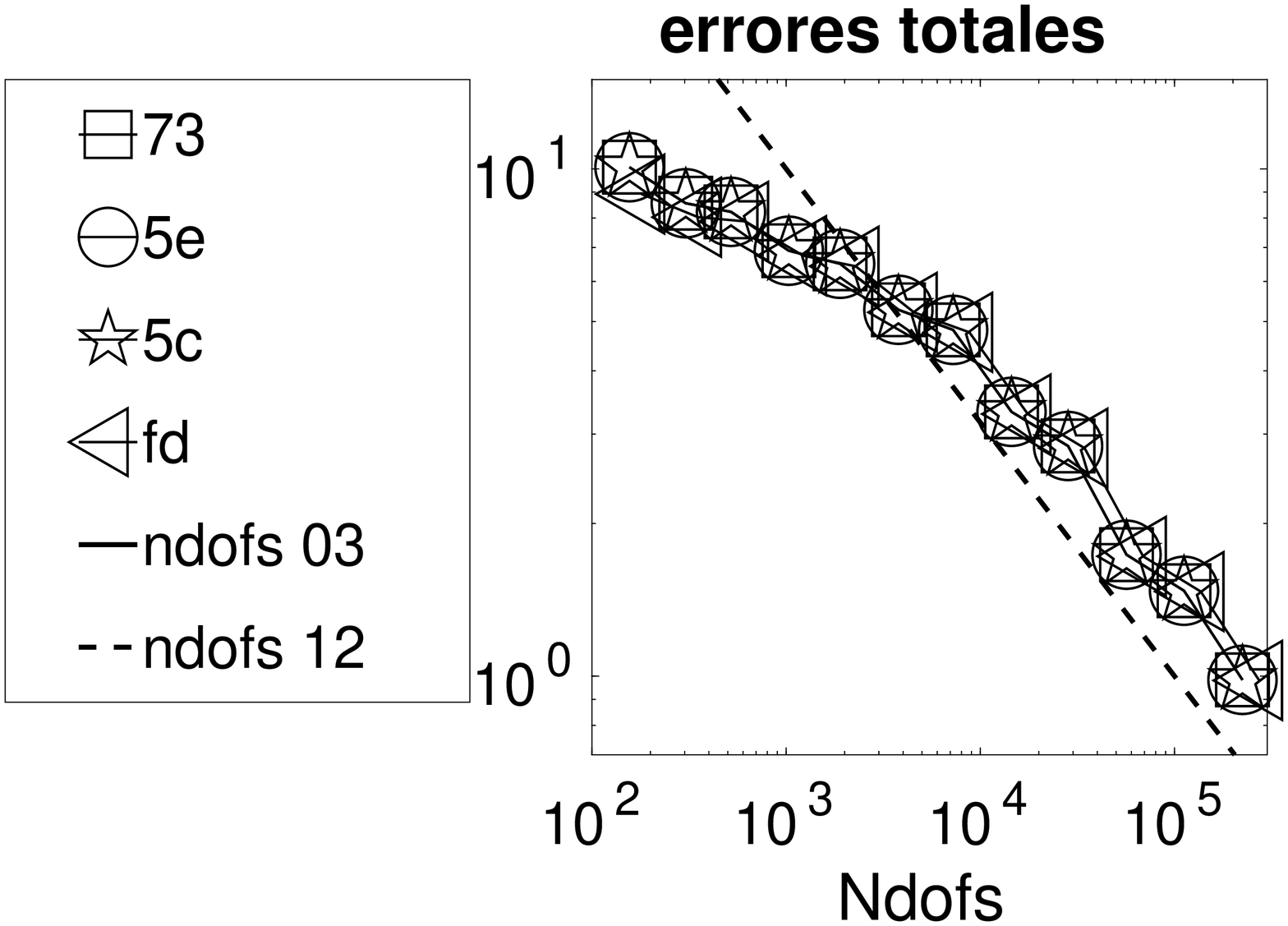}\\
\hspace{1.9cm}\tiny{(A.1)}
\end{minipage}
\begin{minipage}[c]{0.3\textwidth}\centering
\psfrag{errors ct}{\hspace{-1.1cm}\large{$\|\mathbf{e}_{\mathbf{u}}\|_{\mathbf{L}^{2}(\Omega)}$ and $\|\mathfrak{e}_{\mathbf{u}}\|_{\mathbf{L}^{2}(\Omega)}$}}
\includegraphics[trim={0 0 0 0},clip,width=3.15cm,height=3.2cm,scale=0.30]{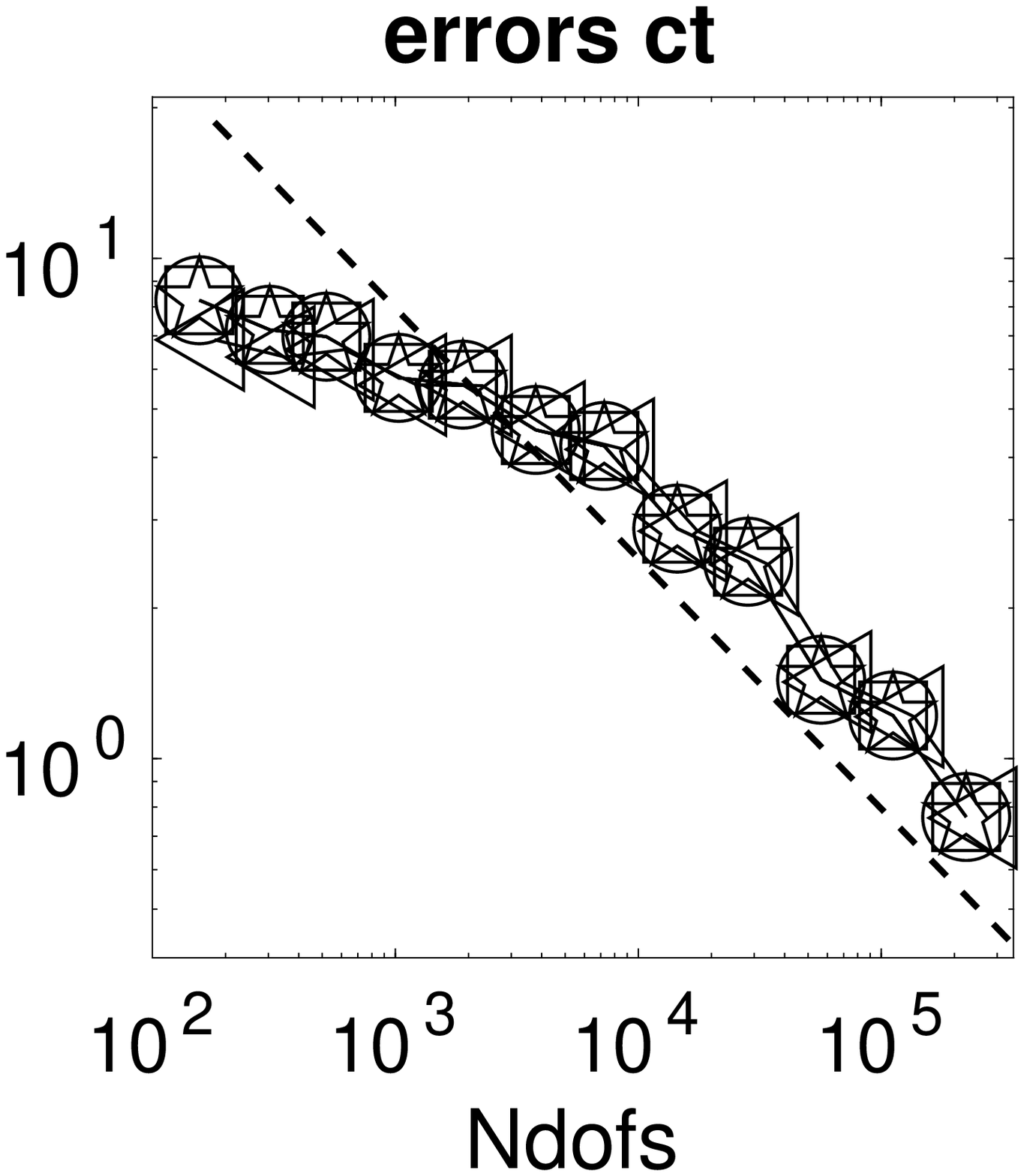}\\
\hspace{0.2cm}\tiny{(A.2)}
\end{minipage}
\\
\begin{minipage}[c]{0.248\textwidth}\centering
\psfrag{errors st}{\hspace{0.0cm}\large{$\|\nabla \mathbf{e}_{\mathbf{y}}\|_{\mathbf{L}^{2}(\Omega)}$}}
\includegraphics[trim={0 0 0 0},clip,width=3.15cm,height=3.2cm,scale=0.30]{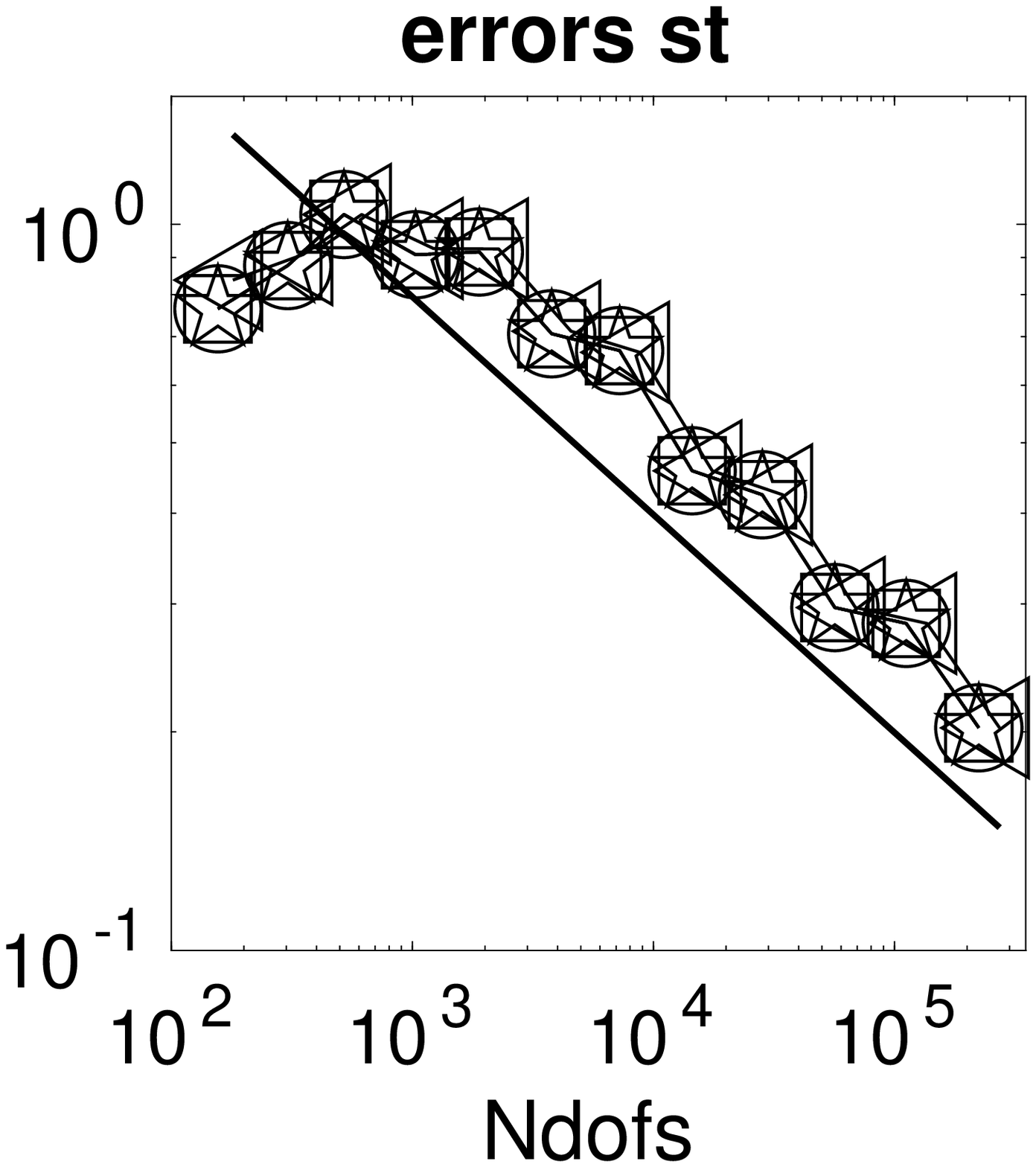}\\
\hspace{0.2cm}\tiny{(A.3)}
\end{minipage}
\begin{minipage}[c]{0.248\textwidth}\centering
\psfrag{errors st}{\hspace{0.4cm}\large{$\|e_{p}\|_{L^{2}(\Omega)}$}}
\includegraphics[trim={0 0 0 0},clip,width=3.15cm,height=3.2cm,scale=0.30]{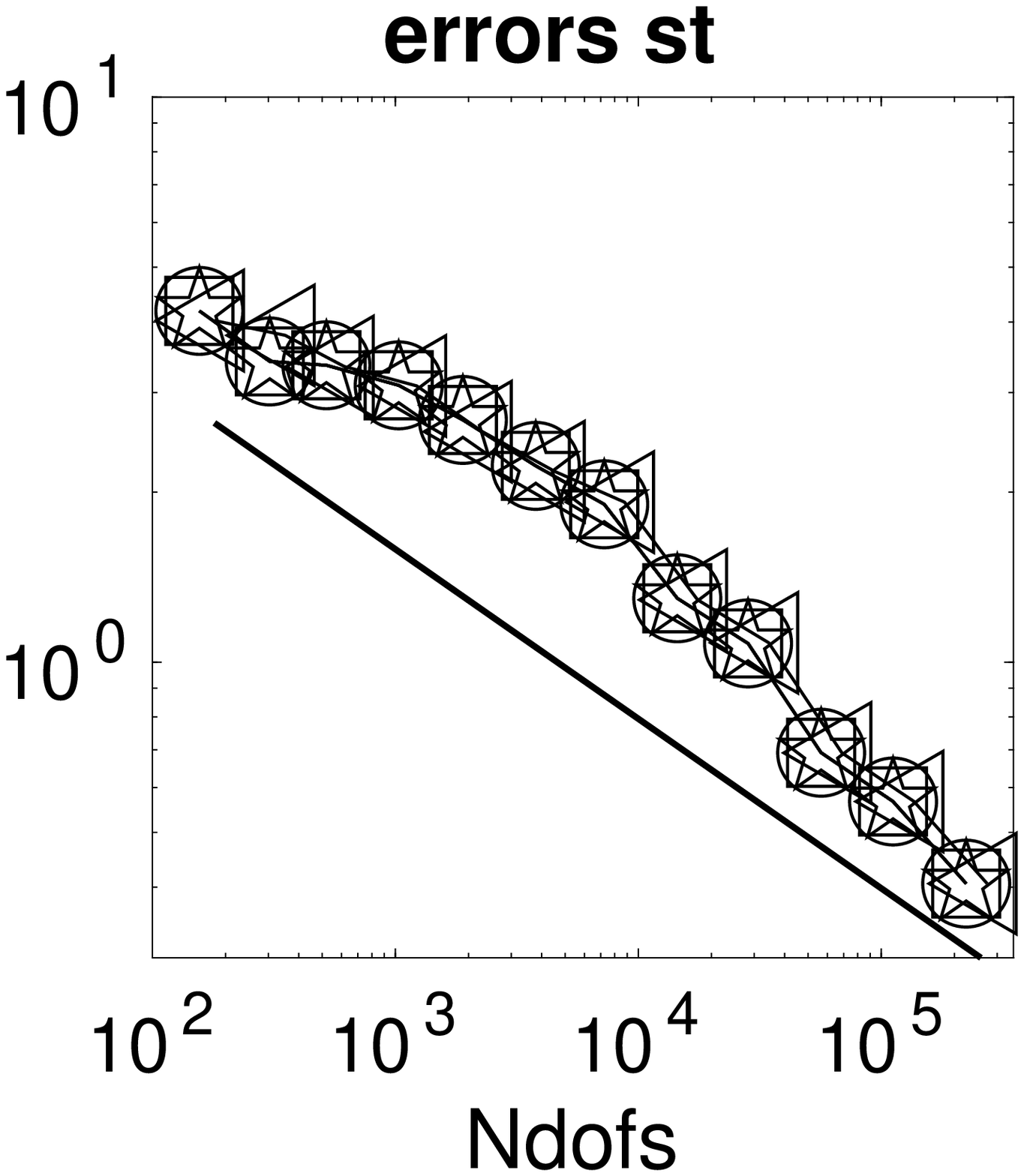}\\
\hspace{0.2cm}\tiny{(A.4)}
\end{minipage}
\begin{minipage}[c]{0.248\textwidth}\centering
\psfrag{errors ad}{\hspace{0.0cm}\large{$\|\nabla \mathbf{e}_{\mathbf{z}}\|_{\mathbf{L}^{2}(\Omega)}$}}
\includegraphics[trim={0 0 0 0},clip,width=3.15cm,height=3.2cm,scale=0.30]{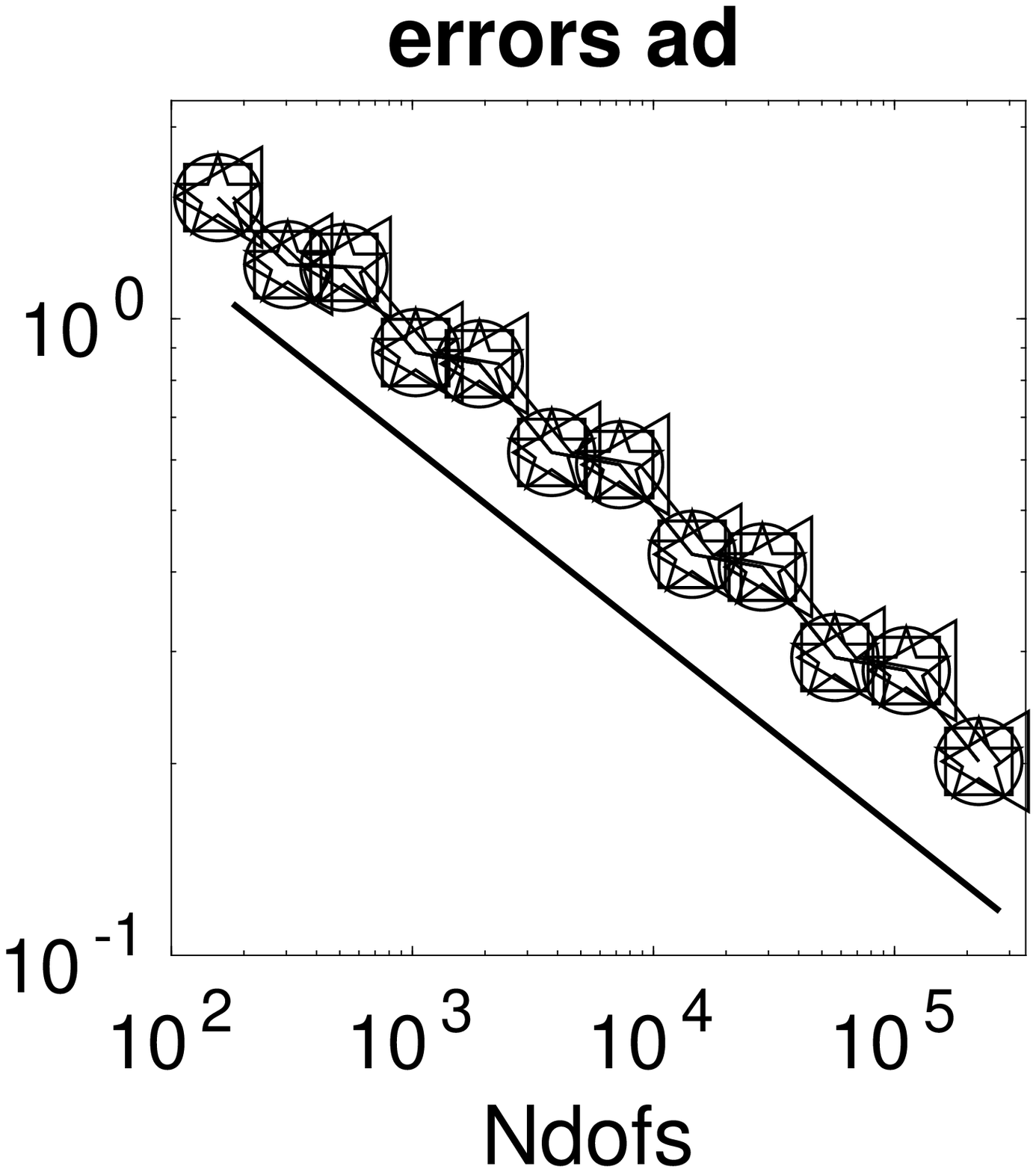}\\
\hspace{0.2cm}\tiny{(A.5)}
\end{minipage}
\begin{minipage}[c]{0.248\textwidth}\centering
\psfrag{errors ad}{\hspace{0.4cm}\large{$\|e_{r}\|_{L^{2}(\Omega)}$}}
\includegraphics[trim={0 0 0 0},clip,width=3.15cm,height=3.2cm,scale=0.30]{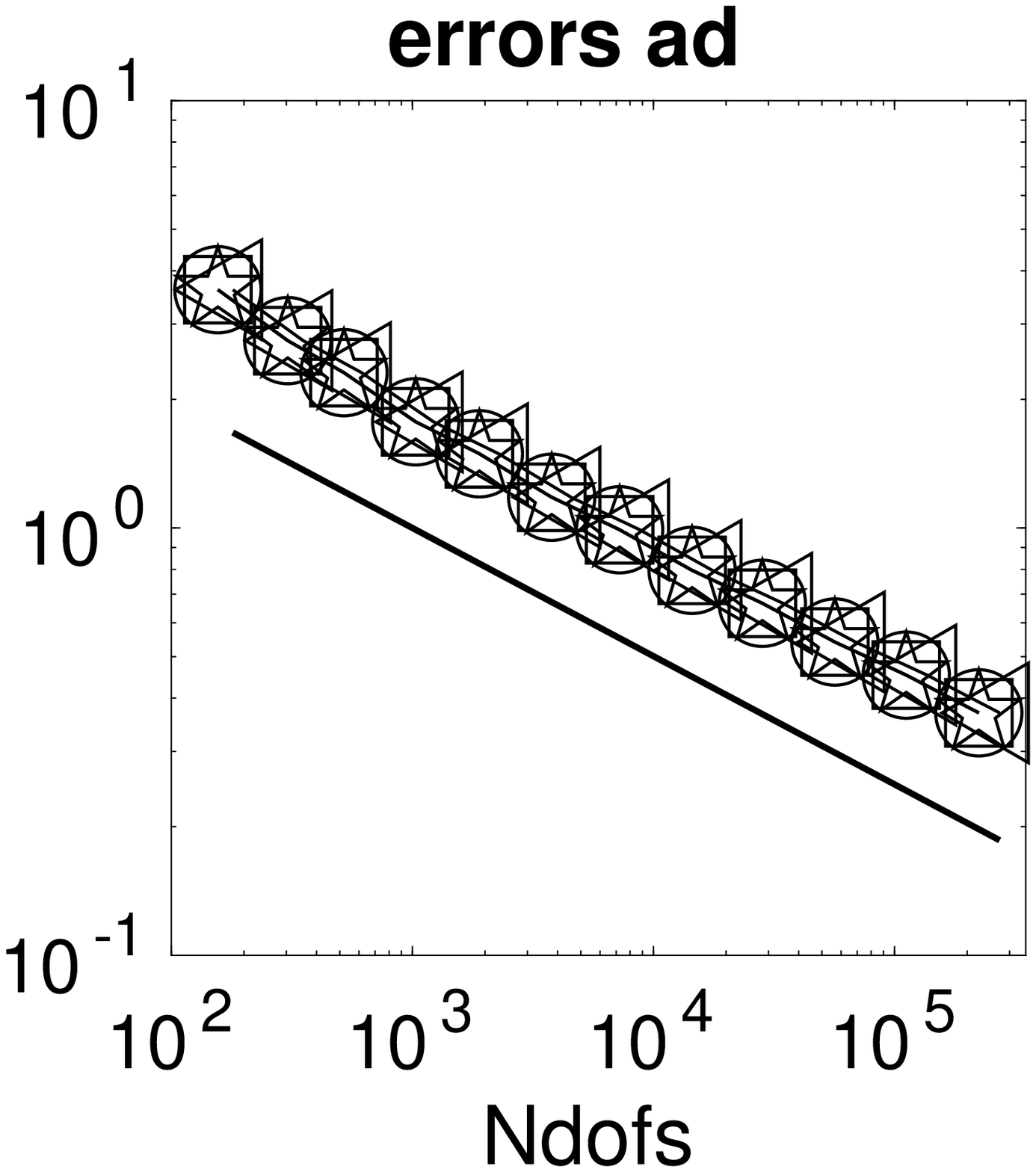}\\
\hspace{0.2cm}\tiny{(A.6)}
\end{minipage}
\caption{Example 1. Experimental rates of convergence, within uniform refinement, for the total errors $\|\mathbf{e}\|_{\Omega}$ and $\|\mathfrak{e}\|_{\Omega}$ (A.1) and each of their contributions (A.2)--(A.6) by considering the semi discrete scheme (with the implementations $\mathfrak{S}_{19}$, $\mathfrak{S}_{5}$, and $\mathfrak{S}_{5c}$) and the fully discrete scheme $\mathfrak{F}$.}
\label{fig:ex-1}
\end{figure}


\begin{figure}[!ht]
\centering
\psfrag{ndof 12}{{\normalsize $\mathsf{Ndof}^{-1.2}$}}
\psfrag{ndof 2}{{\normalsize $\mathsf{Ndof}^{-1}$}}
\psfrag{ndof 32}{{\normalsize $\mathsf{Ndof}^{-3/2}$}}
\psfrag{Ndofs}{{\large $\mathsf{Ndof}$}}
\psfrag{I 73}{{\large $\mathfrak{S}_{19}$}}
\psfrag{I 5}{{\large $\mathfrak{S}_{5}$}}
\psfrag{I 5c}{{\large $\mathfrak{S}_{5c}$}}
\psfrag{I fd}{{$\mathfrak{F}$}}
\begin{minipage}[c]{0.4\textwidth}\centering
\psfrag{etas totales}{\hspace{-0.0cm}\large{$\mathcal{E}_{ocp}$ and $\mathfrak{E}_{ocp}$}}
\includegraphics[trim={0 0 0 0},clip,width=4.7cm,height=3.2cm,scale=0.30]{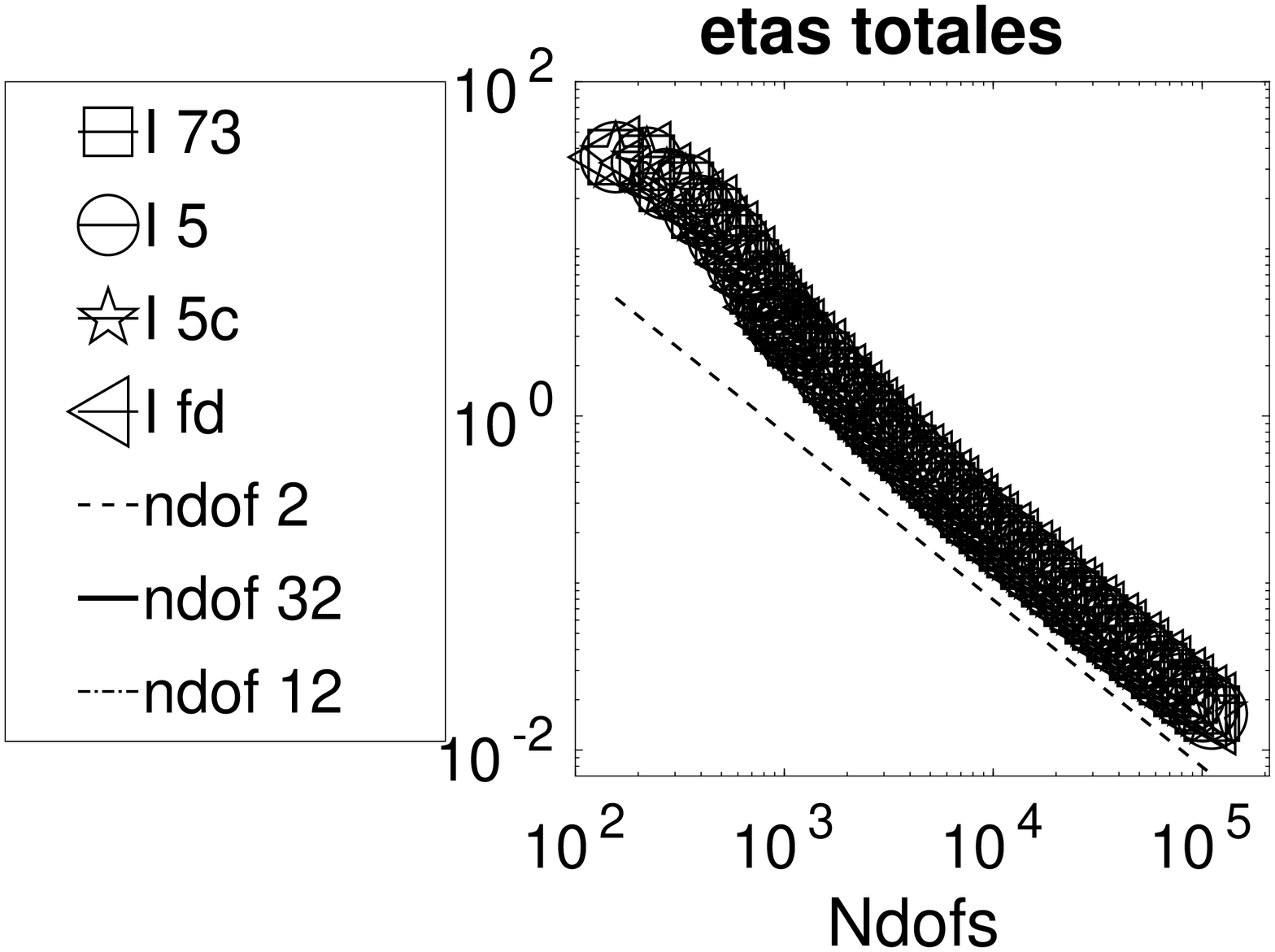}\\
\hspace{1.9cm}\tiny{(B.1)}
\end{minipage}
\begin{minipage}[c]{0.265\textwidth}\centering
\psfrag{indice ef}{\hspace{-0.0cm}\large{$\mathcal{I}$ and $\mathfrak{I}$}}
\includegraphics[trim={0 0 0 0},clip,width=3.4cm,height=3.2cm,,scale=0.30]{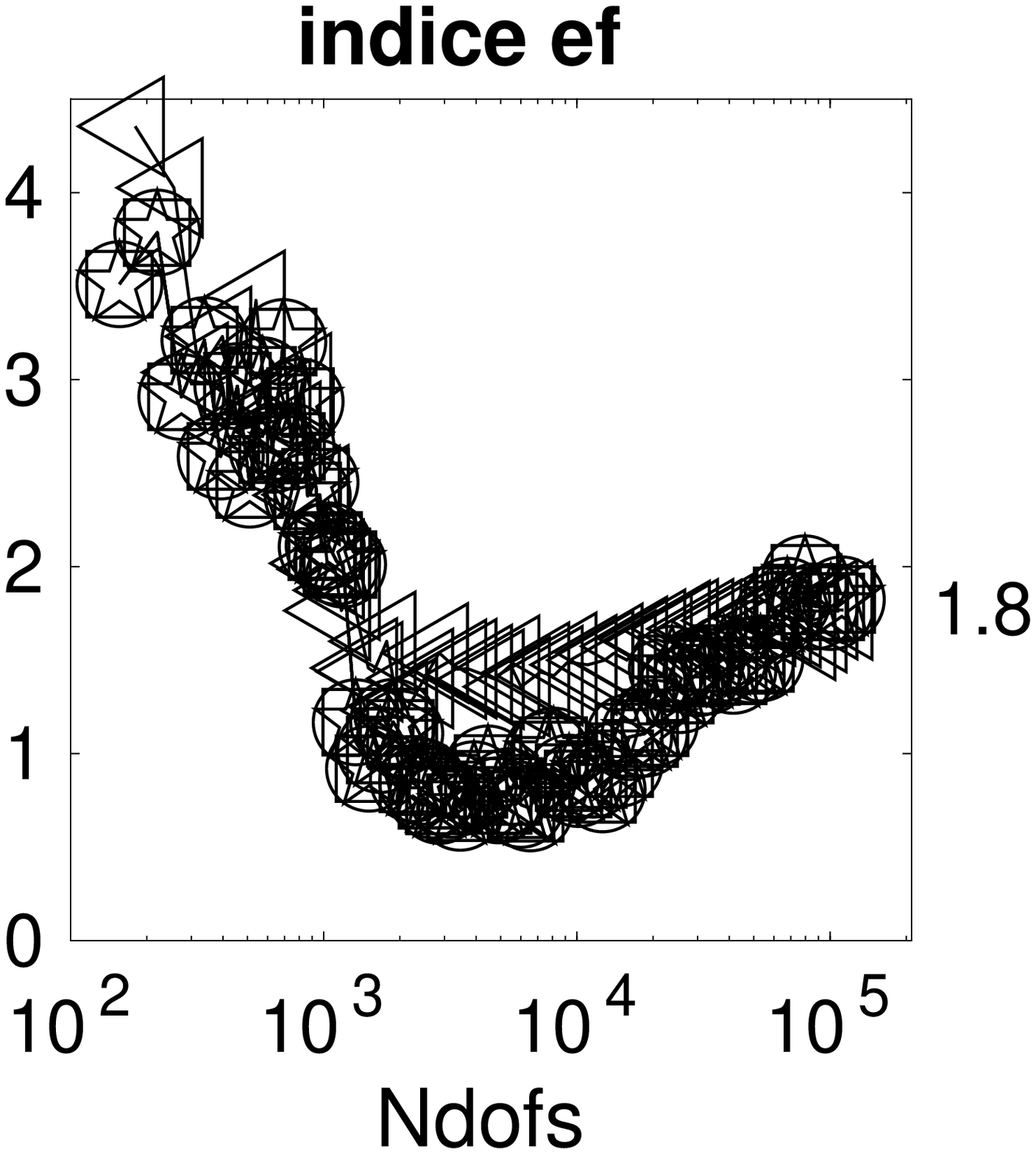}\\
\hspace{0.20cm}\tiny{(B.2)}
\end{minipage}
\begin{minipage}[c]{0.25\textwidth}\centering
\psfrag{errors ct}{\hspace{-1.0cm}\large{$\|\mathbf{e}_{\mathbf{u}}\|_{L^{2}(\Omega)}$ and $\|\mathfrak{e}_{\mathbf{u}}\|_{L^{2}(\Omega)}$}}
\includegraphics[trim={0 0 0 0},clip,width=3.2cm,height=3.2cm,scale=0.30]{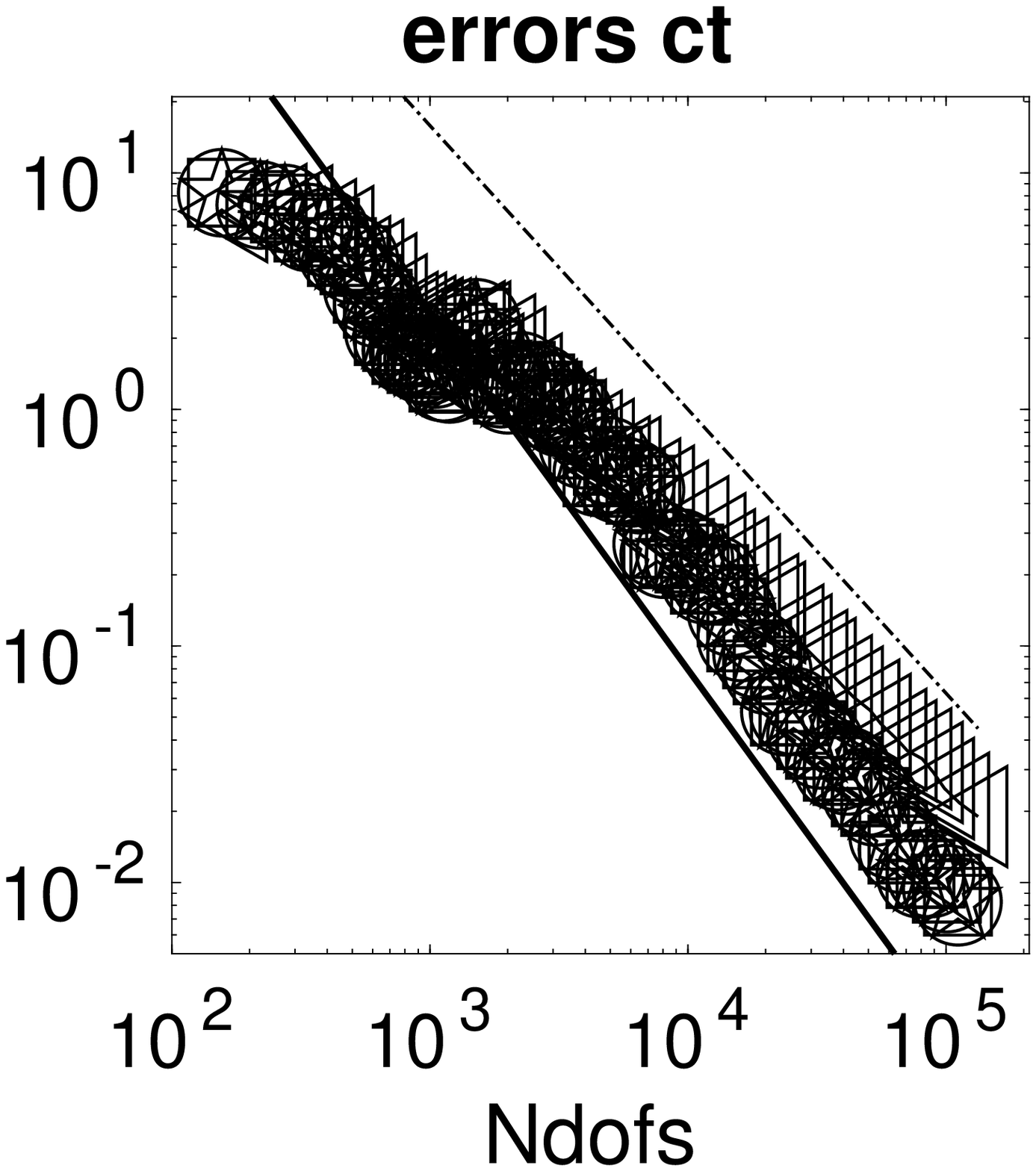}\\
\hspace{0.4cm}\tiny{(B.3)}
\end{minipage}
\\
\begin{minipage}[c]{0.248\textwidth}\centering
\psfrag{errors st}{\hspace{0.0cm}\large{$\|\nabla \mathbf{e}_{\mathbf{y}}\|_{\mathbf{L}^{2}(\Omega)}$}}
\includegraphics[trim={0 0 0 0},clip,width=3.2cm,height=3.2cm,scale=0.30]{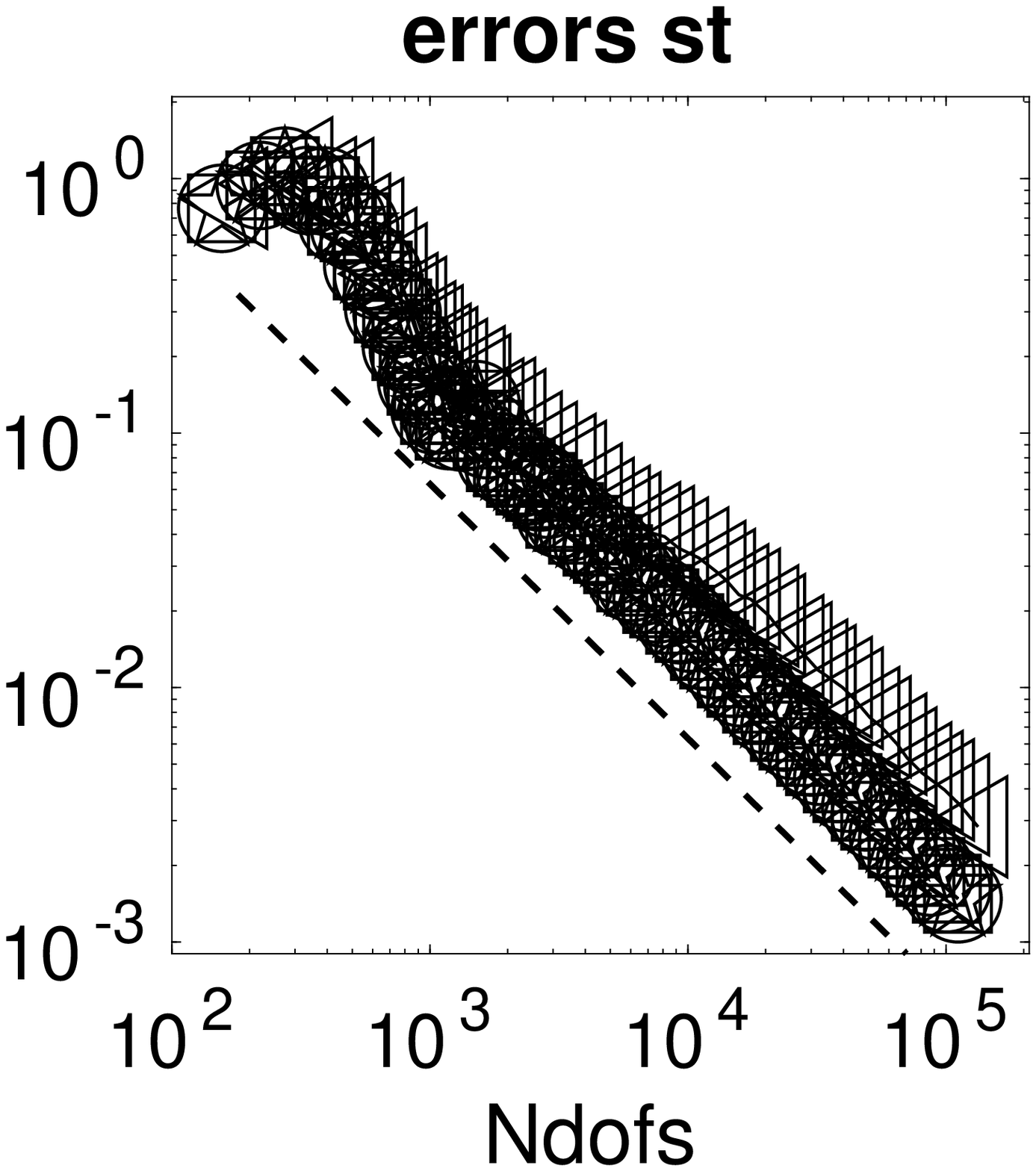}\\
\hspace{0.35cm}\tiny{(B.4)}
\end{minipage}
\begin{minipage}[c]{0.248\textwidth}\centering
\psfrag{errors st}{\hspace{0.2cm}\large{$\|e_{p}\|_{L^{2}(\Omega)}$}}
\includegraphics[trim={0 0 0 0},clip,width=3.2cm,height=3.2cm,scale=0.30]{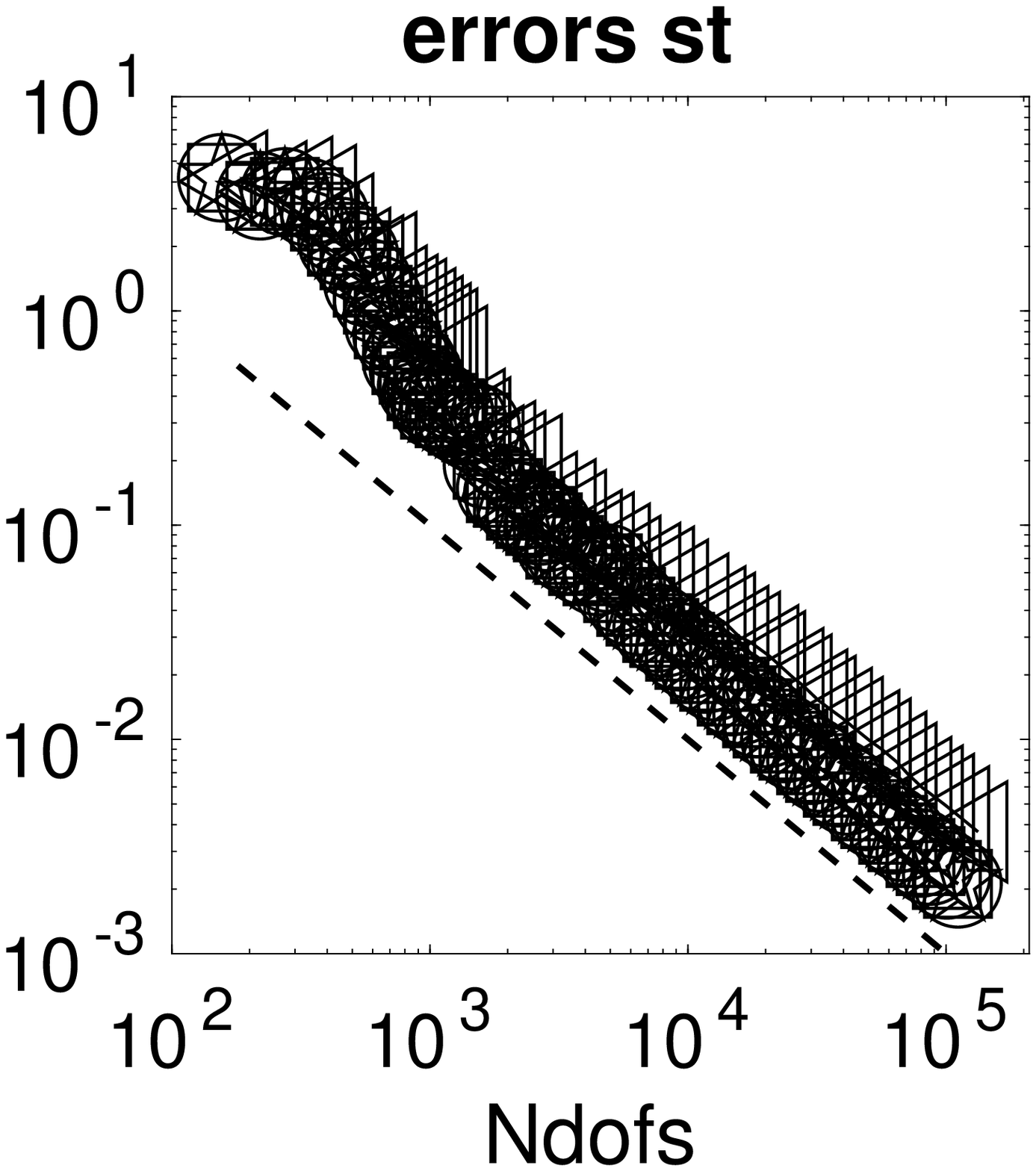}\\
\hspace{0.35cm}\tiny{(B.5)}
\end{minipage}
\begin{minipage}[c]{0.248\textwidth}\centering
\psfrag{errors ad}{\hspace{0.0cm}\large{$\|\nabla \mathbf{e}_{\mathbf{z}}\|_{\mathbf{L}^{2}(\Omega)}$}}
\includegraphics[trim={0 0 0 0},clip,width=3.2cm,height=3.2cm,scale=0.30]{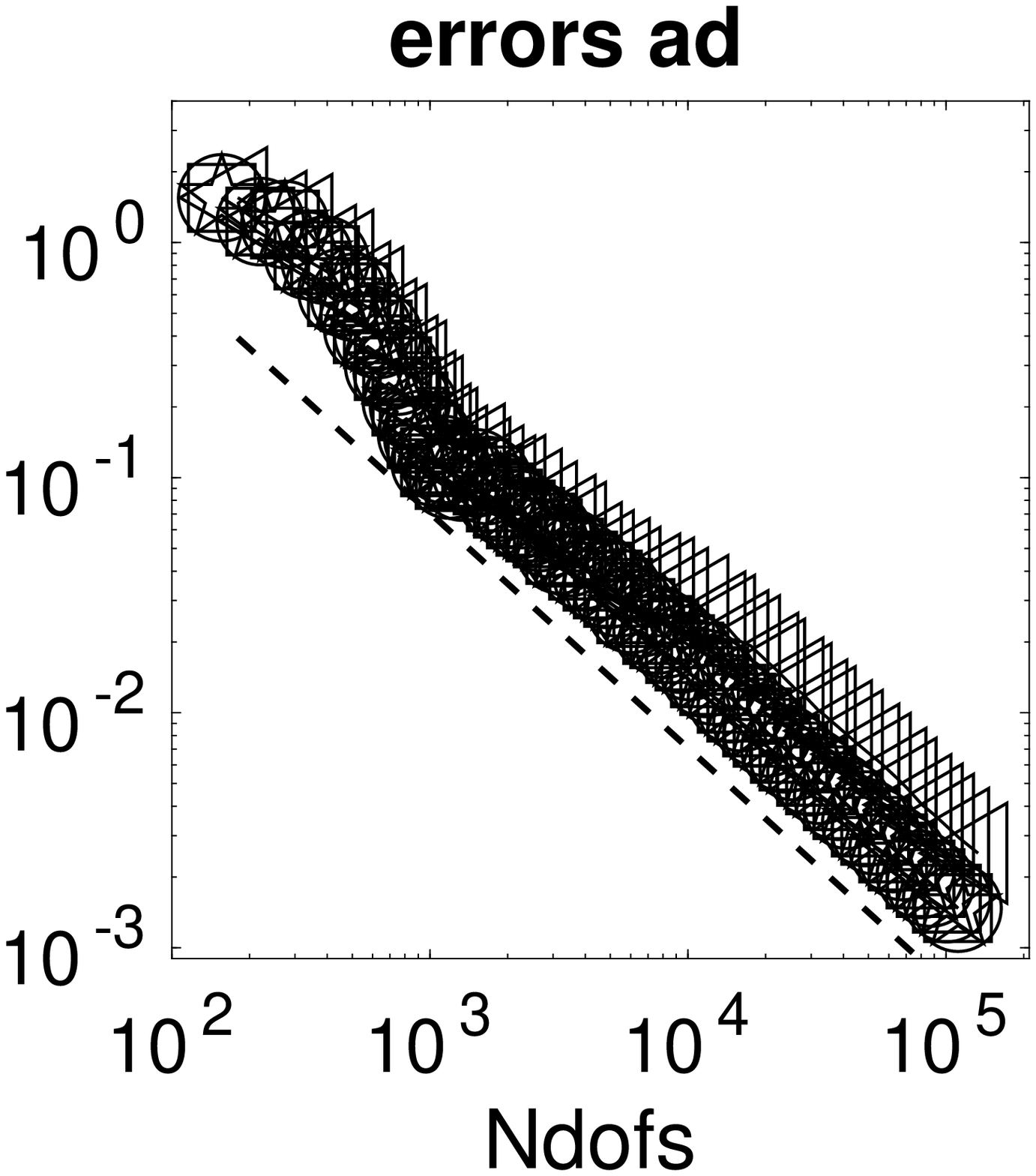}\\
\hspace{0.35cm}\tiny{(B.6)}
\end{minipage}
\begin{minipage}[c]{0.248\textwidth}\centering
\psfrag{errors ad}{\hspace{0.2cm}\large{$\|e_{r}\|_{L^{2}(\Omega)}$}}
\includegraphics[trim={0 0 0 0},clip,width=3.2cm,height=3.2cm,scale=0.30]{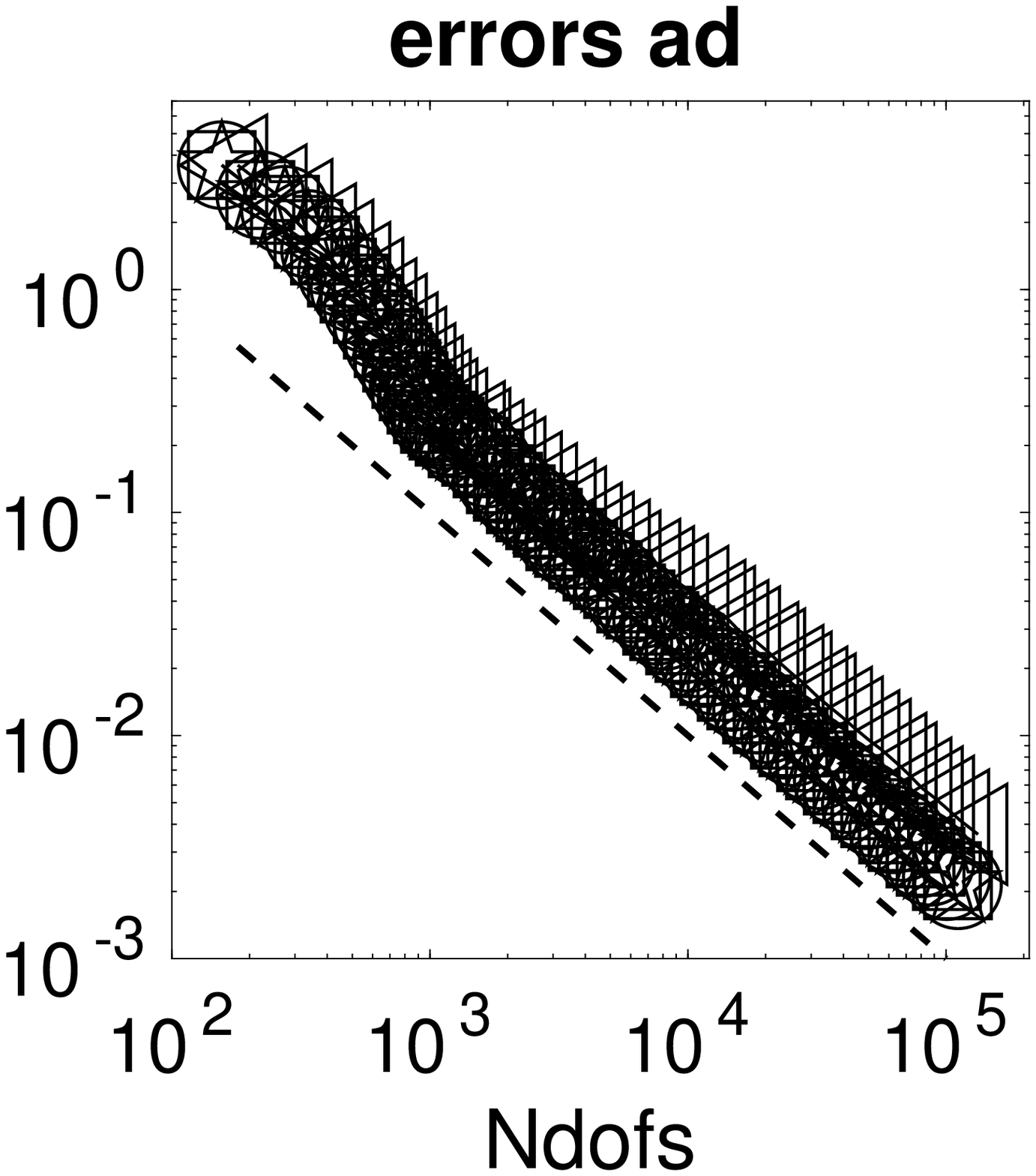}\\
\hspace{0.35cm}\tiny{(B.7)}
\end{minipage}
\\
\begin{minipage}[c]{0.248\textwidth}\centering
\includegraphics[trim={0 0 0 0},clip,width=3.0cm,height=3.0cm,scale=0.30]{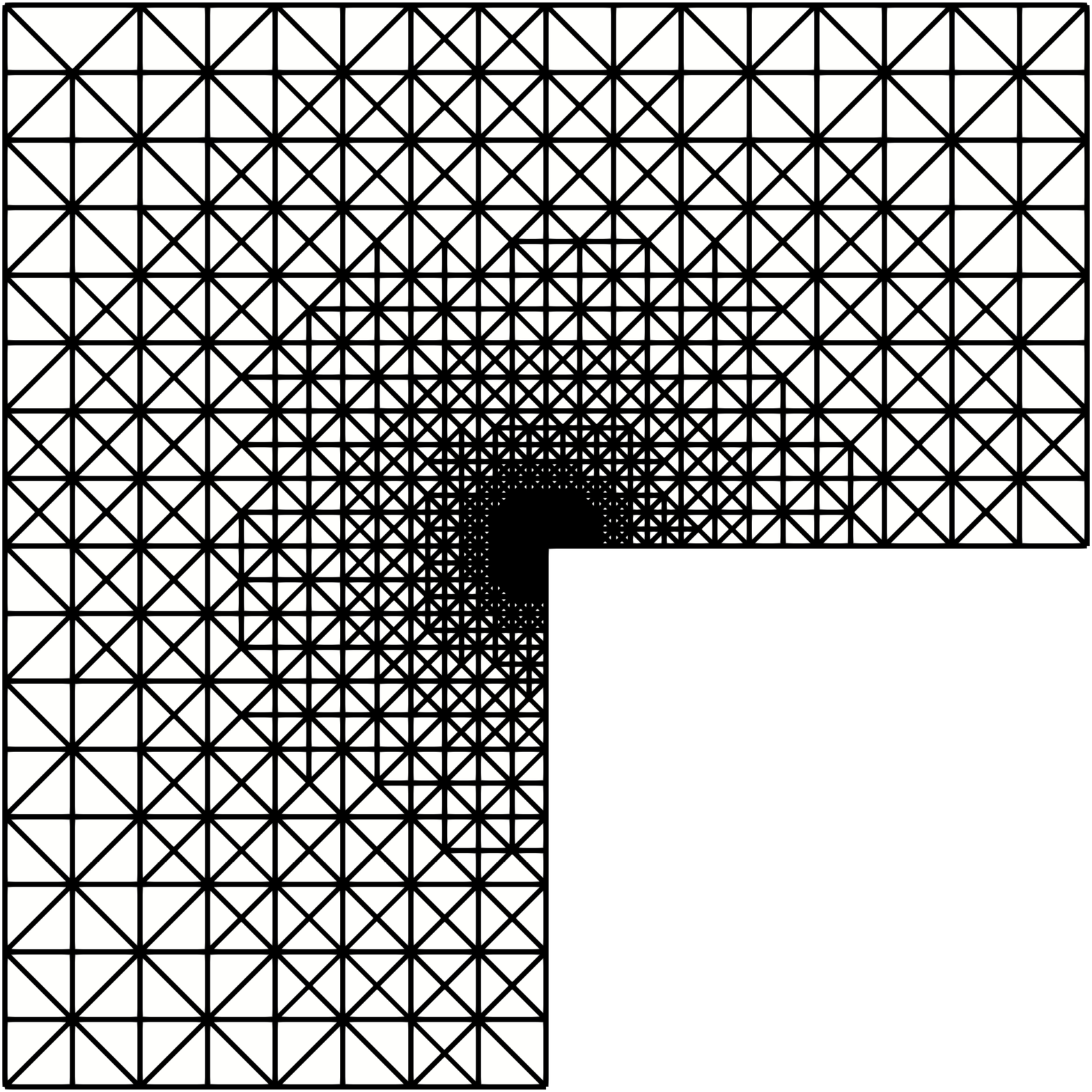}\\
\hspace{0.0cm}\tiny{(B.8)}
\end{minipage}
\begin{minipage}[c]{0.248\textwidth}\centering
\includegraphics[trim={0 0 0 0},clip,width=3.0cm,height=3.0cm,scale=0.30]{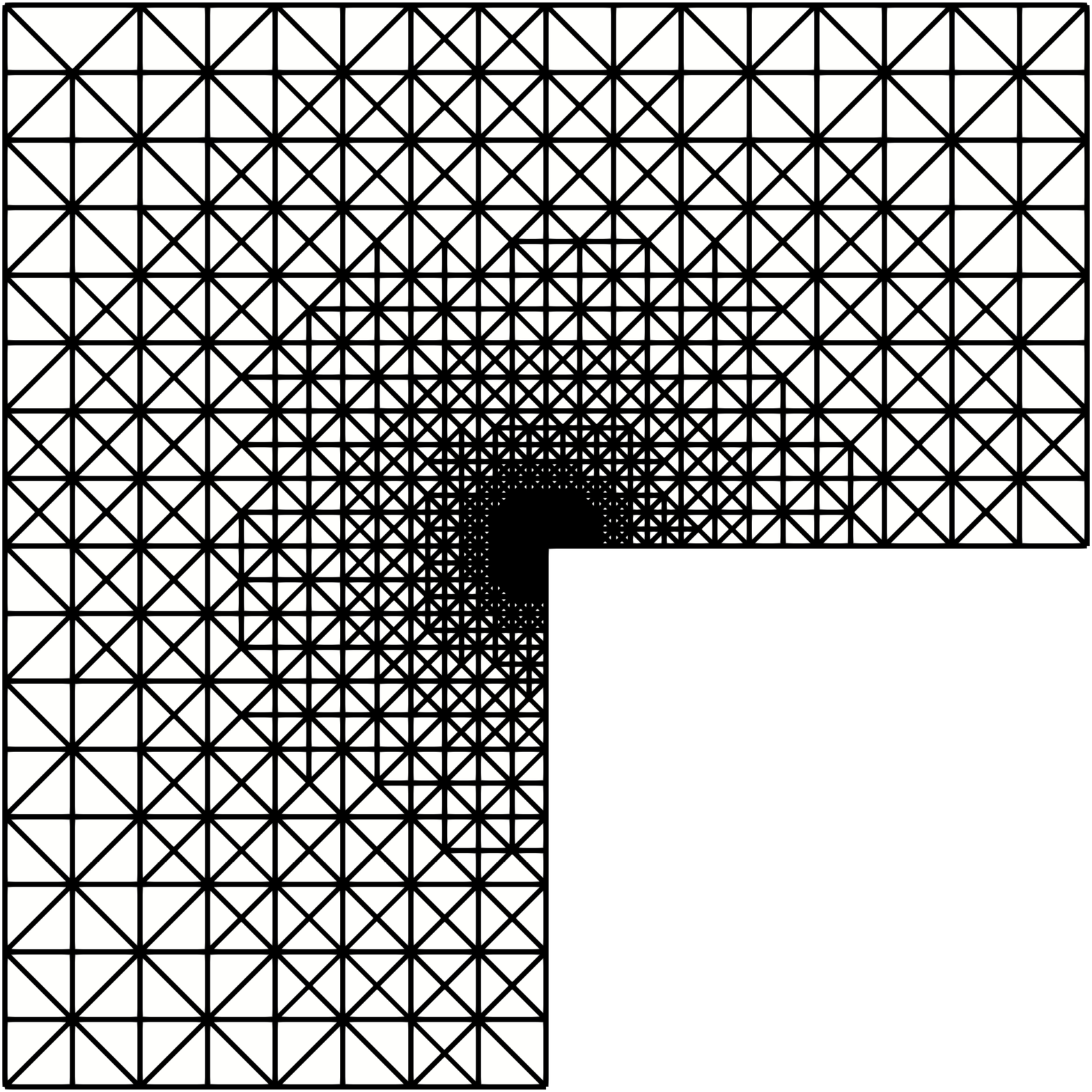}\\
\hspace{0.0cm}\tiny{(B.9)}
\end{minipage}
\begin{minipage}[c]{0.248\textwidth}\centering
\includegraphics[trim={0 0 0 0},clip,width=3.0cm,height=3.0cm,scale=0.30]{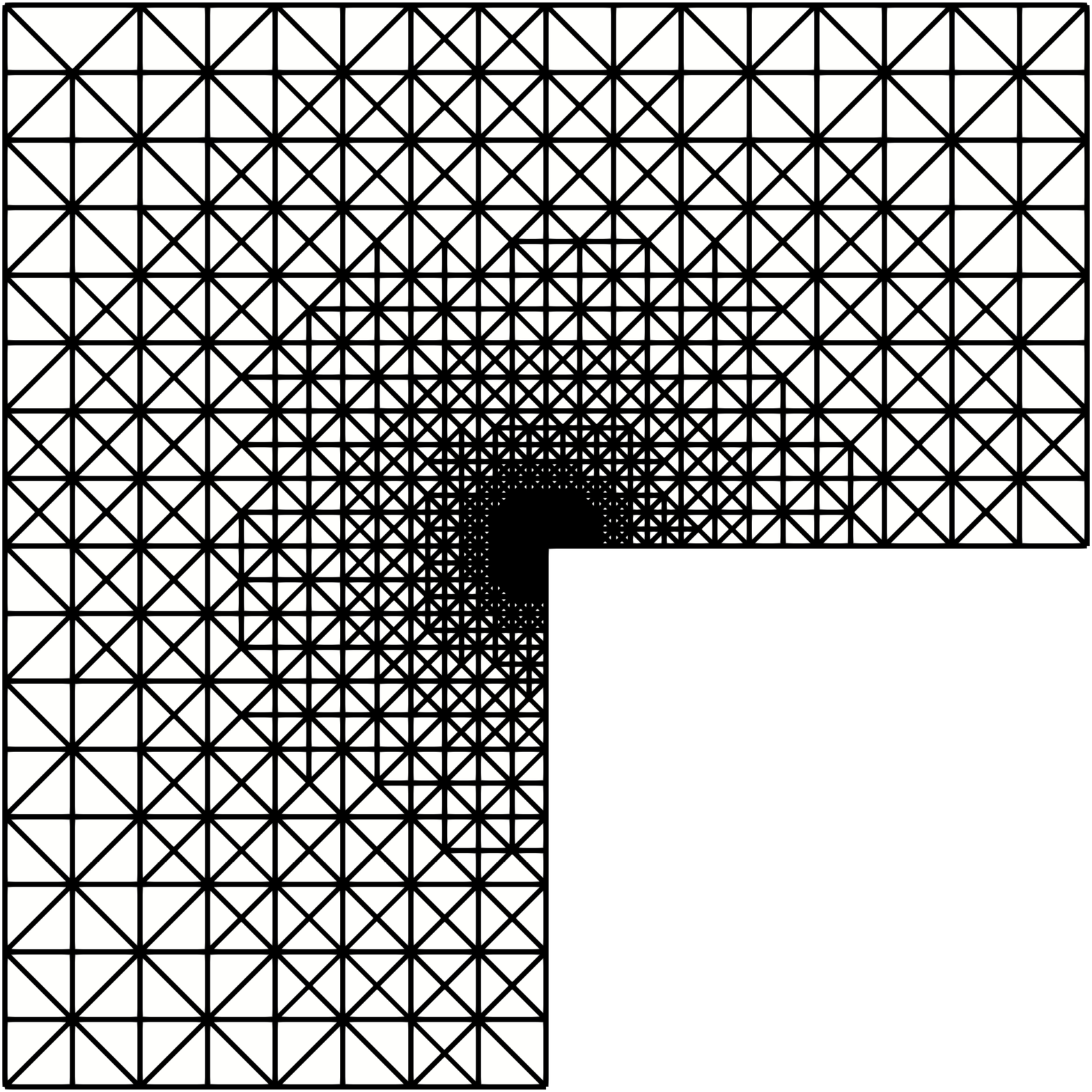}\\
\hspace{0.0cm}\tiny{(B.10)}
\end{minipage}
\begin{minipage}[c]{0.248\textwidth}\centering
\includegraphics[trim={0 0 0 0},clip,width=3.0cm,height=3.0cm,scale=0.30]{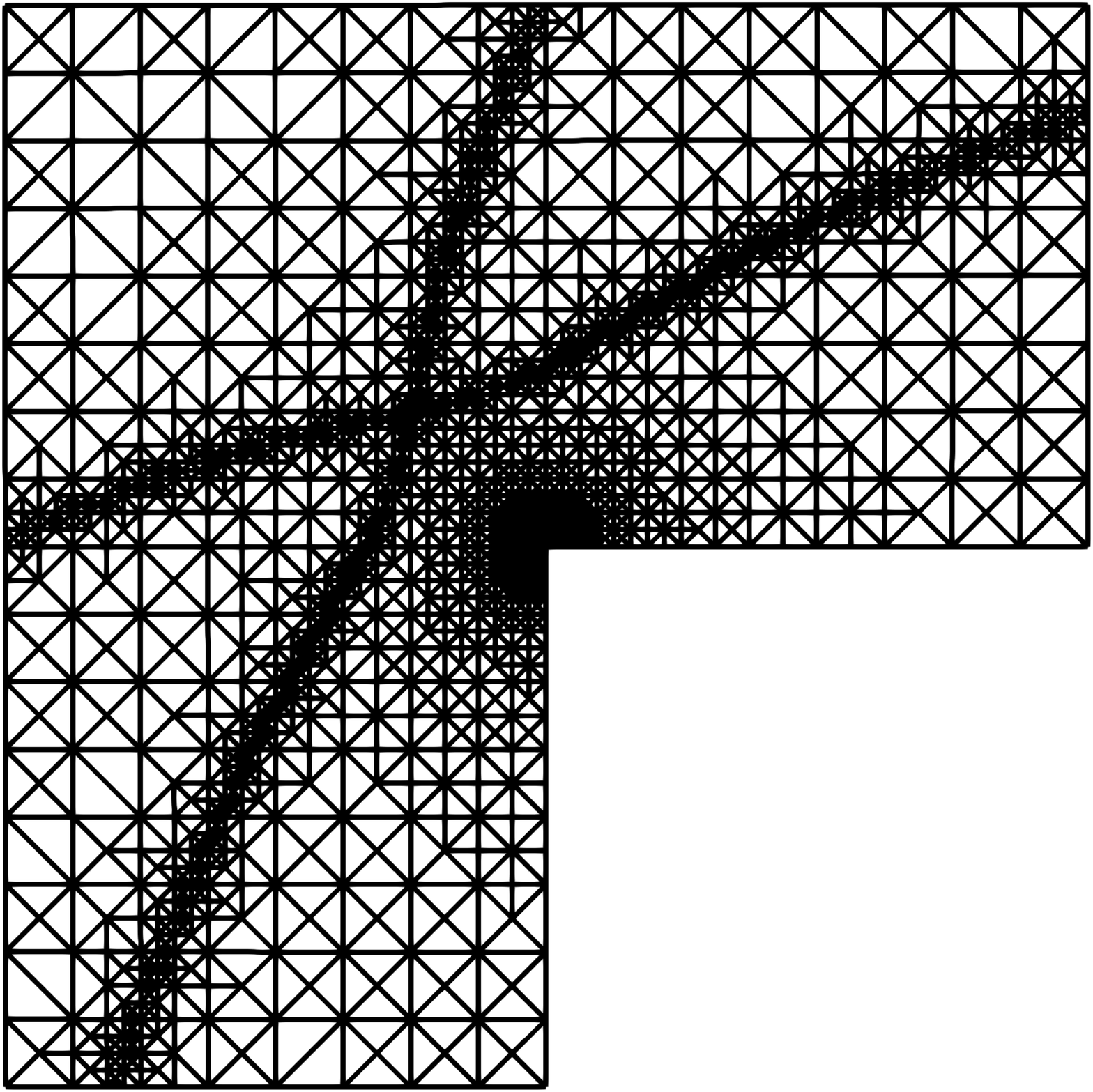}\\
\hspace{0.0cm}\tiny{(B.11)}
\end{minipage}
\caption{Example 1. Experimental rates of convergence, within adaptive refinement, for the total error estimators $\mathcal{E}_{ocp}$ and $\mathfrak{E}_{ocp}$ (B.1), effectivity indices $\mathcal{I}$ and $\mathfrak{I}$ (B.2), experimental rates of convergence of each contribution of $\|\mathbf{e}\|_{\Omega}$ and $\|\mathfrak{e}\|_{\Omega}$ (B.3)--(B.7), and the adaptively refined meshes obtained after $45$ iterations of the adaptive loop by considering the semi discrete schemes $\mathfrak{S}_{19}$, $\mathfrak{S}_{5}$, and $\mathfrak{S}_{5c}$ and the fully discrete scheme $\mathfrak{F}$ (B.8)--(B.11).}
\label{fig:ex-1.2}
\end{figure}


\begin{figure}[!ht]
\centering
\begin{minipage}[c]{0.245\textwidth}\centering
\includegraphics[trim={0 0 0 0},clip,width=2.8cm,height=2.5cm,scale=0.30]{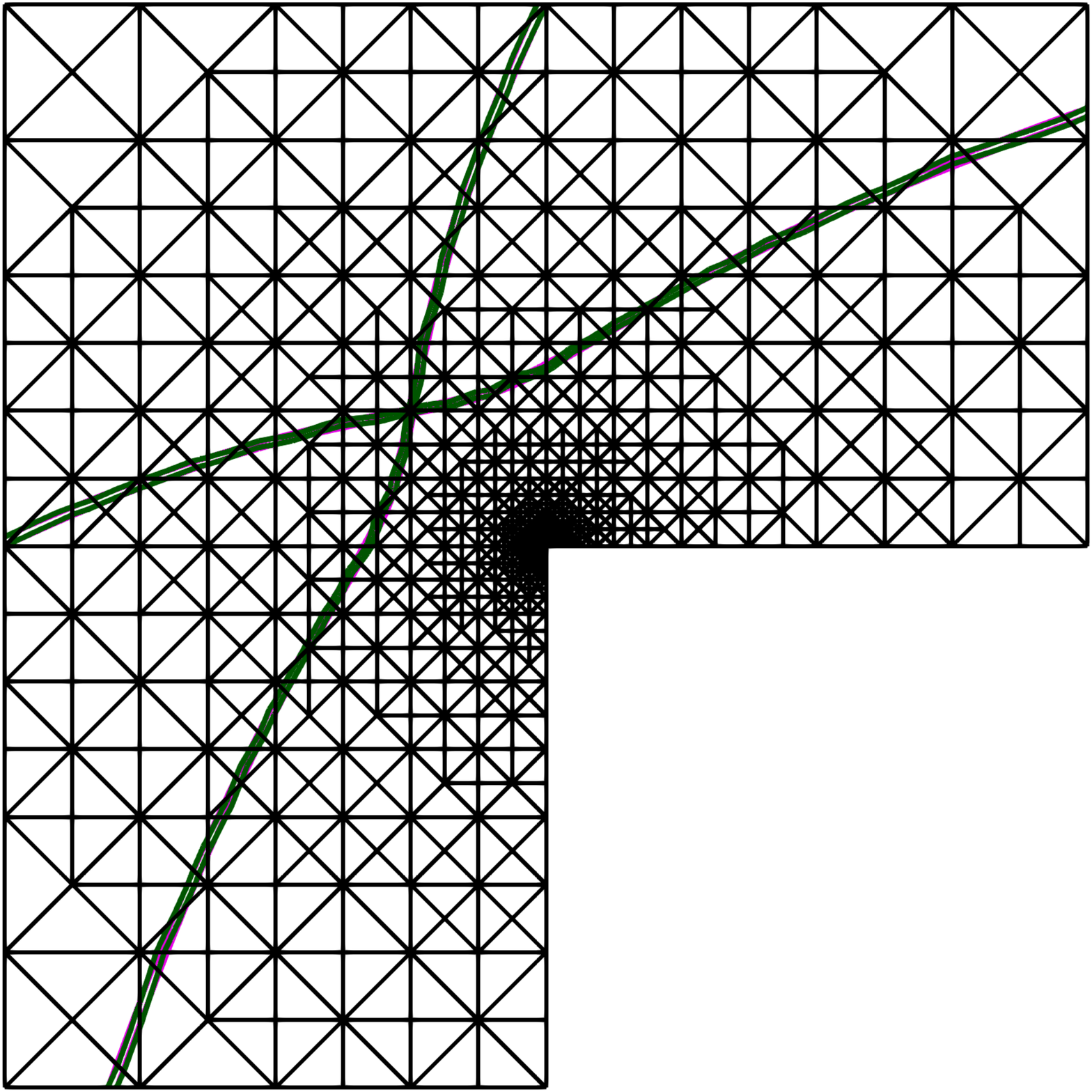}\\
\hspace{0.0cm}\tiny{(C.1)}
\end{minipage}
\begin{minipage}[c]{0.245\textwidth}\centering
\includegraphics[trim={0 0 0 0},clip,width=2.8cm,height=2.5cm,scale=0.30]{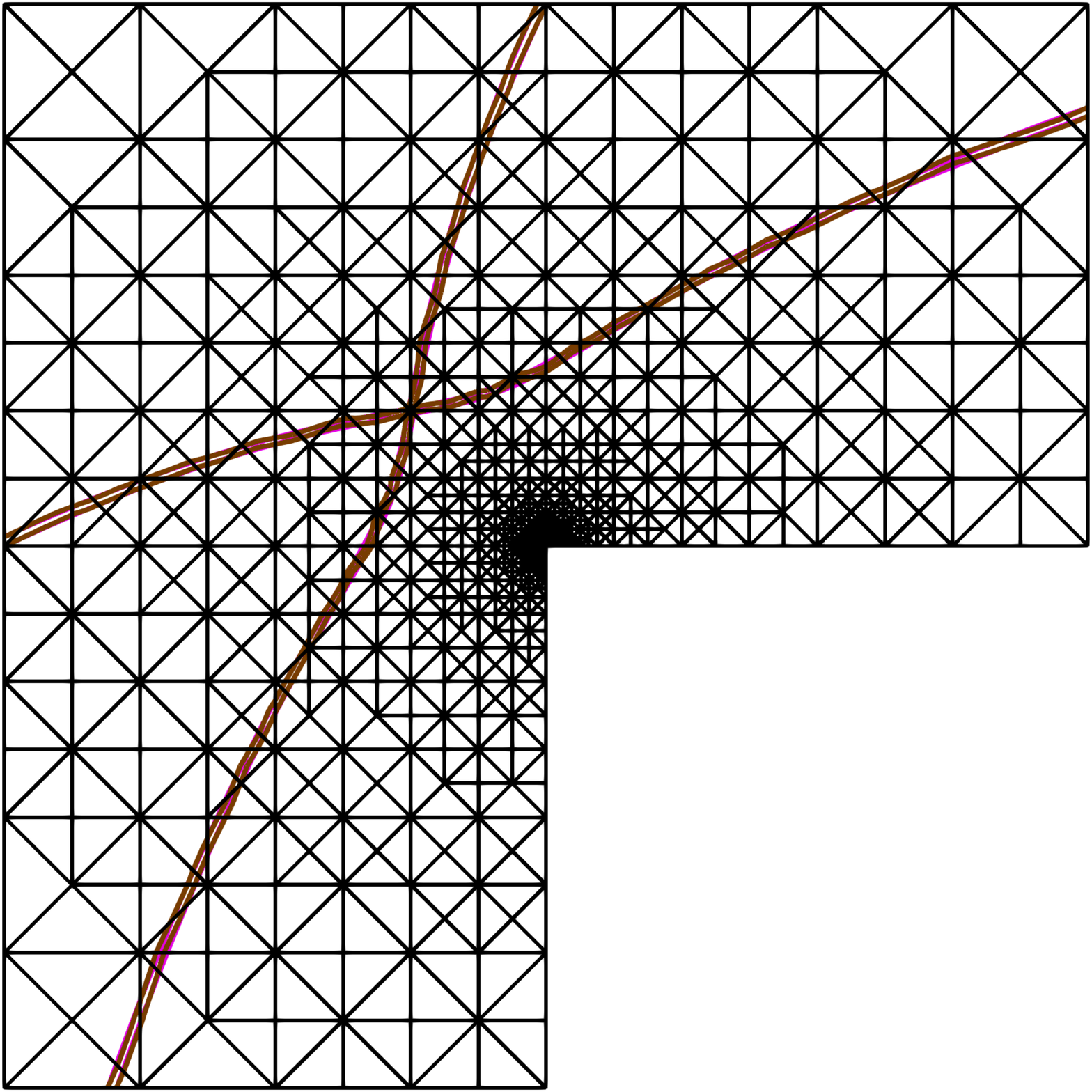}\\
\hspace{0.0cm}\tiny{(C.2)}
\end{minipage}
\begin{minipage}[c]{0.245\textwidth}\centering
\includegraphics[trim={0 0 0 0},clip,width=2.8cm,height=2.5cm,scale=0.30]{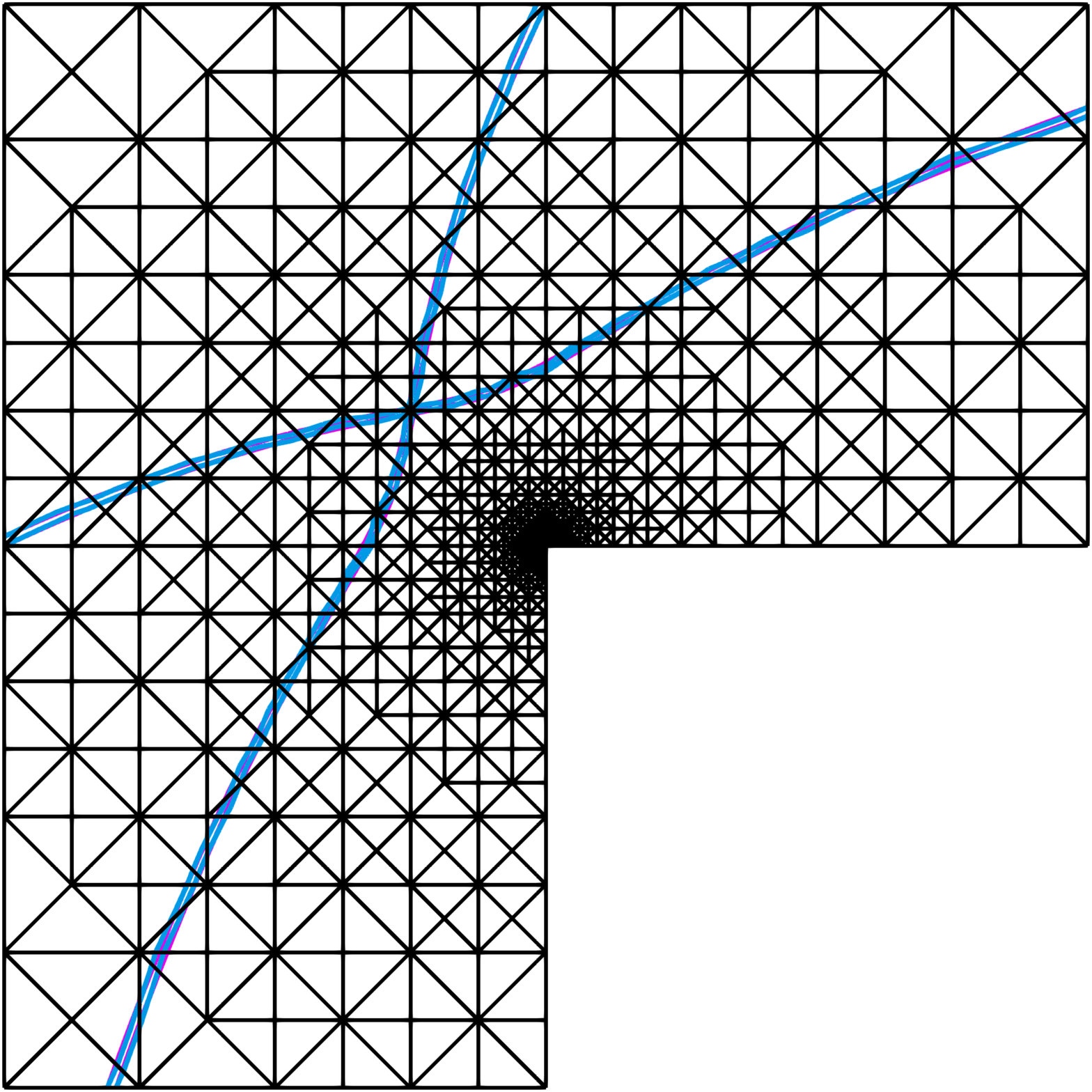}\\
\hspace{0.0cm}\tiny{(C.3)}
\end{minipage}
\begin{minipage}[c]{0.245\textwidth}\centering
\includegraphics[trim={0 0 0 0},clip,width=2.8cm,height=2.5cm,scale=0.30]{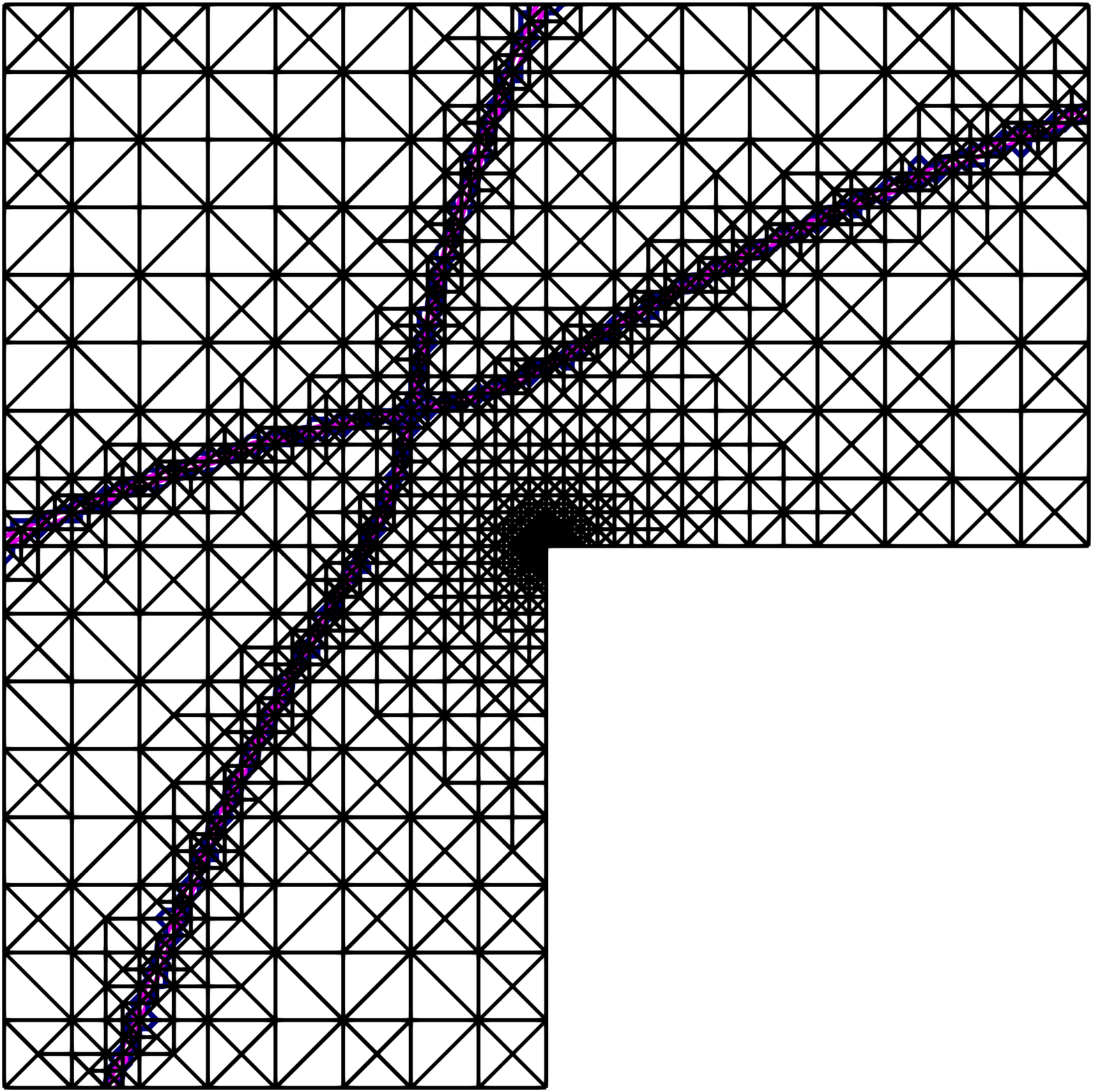}\\
\hspace{0.0cm}\tiny{(C.4)}
\end{minipage}
\\
\begin{minipage}[c]{0.245\textwidth}\centering
\includegraphics[trim={0 0 0 0},clip,width=2.8cm,height=2.8cm,scale=0.30]{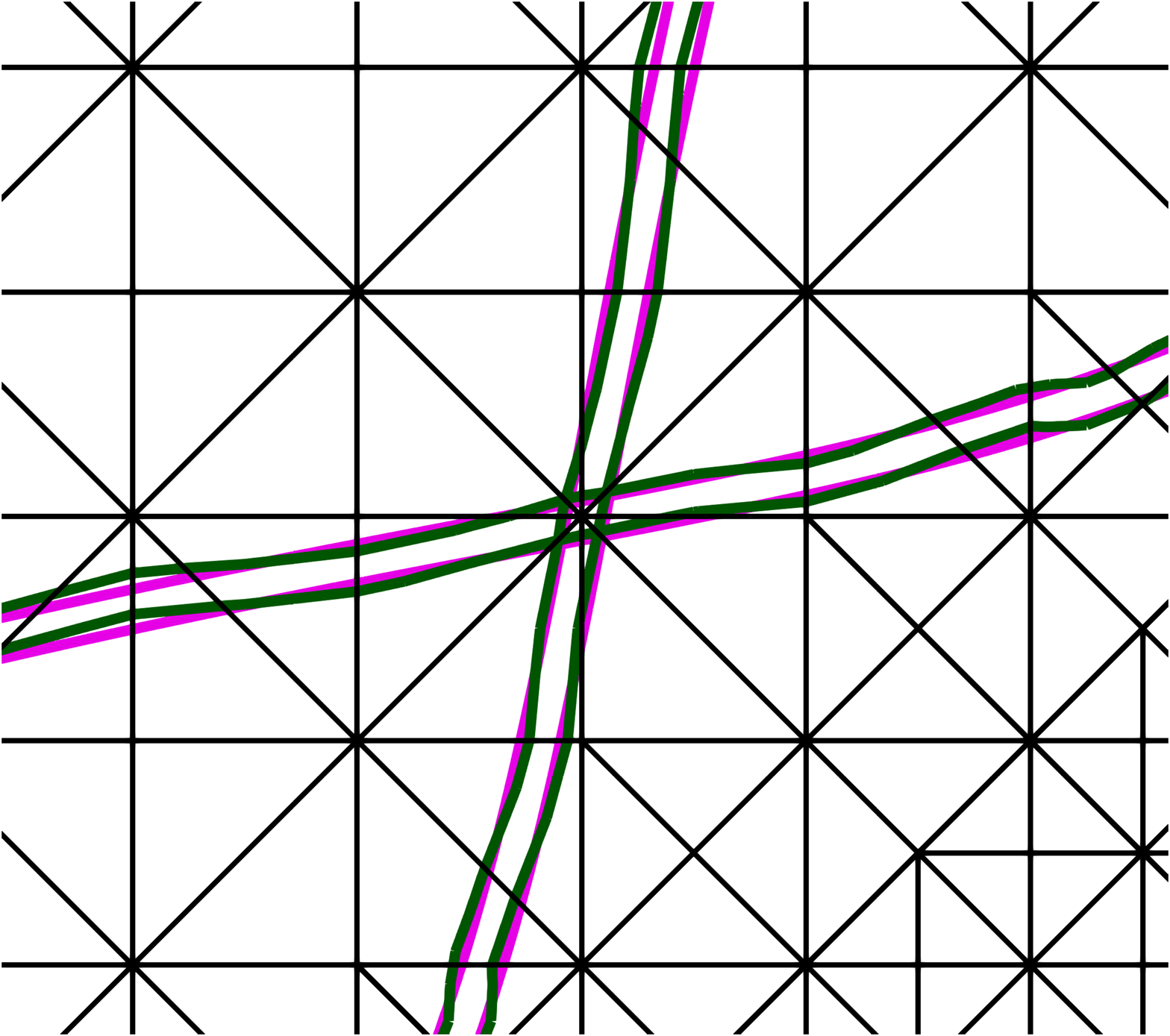}\\
\hspace{0.0cm}\tiny{(C.5)}
\end{minipage}
\begin{minipage}[c]{0.245\textwidth}\centering
\includegraphics[trim={0 0 0 0},clip,width=2.8cm,height=2.8cm,scale=0.30]{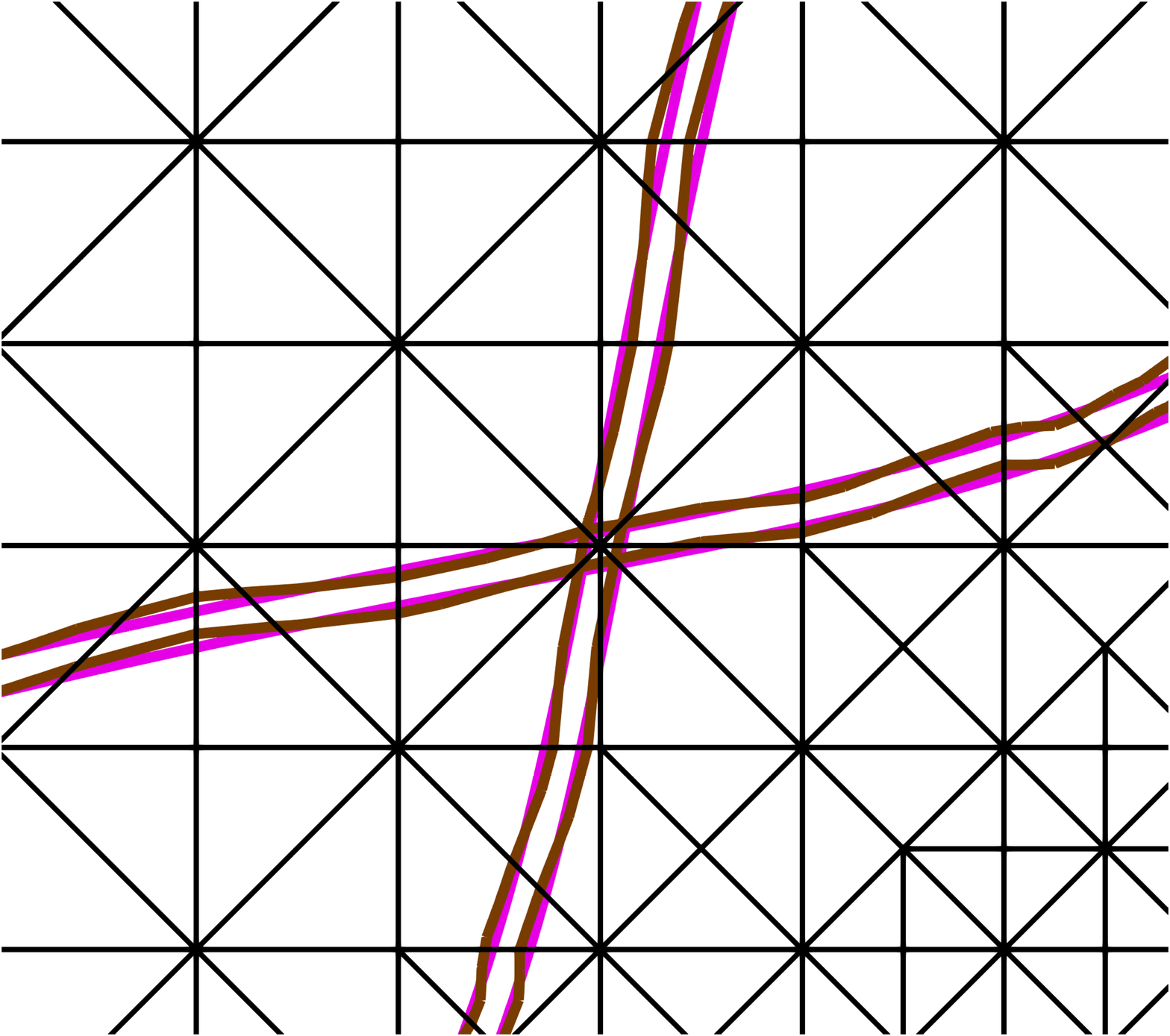}\\
\hspace{0.0cm}\tiny{(C.6)}
\end{minipage}
\begin{minipage}[c]{0.245\textwidth}\centering
\includegraphics[trim={0 0 0 0},clip,width=2.8cm,height=2.8cm,scale=0.30]{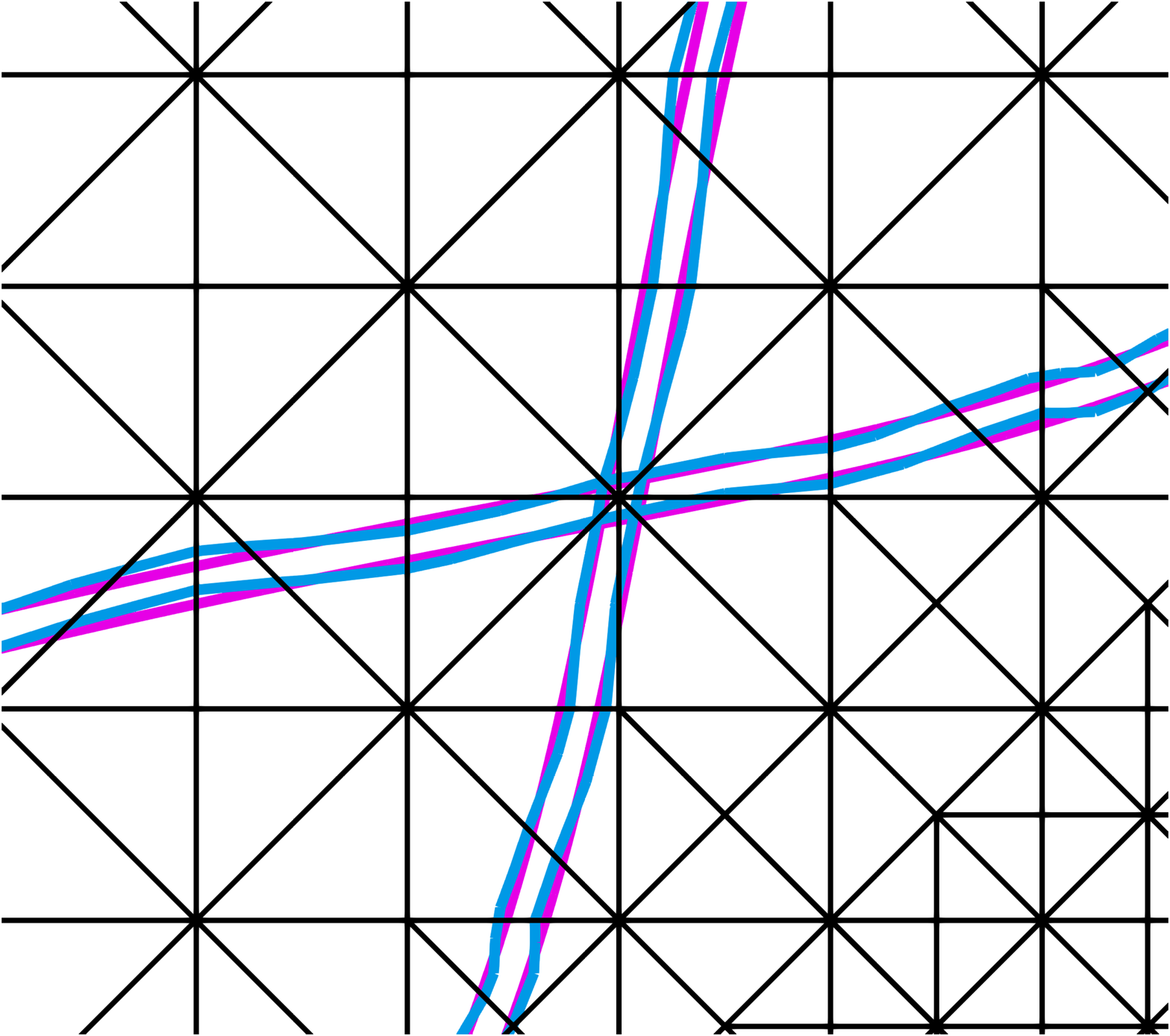}\\
\hspace{0.0cm}\tiny{(C.7)}
\end{minipage}
\begin{minipage}[c]{0.245\textwidth}\centering
\includegraphics[trim={0 0 0 0},clip,width=2.8cm,height=2.8cm,scale=0.30]{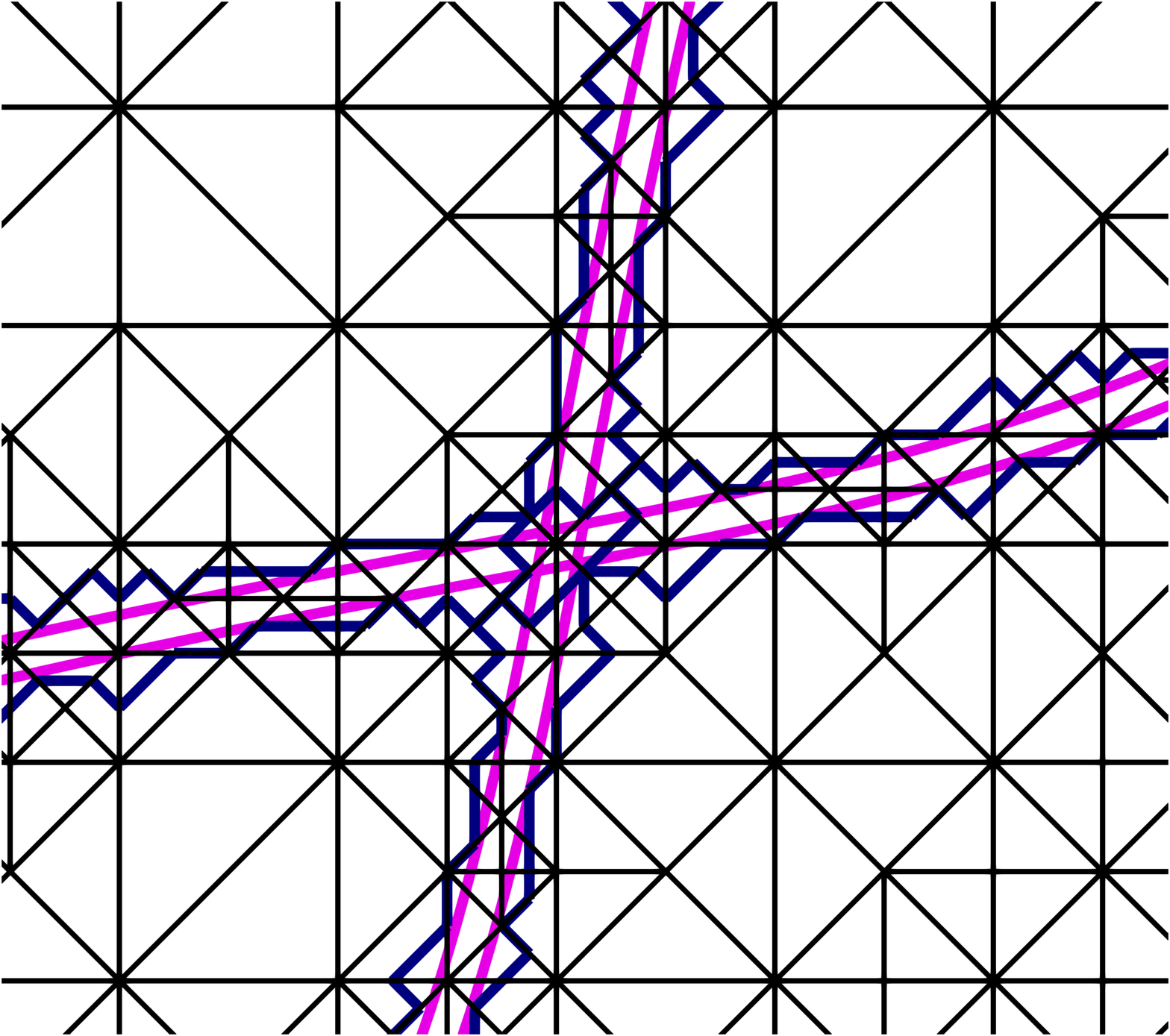}\\
\hspace{0.0cm}\tiny{(C.8)}
\end{minipage}
\caption{Example 1: Comparison of the continuous (magenta) and discrete level sets of the boundaries of the control active/inactive regions on the adaptively refined meshes obtained after $42$ iterations of the adaptive loop by considering the semi discrete schemes $\mathfrak{S}_{19}$ in green ((C.1) and (C.5)), $\mathfrak{S}_{5}$ in brown ((C.2) and (C.6)), and $\mathfrak{S}_{5c}$ in cyan ((C.3) and (C.7)) and the fully discrete scheme $\mathfrak{F}$ in blue ((C.4) and (C.8)), respectively.}
\label{fig:ex-1.3}
\end{figure}

In Figures \ref{fig:ex-1}, \ref{fig:ex-1.2} and \ref{fig:ex-1.3} we present the results obtained for Example 1. We show experimental rates of convergence for the total errors $\|\mathbf{e}\|_{\Omega}$ and $\|\mathfrak{e}\|_{\Omega}$ and each of their individual contributions when uniform (Figure  \ref{fig:ex-1}) and adaptive refinement (Figure \ref{fig:ex-1.2}) are considered. In Figure \ref{fig:ex-1.3} we present the borders of the active sets for the exact and approximated solutions. The following comments and remarks are now in order:
\begin{itemize}
\item \textit{Adaptive versus uniform refinement}: the devised adaptive strategies based on both the fully discrete scheme and the semi discrete one \emph{outperform} uniform refinement. 

\item \textit{Semi discrete scheme}: the adaptive strategies based on the approximated variational discretization approach deliver optimal experimental rates of convergences \emph{for all of the contributions} of the total error $\|\mathfrak{e}\|_{\Omega}$; see subfigures (B.3)--(B.7).

\item \textit{Fully discrete scheme}: the adaptive strategy based on the fully discrete scheme delivers optimal experimental rates of convergences for all of the contributions of the total error $\|\mathbf{e}\|_{\Omega}$ \emph{but with the exception} of $\|\mathbf{e}_{\mathbf{u}}\|_{\mathbf{L}^2(\Omega)}$; see subfigure (B.3). This is in sharp contrast with the approximated variational discretization approach since the latter is able to deliver optimal experimental rates of convergence for the error approximation of the control variable.

\item \textit{Effectivity indices}: all the effectivity indices are stabilized around the value $1.8$; see subfigure (B.2).

\item \textit{Mesh refinement}: we observe, in Figure \ref{fig:ex-1.2}, that most of the refinement is being performed in the regions of the domain that are close to the geometric singularity (B.8)--(B.11). Furthermore, we observe that, when the fully discrete scheme is considered, the refinement is also being performed in the regions where the restrictions of the control variable become active (B.11). These DOF seem not necessary for an accurate approximation of the control variable; compare with the optimal experimental rate of convergence achieved by the approximated variational discretization scheme, which is observed in subfigure (B.3).

\item \textit{Boundary of the discrete active set}: we observe, in Figure \ref{fig:ex-1.3}, that the discrete level sets obtained with the approximated variational discretization schemes \alert{provide excellent approximations to that of the continuous active set
(see subfigures (C.1)--(C.4) and (C.5)--(C.8)). We stress the fact that this positive behavior is achieved without the necessity of the extra degrees of freedom generated by the fully discrete scheme.}

\end{itemize}

\subsection{Example 2 (three dimensional convex domain)} We consider $\Omega=(0,1)^{3}$, $\mathbf{a} = 10^{-3}(-7,-7,-7)$, $\mathbf{b} = 10^{-3}(7,7,7)$, $\alpha = 10^{-1}$, and $\nu = 10^{-2}$. The exact optimal state and adjoint state are given by
\begin{align*}
\bar{\mathbf{y}}(x_{1},x_{2},x_{3}) =& 10^{-3}\textbf{curl}\left( \left(x_{2}x_{3}(1-x_{2})(1-x_{3})\right)^{2}\left(1 - x_{1} - \dfrac{e^{-x_{1}/\nu} - e^{-1/\nu}}{1 - e^{-1/\nu}} \right) \right),\\
\bar{\mathbf{z}}(x_{1},x_{2},x_{3}) =& \textbf{curl}\left( \left(x_{1}x_{2}x_{3}(1-x_{1})(1-x_{2})(1 - x_{3})\right)^{2} \right),\\
& ~ \bar{p}(x_{1},x_{2}) = \bar{r}(x_{1},x_{2}) = (x_{1}x_{2}x_{3} - 1/8).
\end{align*}

Here, we study the performance of the proposed error estimators $\mathcal{E}_{ocp}$ and $\mathfrak{E}_{ocp}$ by comparing uniform versus adaptive refinement. \alert{We also quantify the effect of utilizing different integration rules when implementing the semi discrete scheme. We describe such integration rules in what follows.}

\begin{itemize}
\item[$\bullet$] We first use a quadrature formula which, in three dimensions, is exact for polynomials of degree fourteen $(14)$.
\item[$\bullet$] Second, we use a quadrature formula which, in three dimensions, is exact for polynomials of degree five $(5)$.
\item[$\bullet$] Third, we consider a composed quadrature: in the elements $T \in \T$ where the control exhibits kinks we use a quadrature formula which is exact for polynomials of degree five, whereas in the remaining elements we utilize a quadrature formula which is exact for polynomials of degree fourteen.
\end{itemize}

In Figures \ref{fig:ex-2_unif} and \ref{fig:ex-2_adap} we present the results obtained for Example 2. Similar conclusions to the ones presented for Example 1 can be derived. In particular, we observe optimal experimental rates of convergence for all the involved variables within the adaptive loops of both discretization schemes.

}



\begin{figure}[!ht]
\centering
\psfrag{error ct}{{\large $\| \mathfrak{e}_{\mathfrak{u}}\|_{\mathbf{L}^{2}(\Omega)}$}}
\psfrag{ndofs 13}{{\normalsize $\mathsf{Ndof}^{-1/3}$}}
\psfrag{ndofs 15}{{\normalsize $\mathsf{Ndof}^{-1/2}$}}
\psfrag{Ndofs}{{\large $\mathsf{Ndof}$}}
\psfrag{73}{{\large $\mathfrak{S}_{14}$}}
\psfrag{5e}{{\large $\mathfrak{S}_{5}$}}
\psfrag{5c}{{\large $\mathfrak{S}_{5c}$}}
\psfrag{fd}{{$\mathfrak{F}$}}
\begin{minipage}[c]{0.4\textwidth}\centering
\psfrag{errores totales}{\hspace{0.1cm}\large{$\|\mathbf{e}\|_{\Omega}$ and $\|\mathfrak{e}\|_{\Omega}$}}
\includegraphics[trim={0 0 0 0},clip,width=5.05cm,height=3.2cm,scale=0.30]{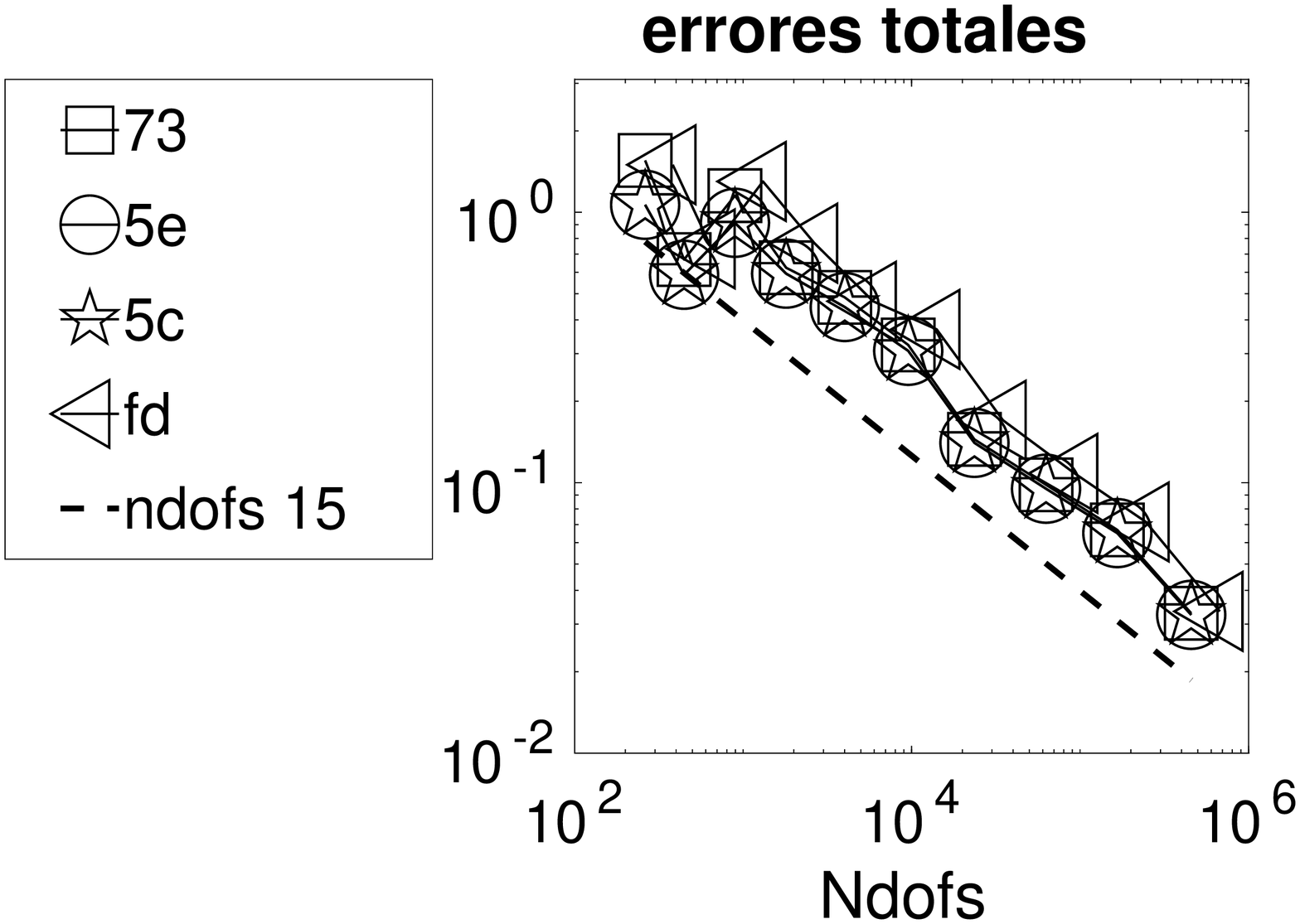}\\
\hspace{1.9cm}\tiny{(D.1)}
\end{minipage}
\begin{minipage}[c]{0.3\textwidth}\centering
\psfrag{errors ct}{\hspace{-1.1cm}\large{$\|\mathbf{e}_{\mathbf{u}}\|_{\mathbf{L}^{2}(\Omega)}$ and $\|\mathfrak{e}_{\mathbf{u}}\|_{\mathbf{L}^{2}(\Omega)}$}}
\includegraphics[trim={0 0 0 0},clip,width=3.15cm,height=3.2cm,scale=0.30]{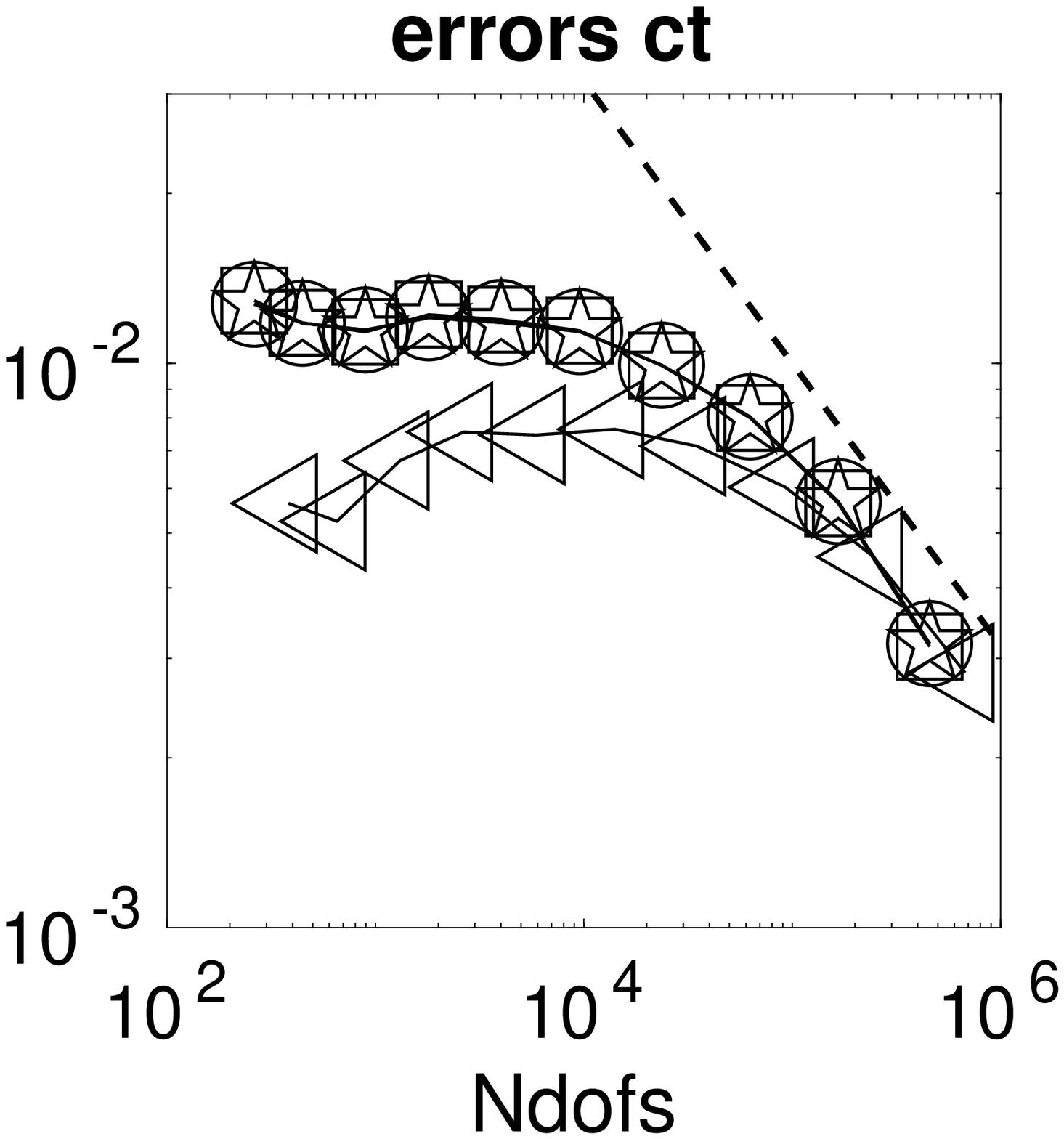}\\
\hspace{0.2cm}\tiny{(D.2)}
\end{minipage}
\\
\begin{minipage}[c]{0.248\textwidth}\centering
\psfrag{errors st}{\hspace{0.0cm}\large{$\|\nabla \mathbf{e}_{\mathbf{y}}\|_{\mathbf{L}^{2}(\Omega)}$}}
\includegraphics[trim={0 0 0 0},clip,width=3.15cm,height=3.2cm,scale=0.30]{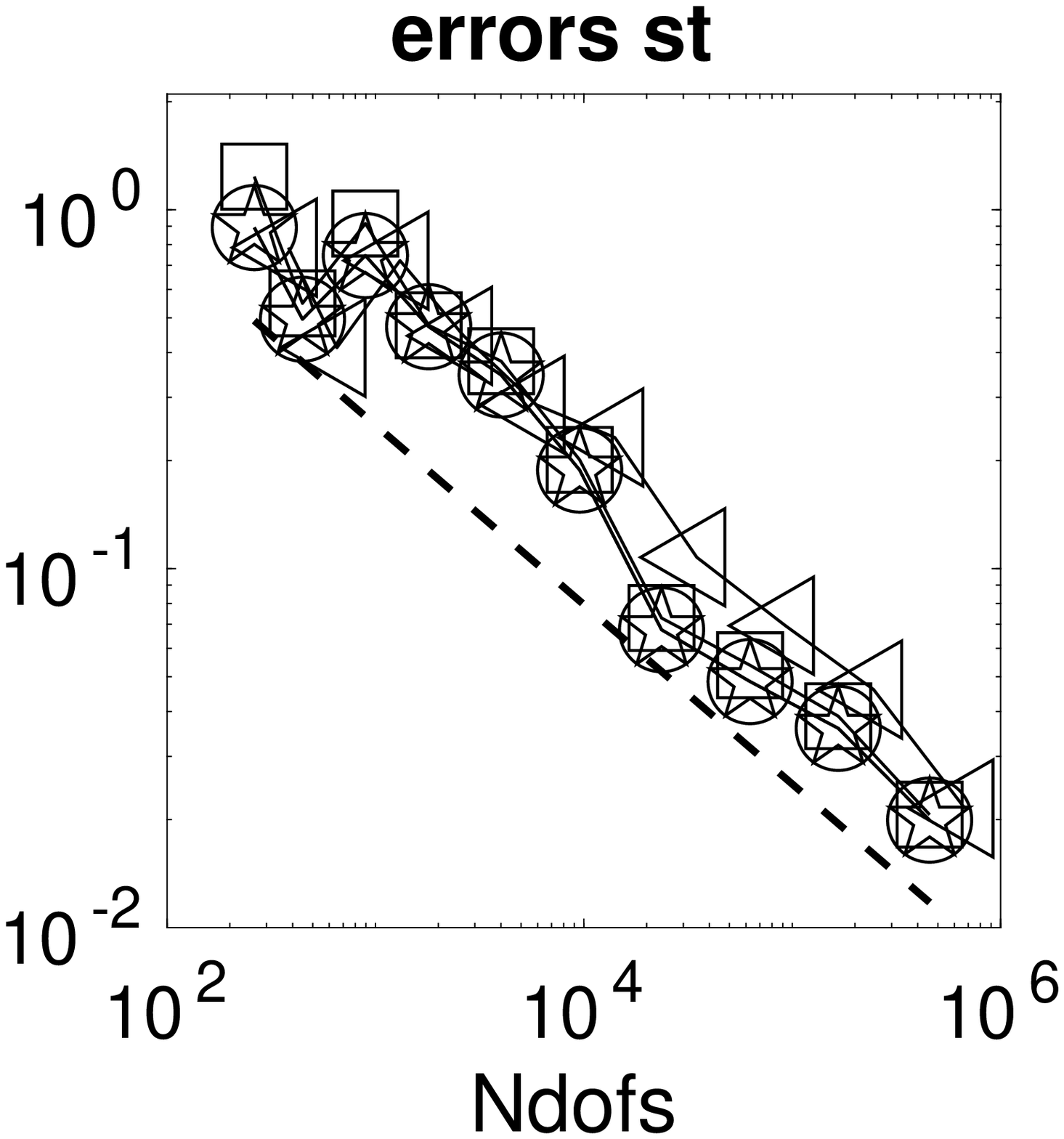}\\
\hspace{0.2cm}\tiny{(D.3)}
\end{minipage}
\begin{minipage}[c]{0.248\textwidth}\centering
\psfrag{errors st}{\hspace{0.4cm}\large{$\|e_{p}\|_{L^{2}(\Omega)}$}}
\includegraphics[trim={0 0 0 0},clip,width=3.15cm,height=3.2cm,scale=0.30]{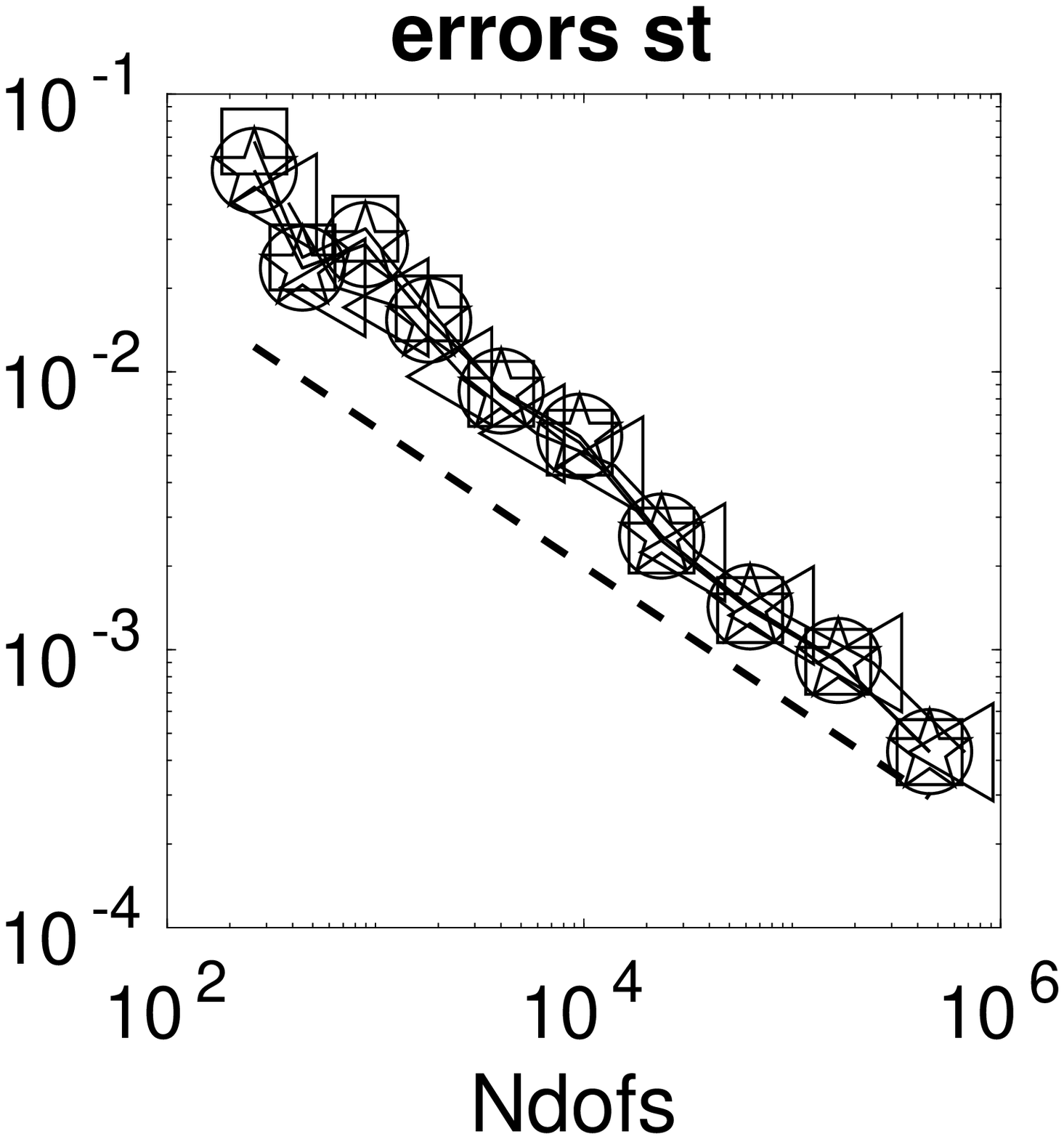}\\
\hspace{0.2cm}\tiny{(D.4)}
\end{minipage}
\begin{minipage}[c]{0.248\textwidth}\centering
\psfrag{errors ad}{\hspace{0.0cm}\large{$\|\nabla \mathbf{e}_{\mathbf{z}}\|_{\mathbf{L}^{2}(\Omega)}$}}
\includegraphics[trim={0 0 0 0},clip,width=3.15cm,height=3.2cm,scale=0.30]{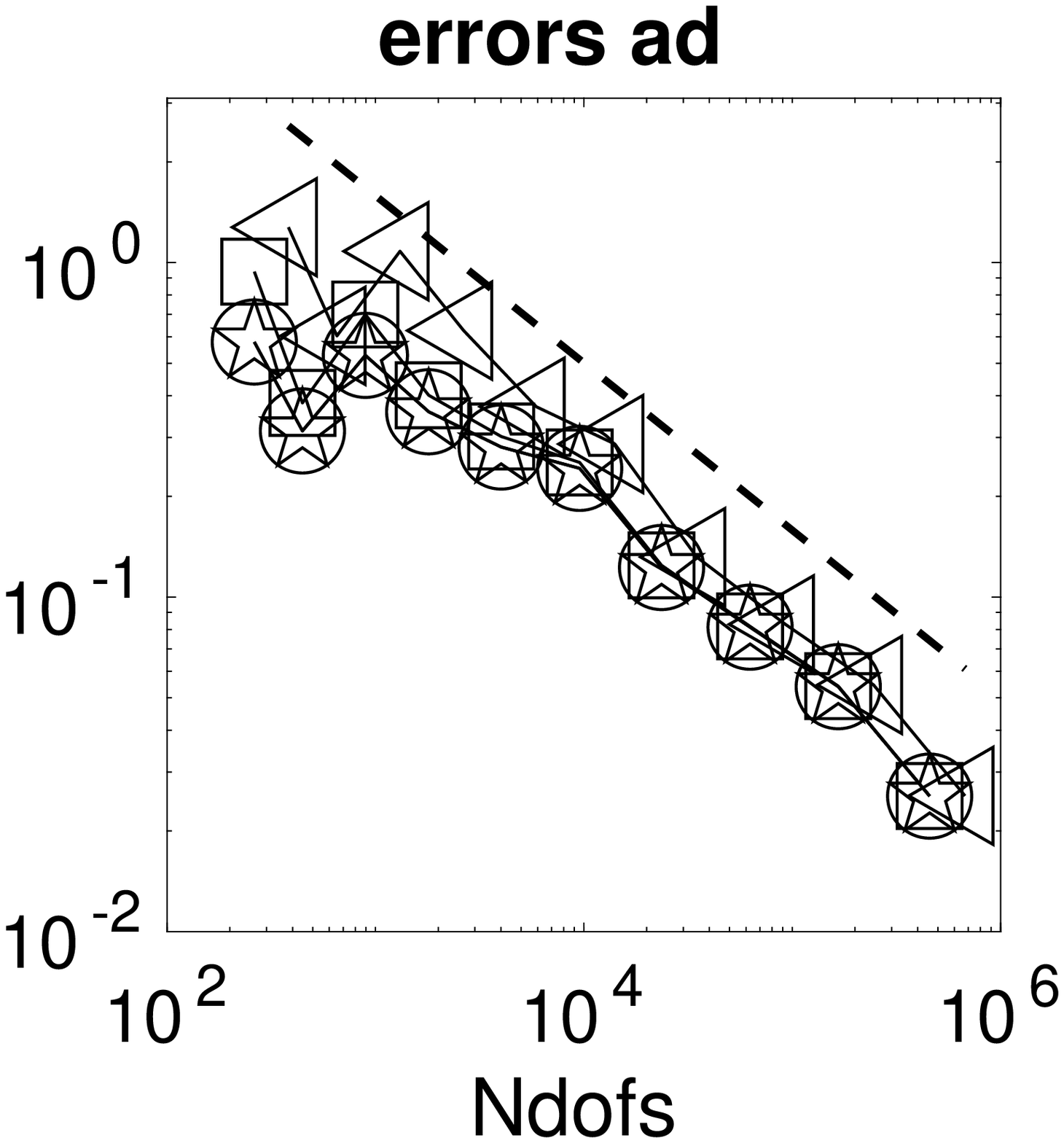}\\
\hspace{0.2cm}\tiny{(D.5)}
\end{minipage}
\begin{minipage}[c]{0.248\textwidth}\centering
\psfrag{errors ad}{\hspace{0.4cm}\large{$\|e_{r}\|_{L^{2}(\Omega)}$}}
\includegraphics[trim={0 0 0 0},clip,width=3.15cm,height=3.2cm,scale=0.30]{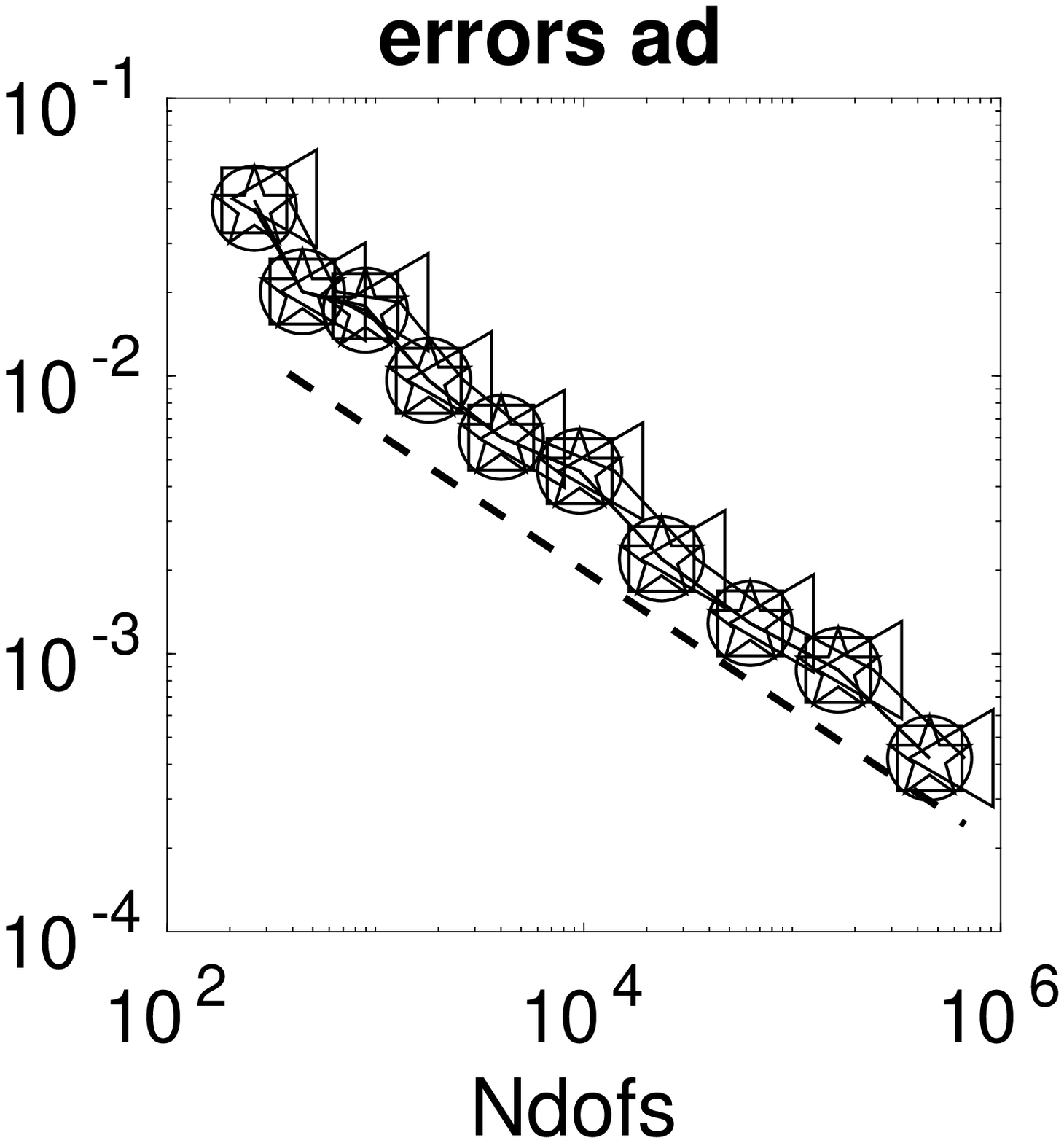}\\
\hspace{0.2cm}\tiny{(D.6)}
\end{minipage}
\caption{Example 1. Experimental rates of convergence, with uniform refinement, for the total errors $\|\mathbf{e}\|_{\Omega}$ and $\|\mathfrak{e}\|_{\Omega}$ (D.1) and each of their contributions (D.2)--(D.6) by considering the semi discrete scheme (with the implementations $\mathfrak{S}_{14}$, $\mathfrak{S}_{5}$, and $\mathfrak{S}_{5c}$) and the fully discrete scheme $\mathfrak{F}$.}
\label{fig:ex-2_unif}
\end{figure}


\begin{figure}[!ht]
\centering
\psfrag{ndof 23}{{\normalsize $\mathsf{Ndof}^{-2/3}$}}
\psfrag{ndof 33}{{\normalsize $\mathsf{Ndof}^{-1}$}}
\psfrag{Ndofs}{{\large $\mathsf{Ndof}$}}
\psfrag{I 73}{{\large $\mathfrak{S}_{14}$}}
\psfrag{I 5}{{\large $\mathfrak{S}_{5}$}}
\psfrag{I 5c}{{\large $\mathfrak{S}_{5c}$}}
\psfrag{I fd}{{$\mathfrak{F}$}}
\begin{minipage}[c]{0.4\textwidth}\centering
\psfrag{etas totales}{\hspace{-0.0cm}\large{$\mathcal{E}_{ocp}$ and $\mathfrak{E}_{ocp}$}}
\includegraphics[trim={0 0 0 0},clip,width=4.7cm,height=3.2cm,scale=0.30]{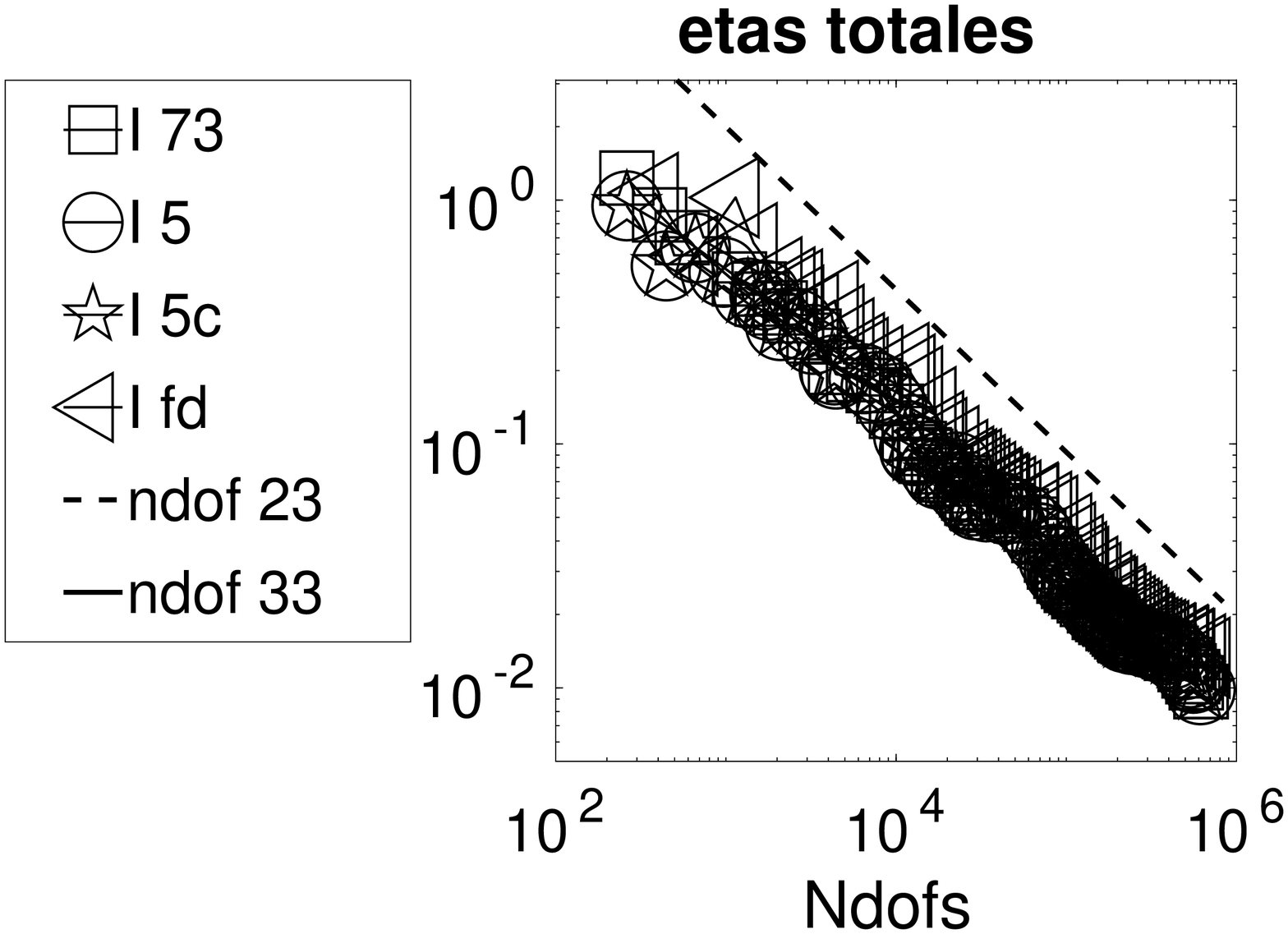}\\
\hspace{1.9cm}\tiny{(E.1)}
\end{minipage}
\begin{minipage}[c]{0.265\textwidth}\centering
\psfrag{indice ef}{\hspace{-0.0cm}\large{$\mathcal{I}$ and $\mathfrak{I}$}}
\includegraphics[trim={0 0 0 0},clip,width=3.4cm,height=3.2cm,,scale=0.30]{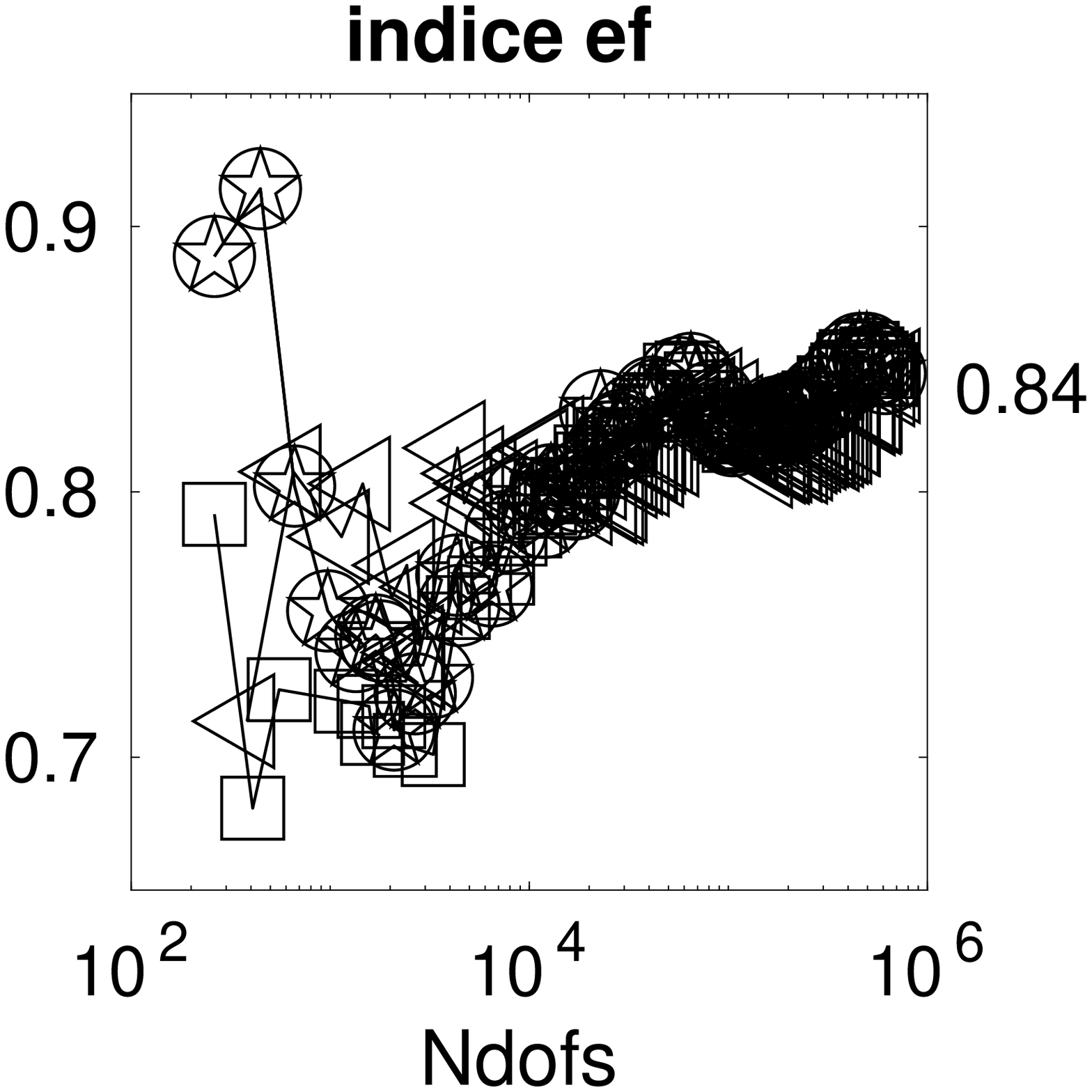}\\
\hspace{-0.30cm}\tiny{(E.2)}
\end{minipage}
\begin{minipage}[c]{0.25\textwidth}\centering
\psfrag{errors ct}{\hspace{-1.0cm}\large{$\|\mathbf{e}_{\mathbf{u}}\|_{L^{2}(\Omega)}$ and $\|\mathfrak{e}_{\mathbf{u}}\|_{L^{2}(\Omega)}$}}
\includegraphics[trim={0 0 0 0},clip,width=3.2cm,height=3.2cm,scale=0.30]{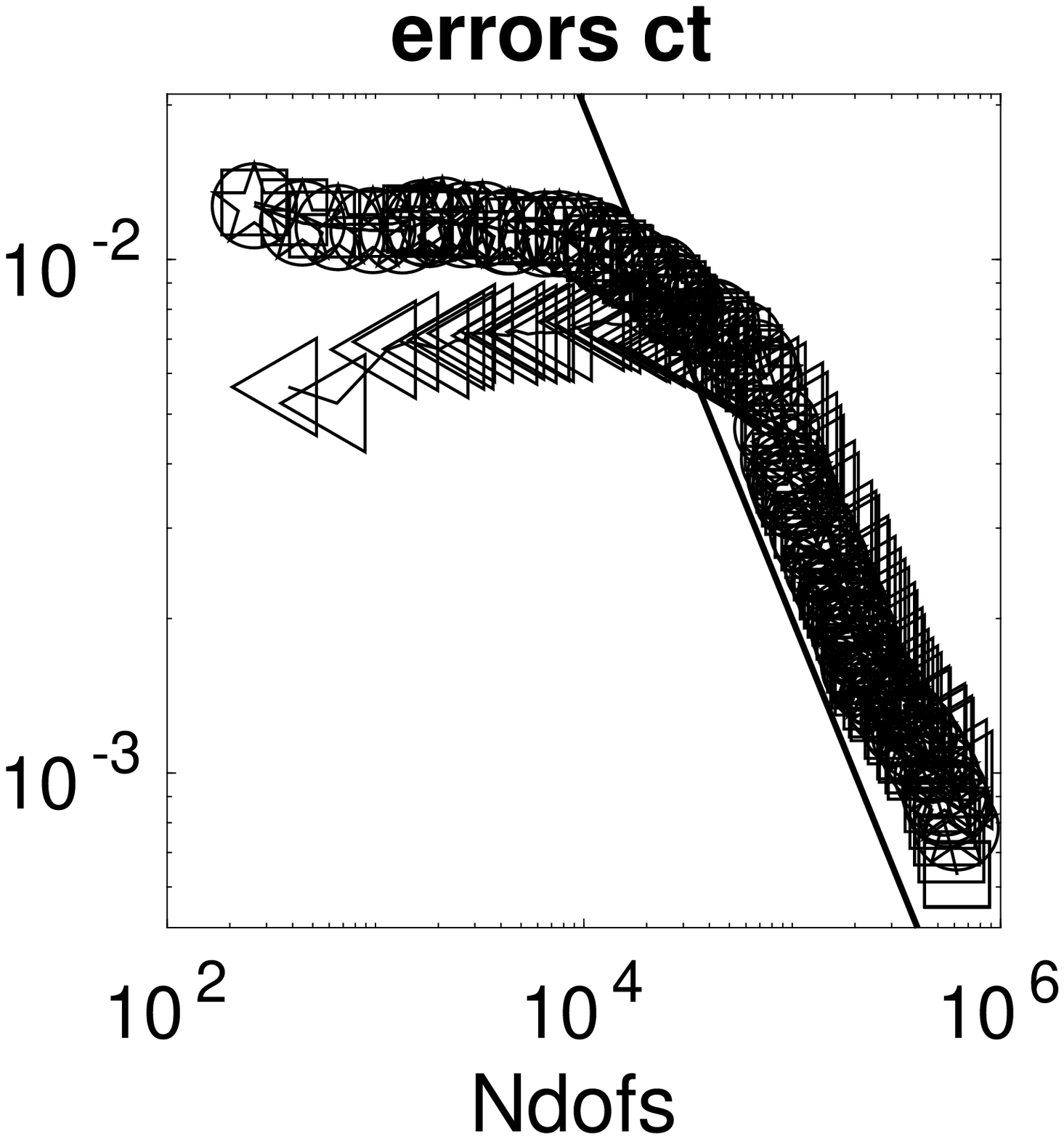}\\
\hspace{0.4cm}\tiny{(E.3)}
\end{minipage}
\\
\begin{minipage}[c]{0.248\textwidth}\centering
\psfrag{erros st}{\hspace{0.0cm}\large{$\|\nabla \mathbf{e}_{\mathbf{y}}\|_{\mathbf{L}^{2}(\Omega)}$}}
\includegraphics[trim={0 0 0 0},clip,width=3.2cm,height=3.2cm,scale=0.30]{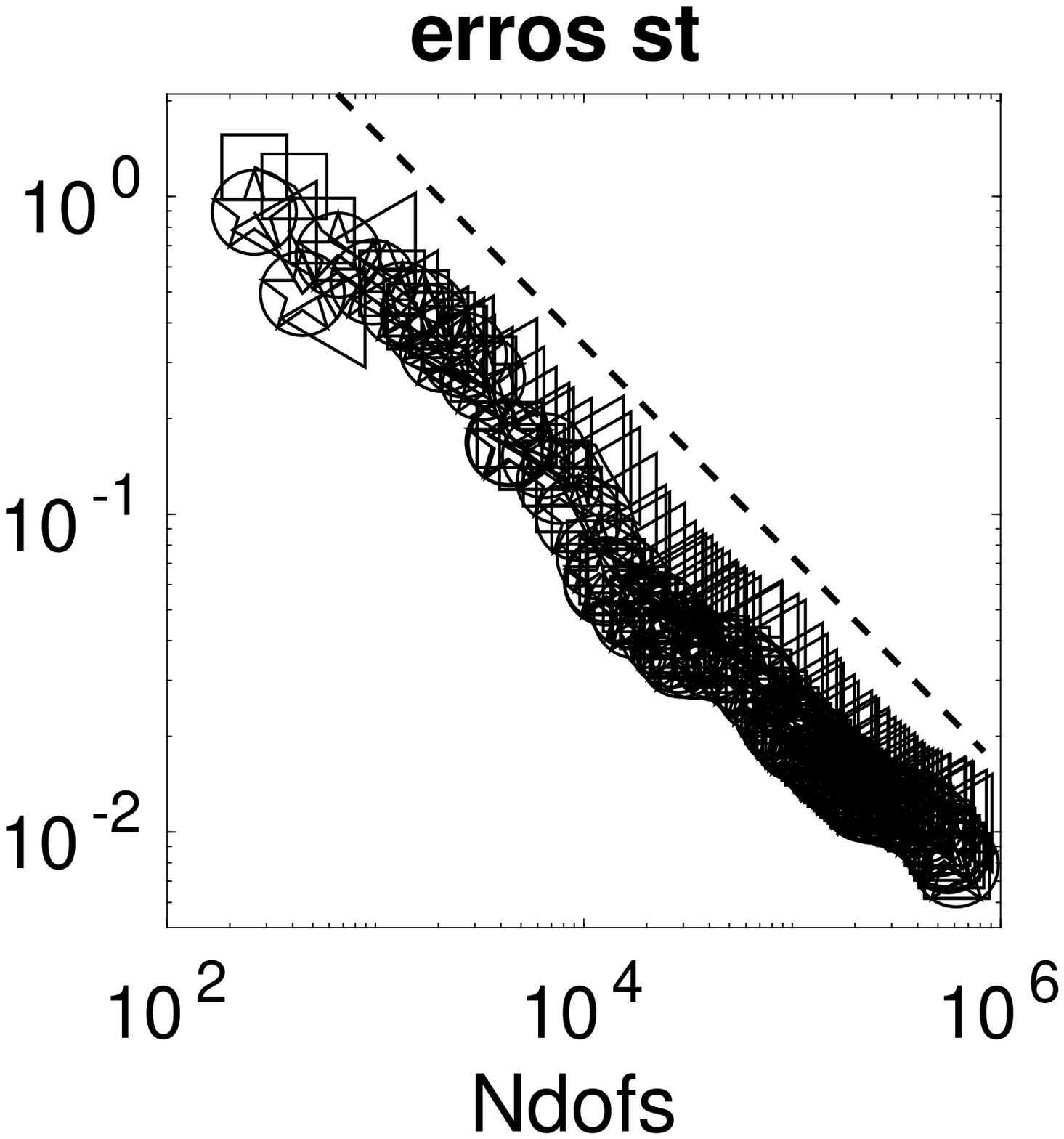}\\
\hspace{0.35cm}\tiny{(E.4)}
\end{minipage}
\begin{minipage}[c]{0.248\textwidth}\centering
\psfrag{errors st}{\hspace{0.2cm}\large{$\|e_{p}\|_{L^{2}(\Omega)}$}}
\includegraphics[trim={0 0 0 0},clip,width=3.2cm,height=3.2cm,scale=0.30]{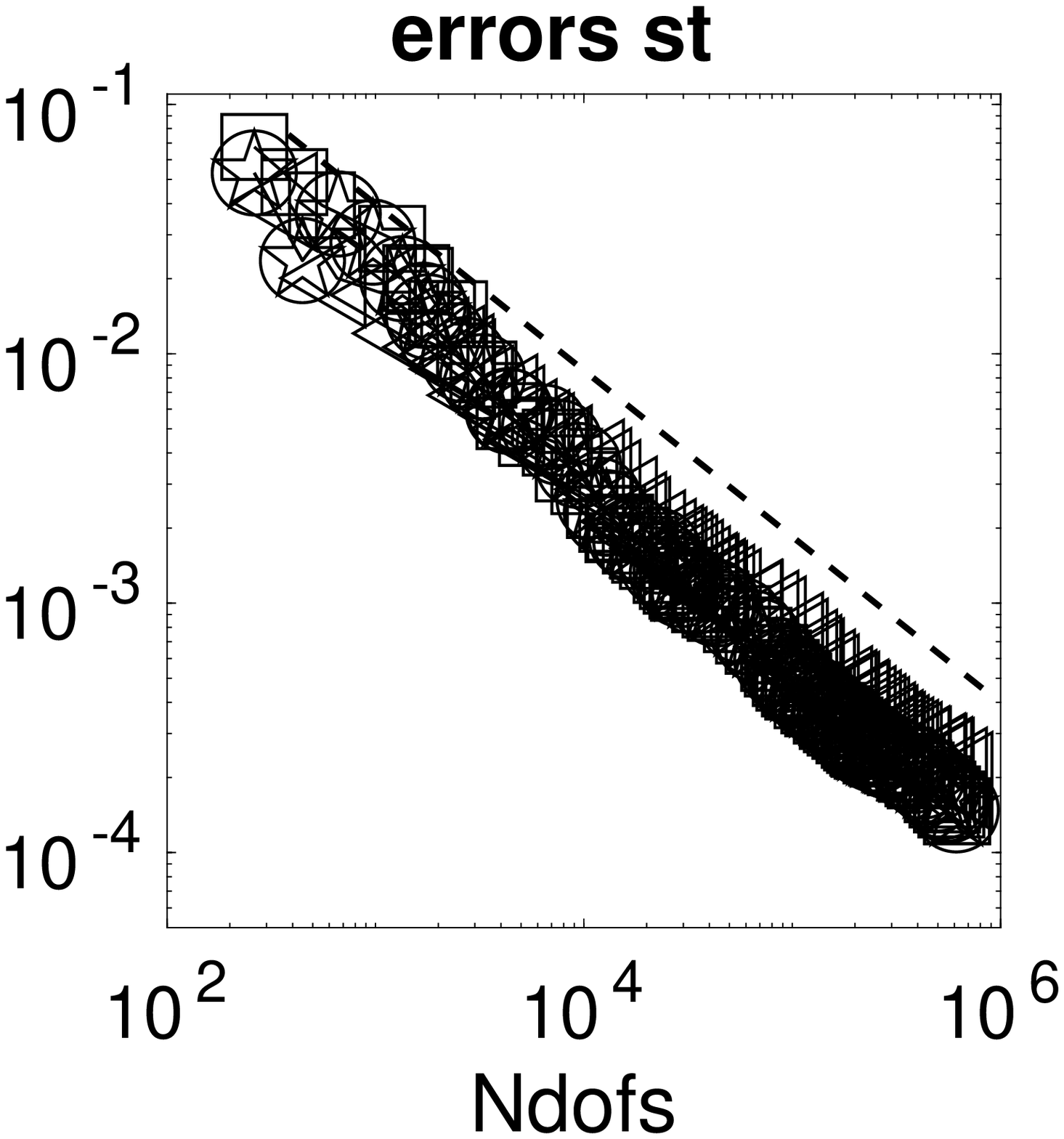}\\
\hspace{0.35cm}\tiny{(E.5)}
\end{minipage}
\begin{minipage}[c]{0.248\textwidth}\centering
\psfrag{erros ad}{\hspace{0.0cm}\large{$\|\nabla \mathbf{e}_{\mathbf{z}}\|_{\mathbf{L}^{2}(\Omega)}$}}
\includegraphics[trim={0 0 0 0},clip,width=3.2cm,height=3.2cm,scale=0.30]{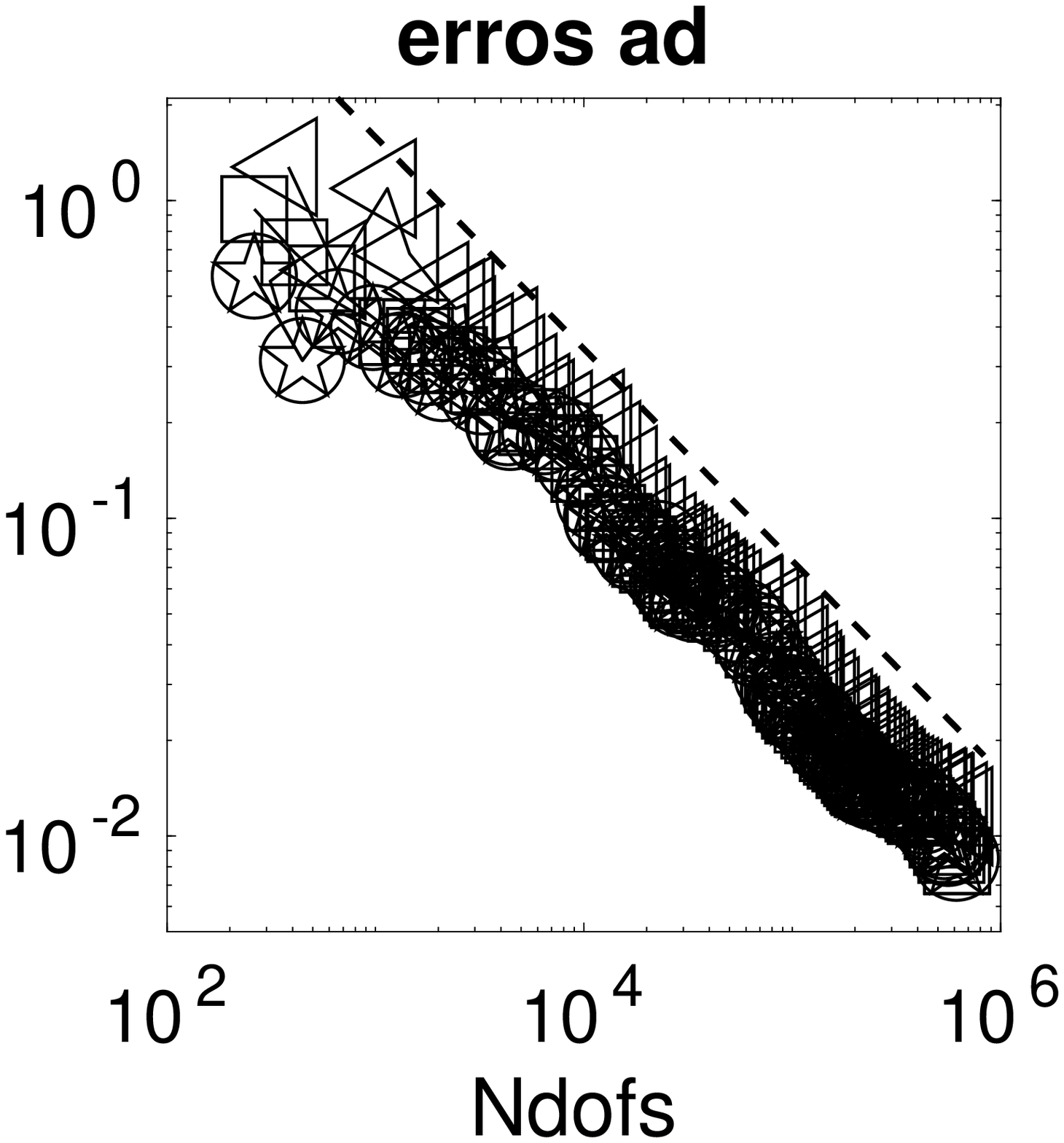}\\
\hspace{0.35cm}\tiny{(E.6)}
\end{minipage}
\begin{minipage}[c]{0.248\textwidth}\centering
\psfrag{errors ad}{\hspace{0.2cm}\large{$\|e_{r}\|_{L^{2}(\Omega)}$}}
\includegraphics[trim={0 0 0 0},clip,width=3.2cm,height=3.2cm,scale=0.30]{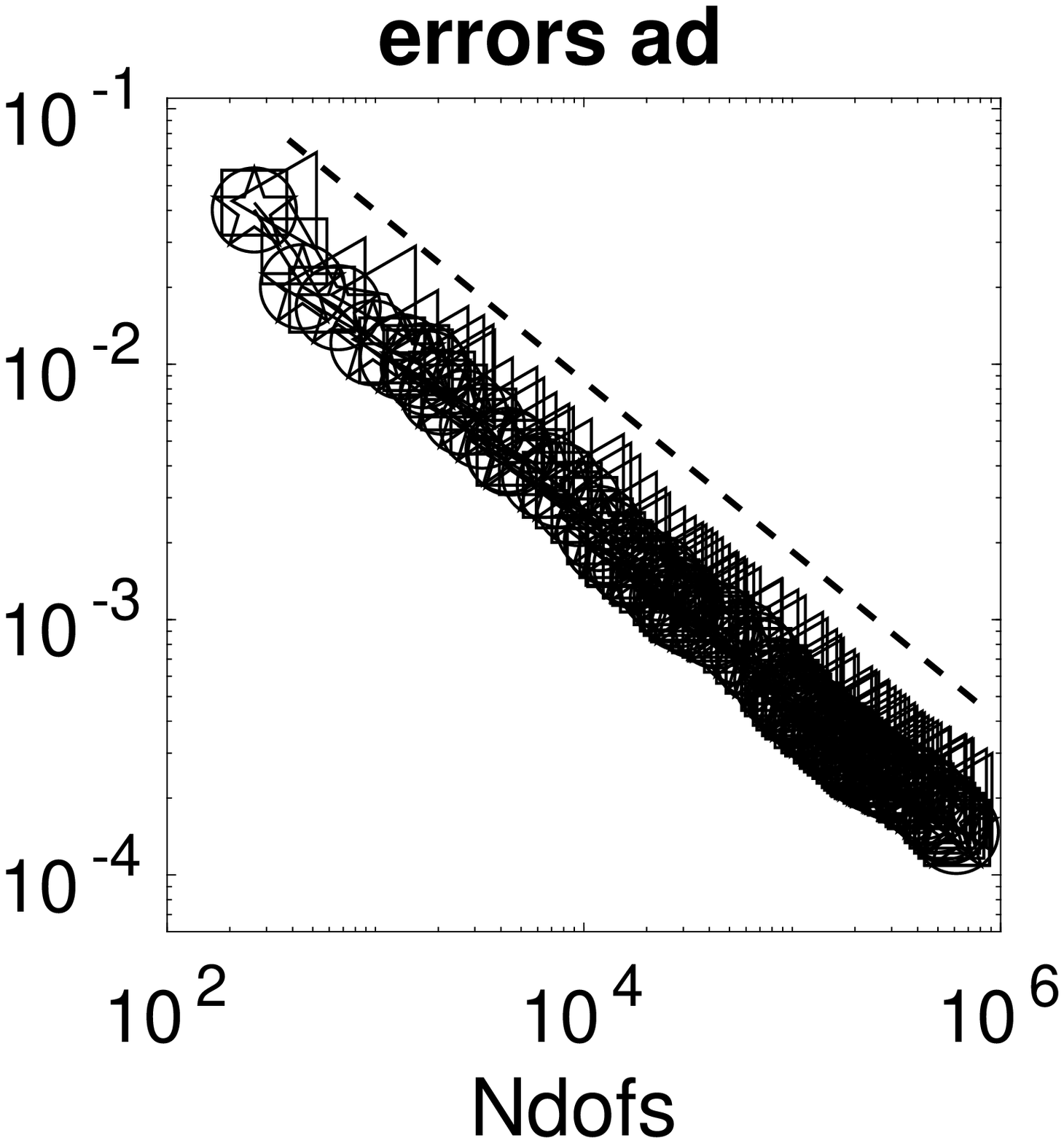}\\
\hspace{0.35cm}\tiny{(E.7)}
\end{minipage}
\caption{Example 2. Experimental rates of convergence, with adaptive refinement, for the total error estimators $\mathcal{E}_{ocp}$ and $\mathfrak{E}_{ocp}$ (E.1), effectivity indices $\mathcal{I}$ and $\mathfrak{I}$ (E.2), and experimental rates of convergence of each contribution of $\|\mathbf{e}\|_{\Omega}$ and $\|\mathfrak{e}\|_{\Omega}$ (E.3)--(E.7) 
by considering the semi discrete scheme (with the implementations $\mathfrak{S}_{14}$, $\mathfrak{S}_{5}$, and $\mathfrak{S}_{5c}$) and the fully discrete scheme $\mathfrak{F}$.}
\label{fig:ex-2_adap}
\end{figure}

\bibliographystyle{siam}
\footnotesize
\bibliography{biblio}

\end{document}